\documentclass[11pt, leqno]{amsart}
\usepackage{mathrsfs}
\usepackage{graphicx}
\usepackage{amsfonts,delarray,amssymb,amsmath,amsthm,a4,a4wide}
\usepackage{latexsym}
\usepackage{epsfig}
\usepackage{color}
\usepackage[T1]{fontenc}

\newtheorem{thm}{Theorem}[section]
\newtheorem{cor}[thm]{Corollary}

\newtheorem{lem}[thm]{Lemma}
\newtheorem{exa}[thm]{Example}
\newtheorem{prop}[thm]{Proposition}

\newtheorem{question}[thm]{Question}

\theoremstyle{remark}
\newtheorem{rem}[thm]{Remark}
\theoremstyle{definition}
\newtheorem{defn}[thm]{Definition}

\numberwithin{equation}{section}
\newcommand{\R}{\mathbb R}
\newcommand{\K}{\mathbb K}
\newcommand{\N}{\mathbb N}
\newcommand{\E}{\mathbb E}
\renewcommand{\L}{\mathcal L}
\newcommand{\e}{\varepsilon}

\newcommand{\p}{\partial}
\newcommand{\dist}{\mbox{dist}\,}
\newcommand{\diam}{\mbox{diam}\,}

\newcommand{\comment}[1]{}
\def\h{\hspace*{.24in}}

\makeatletter
\@namedef{subjclassname@2020}{%
  \textup{2020} Mathematics Subject Classification}
\makeatother

\makeatletter
\def\l@subsection{\@tocline{2}{0pt}{2.5pc}{5pc}{}}

\begin{document} 

\title
[Degenerate Monge--Amp\`ere Equations via Mixed Measures]
{A Variational
Approach to 
Degenerate 
Monge--Amp\`ere Equations with Mixed Measures and Monotonicity 
}
\author{Nam Q. L\^e}
\address{Department of Mathematics, Indiana University, 
Bloomington, IN 47405, USA}
\email{nqle@iu.edu}
\thanks{The author was supported in part by the National Science Foundation under grant DMS-2452320.}

\subjclass[2020]{ 35J96, 35P30, 49R05, 35J70, 35J75}
\keywords{Degenerate Monge--Amp\`ere equation, Monge--Amp\`ere eigenvalue problem, Mixed Monge--Amp\`ere measure, Convex envelope, Maximum principle, Comparison principle}

\begin{abstract}
We study the solvability and uniqueness for several degenerate Monge--Amp\`ere equations including the Monge--Amp\`ere eigenvalue problem in real Euclidean spaces that involve singular Borel measures. Our approach 
 systematically analyzes the Monge--Amp\`ere energy from the variational point of view and appropriately exploits monotonicity arguments.
 Our main tools consist of the mixed Monge--Amp\`ere measure, Aleksandrov--Blocki--Jerison-type maximum principles, integration by parts, convex envelope, and comparison principles for subcritical equations. 
For the Monge--Amp\`ere eigenvalue problem, we contrast the analysis within and without the energy class;
even if it might not have solutions in the energy class, we show that the 
infimum of the Rayleigh quotient can be approximated from above by Monge--Amp\`ere eigenvalues of the truncated measures, and by
Rayleigh quotients of an inverse iterative scheme. We give examples showing that for very singular Borel measures, the Monge--Amp\`ere eigenvalue problem has only solutions outside the energy class together with symmetry breaking and nonuniqueness.

\end{abstract}
\maketitle

\tableofcontents
 
 \section{Introduction and statement of the main results}
 
 In this paper, we study, via variational method and monotonicity, 
 degenerate Monge--Amp\`ere equations including the Monge--Amp\`ere eigenvalue problem that involve singular Borel measures on general bounded convex domains in Euclidean spaces.

 \medskip

 Given a Borel measure $\nu$ on a bounded convex domain $\Omega\subset\R^n$ with positive mass and $p\in [0,\infty)$,  we are interested in the Monge--Amp\`ere equation of the type
 \begin{equation}
 \label{eqp1}
   \det D^2 u = \lambda|u|^p\nu\quad\text{in} ~  \Omega, \quad
u = 0\quad\text{on}~ \p\Omega,
\end{equation}
 where $u\in C(\overline{\Omega})$ is an unknown convex function and $\lambda>0$ is an unknown constant. Due to homogeneity, we can always take $\lambda=1$ when $p\neq n$. 
\subsection{Case of  Lebesgue measure}
 When $\nu$ is the $n$-dimensional Lebesgue measure $d\L^n$, \eqref{eqp1} becomes
  \begin{equation}
 \label{eqp2}
 \det D^2 u=\lambda |u|^p\quad \text{in } \Omega\subset\R^n,\quad  u=0\quad \text{on } \p\Omega,\end{equation}
 and this equation has been extensively studied.
 When $p=n$, \eqref{eqp2} is the Monge--Amp\`ere eigenvalue problem for the domain $\Omega$ which was first studied by Lions \cite{Ls}. When $\Omega$ is a 
  smooth, uniformly convex domain in $\R^n$, Lions showed that there exist a unique positive constant $\lambda=\lambda(\Omega)$ and a unique (up to positive multiplicative constants) nonzero 
convex function $u\in C^{1,1}(\overline{\Omega})\cap C^{\infty}(\Omega)$ solving the eigenvalue problem for the Monge-Amp\`ere operator 
\begin{equation}\label{EPLi}
   \det D^2 u~ = \lambda|u|^n~\quad\text{in} ~  \Omega, \quad
u = 0~\quad\text{on}~ \p\Omega.
  \end{equation}
  The constant $\lambda(\Omega)$ is called 
the Monge--Amp\`ere eigenvalue of $\Omega$, and the nonzero convex functions $u$ solving (\ref{EPLi}) are called the Monge--Amp\`ere eigenfunctions of $\Omega$. 

\medskip
A variational characterization of $\lambda(\Omega)$ was found by Tso \cite{Tso} who discovered, using gradient flows (parabolic Monge--Amp\`ere equations), that for sufficiently smooth,  uniformly convex domains $\Omega$, the following formula involving the Rayleigh quotient holds:
\begin{multline}
 \label{lamTs}
 \lambda(\Omega)=\inf\bigg\{ \frac{\int_{\Omega} (-u)\det D^2 u~dx}{\int_{\Omega}(-u)^{n+1}~dx}: \text{convex}\,\, u\in C^{0,1}(\overline{\Omega})\cap C^2(\Omega)\setminus\{0\},\,\, u=0~\text{on}~\p\Omega\bigg\}.
\end{multline}
In fact, with $\lambda=\lambda(\Omega)$ defined by \eqref{lamTs}, using  gradient flows and variational arguments, Tso showed that \eqref{eqp2} has a nonzero convex solution $u_p\in C^{0,1}(\overline{\Omega})\cap C^\infty(\Omega)$ when $0<p\neq n<\infty$. When $0<p<n$, Tso established the uniqueness of $u_p$ and that $\|u_p\|_{L^{\infty}(\Omega)}$ is uniformly bounded from below and above by positive constants, so when $p\nearrow n$, $u_p$ tends to a Monge--Amp\`ere eigenfunction of $\Omega$. 

\medskip
For general bounded convex domains $\Omega\subset\R^n$, building on Tso's work and using approximation arguments together with regularity techniques based on the works of Caffarelli \cite{C1, C2}, the author (see \cite[Theorem 1.1 and Proposition 2.8]{L} and  \cite[Theorem 1.1 and Corollary 1.2]{L_AFST}) was able to extend the above-mentioned results of Tso 
 for all $p>0$, except that when $0<p\leq n-2$, the unique nonzero convex solution to \eqref{eqp2} is smooth in the interior of $\Omega$ but its gradient blows up near any flat part of the boundary; see \cite[Theorem 1.1]{LCPAA} and \cite[Theorem 1.5]{L_AFST}.
 Moreover,  
 when $p=n$, the Monge--Amp\`ere eigenvalue problem  \eqref{EPLi} has a unique Monge--Amp\`ere eigenvalue, renamed $\lambda[\Omega]$ to reflect the fact that the boundary $\p\Omega$ might have flat parts or corners, and a unique Monge-Amp\`ere eigenfunction $u\in C^{0, 1}(\overline{\Omega})\cap C^\infty(\Omega)$ (up to positive multiplicative constants). Here, $\lambda[\Omega]$ is characterized by either of the following variational formulas:
 \begin{equation}
\label{lam_def}
\begin{split}
\lambda[\Omega] &=\inf\bigg\{ \frac{\int_{\Omega} |u| \, d\mu_u }{\int_{\Omega}|u|^{n+1}\, dx}: \text{convex}\,\, u\in C(\overline{\Omega})\setminus\{0\},\,~ u=0~\text{on}~\p\Omega\bigg\}\\
  &=\inf\bigg\{ \frac{\int_{\Omega} (-u)\det D^2 u~dx}{\int_{\Omega}(-u)^{n+1}~dx}: \text{convex}\,\, u\in C^{0,1}(\overline{\Omega})\cap C^2(\Omega)\setminus\{0\},\,
 ~u=0~\text{on}~\p\Omega\bigg\},
  \end{split}
 \end{equation}
 where $\mu_u$ denotes the Monge--Amp\`ere measure of the convex function $u$ (see \eqref{MAmu}).

\medskip
The uniqueness is a widely open problem in the case $p>n$.
Recent progresses have been made by Huang \cite{H} and Cheng--Huang--Xu \cite{CHX}. When $\Omega$ is smooth and uniformly convex, Huang proved the uniqueness
of least energy solution to \eqref{eqp2} for any $p\in (2, \infty)$ when $n=2$, and
the uniqueness of nonzero convex solution to \eqref{eqp2} for any $2\leq n<p<\alpha_0(n)$ where $\alpha_0(n)$ is close to $n$. In two dimensions, Cheng, Huang, and Xu
obtained uniqueness when the bounded convex domain  has at least two different symmetric axes (such as a triangle).

\subsection{Case of general Borel measure}
Recently, Lu and Zeriahi \cite{LZ} studied \eqref{eqp1} with a general Borel measure $\nu$ when $p=0$ (the Dirichlet problem) and $p=n$, using techniques of complex Monge--Amp\`ere equations including variational method, and then transferring the results to the real  Monge--Amp\`ere equations. This passage is rooted in a deep connection between the complex and real Monge--Amp\`ere operators via the logarithmic map; see Berman--Berndtsson \cite[Section 2.2]{BBe} and Coman--Guedj--Sahin--Zeriahi \cite[Lemma 2.2]{CGSZ}.
Lu and Zeriahi obtained very interesting existence and uniqueness results when the measure $\nu$ satisfies
 \begin{equation}\label{nuLZ}\nu(\Omega)>0,\quad  \int_\Omega (-v)\, d\nu<\infty\quad\text{for some convex function } v\in C(\Omega),\quad v<0\quad\text{in }\Omega.\end{equation}
Since the convex function $v$ in \eqref{nuLZ} must have the growth $|v(x)|\geq \frac{\|v\|_{L^{\infty}(\Omega)}}{\diam(\Omega)}\dist(x,\p\Omega)$,
the measure $\nu$ in \eqref{nuLZ} must satisfy the finiteness condition $\int_\Omega \dist(\cdot,\p\Omega) d\nu<\infty$.
This is the class of measures that we will focus on.

\medskip
 When $p=n$, it was only assumed in \cite{LZ} for the Monge--Amp\`ere eigenvalue problem that $\nu=\mu_w$ for some convex function $w\in C(\overline{\Omega})$.  In view of Blocki's inequality (see Theorem \ref{BlemR} (ii)), this measure is natural for the variational method which requires the finiteness of $\int_\Omega |u|^{n+1}\, d\nu$ when $u$ is in the finite energy class $\E(\Omega)$ (see \eqref{EEdef}).
 In this setting, Lu and Zeriahi \cite[Theorem 1.4]{LZ} solved the Monge--Amp\`ere eigenvalue problem including existence and uniqueness in the energy class, established a variational characterization of the Monge--Amp\`ere eigenvalue via the infimum $\lambda[\Omega,\nu]$ of the Rayleigh quotient, and proved the convergence of an inverse iterative scheme (first introduced by Abedin--Kitagawa \cite{AK}) to the Monge--Amp\`ere eigenvalue problem.  When $\Omega$ contains the origin and $\nu=|x|^{s}\, d\L^n$ where $s>-n$ (so $\nu$ can only have interior integrable singularity), He--Huang \cite[Theorem 1.1]{HH} obtained similar results to those in \cite{LZ} using approximations and elliptic and parabolic techniques as in Tso \cite{Tso} and Wang \cite{W1}. When $\nu$ is the Lebesgue measure, earlier works by Badiane--Zeriahi \cite{BZ1, BZ2} solved the complex Monge--Amp\`ere eigenvalue problem using plurisubharmonic envelopes.

  \medskip
 Our recollections above suggest to view the infimum of the Rayleigh quotient $\lambda[\Omega,\nu]$ as either an independent object or in its relation to the Monge--Amp\`ere eigenvalue problem. It is interesting to see what can be said about $\lambda[\Omega,\nu]$ and the convergence property of an inverse iterative scheme when the Monge--Amp\`ere eigenvalue problem does not have solutions in the energy class. Moreover, in view of the above-mentioned results of Tso and the author for \eqref{eqp1} when $\nu=d\L^n$ and $p>0$, one would like to know if \eqref{eqp1} has nonzero convex solutions when $\nu$ is more singular.

 \subsection{Real Monge--Amp\`ere variational method}
 Here, we provide a unified real Mong--Amp\`ere treatment of \eqref{eqp1} via variational method with mixed Monge--Amp\`ere measures and monotonicity. Some of our arguments are inspired by the work of Lu and Zeriahi \cite{LZ} which relies on complex Monge--Amp\`ere technologies developed in the last two decades. Others push the ideas, including integration by parts, in our previous works on the real Monge--Amp\`ere eigenvalue problem \cite{L, L_RMI, L_scheme} from Lebesgue measure to more singular Borel measures.
 Our intention in this paper is to give a self-contained treatment of the real Monge--Amp\`ere variational method that could be useful for later references and for other researchers.
 
 \medskip
 Convex functions $u\in  C(\overline{\Omega})$ are the most natural candidates for solutions to \eqref{eqp1}.  When $\nu$ is regular enough, solutions to  \eqref{eqp1} belong to the energy class $\E(\Omega)$. This is not the case when $\nu$ is quite singular. See Example \ref{exaH1} where we also have nonuniqueness and symmetry breaking; that is, very nonuniqueness phenomena can happen outside the energy class. This is reminiscent of the convex integration construction in other contexts; see, for example, the survey by Buckmaster--Vicol \cite{BV}, Lewicka \cite{Lw} and the references therein.
 
 \medskip
 As will be seen later, all potential difficulties in studying \eqref{eqp1} come from the singularity of $\nu$ near the boundary.  We also encounter the general lack of compactness for solutions corresponding to measures in the class under consideration;
 see Remark \ref{lackrem}.
 As in \cite{L}, to resolve these issues, approximations will be repeatedly used in our analysis.  Most importantly, we truncate $\nu$ to measures of the form $\nu_m= \chi_{\{x\in\Omega: \dist(x, \p\Omega)>1/m\}}\nu$, and study corresponding problems with $\nu_m$ where uniqueness and energy class member can be guaranteed. 
  In this process, we obtain many monotone quantities which facilitate the convergence analysis.

  \medskip
 We allow the measure $\nu$ to be singular near the boundary $\p\Omega$. It turns out that, as long as $\dist(\cdot,\p\Omega)\nu$ is integrable, the problem \eqref{eqp1} is always solvable for all $p\in [0, \infty)$. When $p>n$, $\nu$ can be more singular. 
 In view of the Aleksandrov--Jerison maximum principle (Theorem \ref{AJ_thm}), the finiteness of 
 $\dist(\cdot,\p\Omega)$ is natural for the existence of solutions to the Dirichlet problem. The variational method suggests $\dist^{\frac{p+1}{n+1}}(\cdot,\p\Omega)$ should be integrable with respect to $\nu$. However, we take advantage of subcriticality when $0\leq p<n$ to allow $\nu$ to be more singular in the sense that only the smaller weight $\dist(\cdot,\p\Omega)$ is required to be integrable.

 \medskip
 Before stating our main results, we recall some standard notions.
 Let $\Omega$ be a 
 domain in $\R^n$ $(n\geq 1)$.  When $\Omega$ is not necessary convex, we say that a function \(u:\Omega \to \R\) is convex if it can be extended to a convex function, possibly taking value $+\infty$, on all of $\R^n$.
For a convex function \(u:\Omega \to \R\), we define
 the subdifferential of $u$  at $x\in\Omega$ by
 \[
\partial u (x):=\big\{p\in \R^{n}\,:\, u(y)\ge u(x)+p\cdot (y-x)\quad \text{for all } y \in \Omega\big\},
\]
and we define the Monge--Amp\`ere measure $\mu_u$ of $u$  by
\begin{equation}
\label{MAmu}
\mu_u(S) = |\p u(S)|\quad\quad\text{where } \p u(S) = \bigcup_{x\in S} \p u(x)\quad \text{for each Borel set } S\subset\Omega.\end{equation}
The Monge--Amp\`ere measure $\mu_u$ is Borel regular on $\Omega$ and is finite on compact subsets of $\Omega$, so it is in fact a Radon measure.
If $u\in C^{2}(\Omega)$, then 
\[
\mu_u=\det D^{2} u\,d\mathcal{L}^n \quad \text{in }\Omega.\] Motivated by this, 
given a Borel measure \(\nu\) on \(\Omega\), we call a convex
function \(u:\Omega \to \R\) an \emph{Aleksandrov solution} to the Monge--Amp\`ere equation
$
\det D^{2} u =\nu
$
if $\nu=\mu_u$ as Borel measures. 
We will say for simplicity  that \(u\) solves $\det D^2 u=\nu$.

\medskip
On a general bounded convex domain $\Omega\subset\R^n$ (not necessarily smooth nor uniformly convex), we define the finite energy class of convex functions with zero boundary values by
\begin{equation}
\label{EEdef}
\E(\Omega) = \{  w \in C(\overline{\Omega}):  
 ~w~\text{is convex in } \Omega,~ w=0~\text{on}~\p\Omega\quad\text{with } \int_\Omega |w|\, d\mu_w<\infty \}.
\end{equation}
For $u\in \E(\Omega)$, we denote its Monge--Amp\`ere energy by
\begin{equation}
\label{Edef}
E(u)= E(u;\Omega):=\int_\Omega (-u)\, d\mu_{u}\equiv \int_\Omega |u|\, d\mu_u.\end{equation}
For a Borel measure $\nu$ with $\nu(\Omega)>0$ on a bounded convex domain $\Omega\subset\R^n$, we 
denote the Rayleigh quotient $R_\nu$ associated with $\nu$ and its infimum in the energy class by
\begin{equation}\label{Radef} R_\nu (u):= \frac{E(u)}{\int_\Omega |u|^{n+1}\, d\nu}\,\quad\text{for }u\in \E(\Omega),\quad \lambda[\Omega, \nu]:=\inf\bigg\{\frac{E(u)}{\int_\Omega |u|^{n+1}\, d\nu}: u\in \E(\Omega)\bigg\}. \end{equation}
\subsection{Main results}
We now state our main results.

Our first theorem is concerned with the solvability, uniqueness, and variational characterization of solutions to degenerate Monge--Amp\`ere equations.
\begin{thm}[Solvability, uniqueness, and variational characterization of degenerate Monge--Amp\`ere equations]
\label{Dist0pn}
Let $\Omega\subset\R^n$ $(n\geq 1)$ be a bounded convex domain, and $p\in (-1, \infty)$.
Let $\nu$ be a locally finite Borel measure on $\Omega$ satisfying $\nu(\Omega)>0$ and 
\begin{equation}
\label{nuthm1}
 \left\{
 \begin{alignedat}{1}
\int_\Omega \dist(\cdot,\p\Omega)d\nu\leq C<\infty& \quad\text{if } \quad 0\leq p\leq n,\\
\int_\Omega \dist^{\frac{p+1}{n+1}}(\cdot,\p\Omega)d\nu\leq C<\infty &\quad \text{if } \quad p\in (-1, 0)\cup (n,\infty).
 \end{alignedat}
 \right.
\end{equation}
\begin{enumerate}
\item If $p\in (-1, 0) \cup (n,\infty)$,
then, there exists a nonzero convex solution $u\in C(\overline{\Omega})$ 
to 
\begin{equation*}
   \mu_u=|u|^p\nu \quad\text{in} ~\Omega, \quad
u =0\quad\text{on}~\p \Omega,
\end{equation*}
which is multiple of a minimizer of the functional 
$\displaystyle \frac{\int_\Omega (-v) d\mu_v}{\big(\int_\Omega |v|^{p+1}\, d\nu\big)^{\frac{n+1}{p+1}}}$
over $\E(\Omega)$.
\item If $p=0$ and $\varphi\in C(\overline{\Omega})$ is a convex function, then, the Dirichlet problem
\begin{equation*}
   \det D^{2} u=\nu \quad\text{in} ~\Omega, \quad
u =\varphi\quad\text{on}~\p \Omega
\end{equation*}
has a unique convex Aleksandrov solution 
$u\in C(\overline{\Omega})$. Moreover, 
\begin{equation}
\label{Dirineq0}
\|(\varphi-u)^+\|^n_{L^{\infty}(\Omega)}  \leq C(n, \diam(\Omega))\int_\Omega \dist(\cdot,\p\Omega) d\nu.\end{equation}
\item If $0<p<n$,
then, there exists a nonzero convex Aleksandrov solution 
$u\in C(\overline{\Omega})$ to
\begin{equation*}
   \det D^{2} u=|u|^p\nu \quad\text{in} ~\Omega, \quad
u =0\quad\text{on}~\p \Omega.
\end{equation*} 

\item If $p=n$ and the following general vanishing mass condition holds instead of \eqref{nuthm1}
 \begin{equation} 
 \label{vamass}
 \small
 \int_{\{x\in\Omega: \dist(x, \p\Omega)\leq\frac{1}{m}\}}|v|^{n+1}\,d\nu\to 0 \,\mbox{when $m\to\infty$ uniformly in $v$ on  bounded subsets of $\E(\Omega)$},
 \end{equation}
then the Monge--Amp\`ere eigenvalue problem 
\begin{equation*}
   \mu_u=\lambda |u|^n\nu \quad\text{in} ~\Omega, \quad
u =0\quad\text{on}~\p \Omega.
\end{equation*}
has a solution $(\lambda, u)\in (0,\infty)\times\E(\Omega)\setminus\{0\}$ where $\lambda$ is uniquely given by
\[\lambda =\lambda[\Omega,\nu]=\sup_{\Sigma_\nu} \Lambda,\]
where $\lambda[\Omega,\nu]$ is defined by \eqref{Radef} and $\Sigma_\nu$ denotes the set of subeigenvalues
\begin{equation*}
\Sigma_\nu:=\Big\{\Lambda\in\R: \mbox{there exists a convex $v\in  C(\overline{\Omega})\setminus\{0\}$, $v=0$ on $\p\Omega$, such that $\mu_v\geq \Lambda|v|^n\nu$}\Big\}.
\end{equation*}
 Moreover, $u$ is unique up to positive multiplicative constants and it is a minimizer of the Rayleigh quotient $R_\nu$ (see \eqref{Radef})
over $\E(\Omega)$.
\end{enumerate}
\end{thm}
As mentioned earlier, the existence and uniqueness part in Theorem \ref{Dist0pn} (ii) is equivalent to Theorem 1.5 in Lu--Zeriahi \cite{LZ}. The estimate \eqref{Dirineq0} is new. 
In Remarks \ref{criticalrem} and \ref{criticalrem2}, we discuss the optimality of the condition \eqref{nuthm1} on the measure $\nu$.
The vanishing mass condition \eqref{vamass} is satisfied when $\nu=\mu_w$ for some convex function $w\in C(\overline{\Omega})$ (see Lemma \ref{vamassho}), so \eqref{nuthm1} implies \eqref{vamass} due to part (ii). Theorem \ref{Dist0pn}(iv) is a reformulation of Theorem 1.4 and  Remark 7.7 in Lu--Zeriahi \cite{LZ}. 
Our formulation \eqref{vamass} emphasizes the importance of controlling the boundary behavior of integrals with respect to 
$\nu$. Example \ref{exaH1} gives an instance of insolvability in the energy class when  \eqref{vamass} fails. The spectral characterization $\lambda[\Omega,\nu]=\sup_{\Sigma_\nu} \Lambda$ is new, and it extends  Birindelli--Payne \cite[Theorem 6.6]{BP} from the Lebesgue measure $\nu=d\L^n$ to singular Borel measures. 

\medskip
We will prove Theorem \ref{Dist0pn} in Sections \ref{Dir_sec} and \ref{deg_sec}. 

\medskip
Our next result analyzes the Monge--Amp\`ere eigenvalue problem even if it might not have solutions in the energy class.  If it has a solution in the class of convex, globally continuous functions, then the measure $\nu$ must satisfy a Poincar\'e-type inequality. It turns out that the 
infimum of the Rayleigh quotient can be approximated from above by  Monge--Amp\`ere eigenvalues of the truncated measures, and by
Rayleigh quotients of 
an inverse iterative scheme. 
The first statement is analogous to the convergence of Monge--Amp\`ere eigenvalues with Lebesgue measure of convex domains converging in the Hausdorff distance.
\begin{thm}[Monge--Amp\`ere eigenvalue problem, Rayleigh quotient,  Poincar\'e inequality, and Inverse scheme]
\label{EVP1}
Let $\nu$ be a locally finite Borel measure on a bounded convex domain $\Omega\subset\R^n$ $(n\geq 1)$ with $\nu(\Omega)>0$. 
 Let $R_\nu$ be the Rayleigh quotient associated with $\nu$ and $\lambda[\Omega, \nu]$ be its infimum  in the energy class (see \eqref{Radef}).
\begin{enumerate}
\item[(i)]  
 For each positive integer $m$, let 
$\nu_m= \chi_{\{x\in\Omega: \dist(x, \p\Omega)>1/m\}}\nu$.
Let $u_m\in\E(\Omega)$ be a nonzero solution to the Monge--Amp\`ere eigenvalue problem (see Theorem \ref{Dist0pn})
\begin{equation*}
   \mu_{u_m}=\lambda_m |u_m|^n\nu_m \quad \text{in} ~\Omega, \quad
u_{m} =0\quad \text{on}~\p \Omega.
\end{equation*}
Then, the sequence $\{\lambda_m\}_{m=1}^\infty$ is nonincreasing and 
\[\lim_{m\to\infty }\lambda_m=\lambda[\Omega,\nu].\]
\item[(ii)] If 
the Monge--Amp\`ere eigenvalue problem
\begin{equation*}
   \mu_{w}=\lambda |w|^n\nu \quad\text{in} ~\Omega, \quad
w=0\quad \text{on}~\p \Omega.
\end{equation*}
has a solution $(\lambda, w)\in (0,\infty)\times C(\overline{\Omega})\setminus \{0\}$ where $w$ is convex, then one has the Poincar\'e-type inequality:
\[R_\nu(u)\geq\lambda\quad\text{for all } u\in \E(\Omega).\]

\item[(iii)] Assume, in addition, $\nu$ satisfies the Poincar\'e-type inequality: There exists a constant $c>0$ such that $R_\nu(u)\geq c$ for all $u\in \E(\Omega)$.
 Fix $u_0\in \E(\Omega)\setminus\{0\}$. 
Consider the following iterative scheme for 
\begin{equation*}
\mu_{u_{k+1}}= R_\nu(u_k) |u_k|^n\nu \quad \text{in } \Omega, \quad
u_{k+1} = 0 \quad \text{on } \partial \Omega,
\end{equation*}
which has a unique solution in $\E(\Omega)$ for each nonnegative integer $k$. Then 
\[\lim_{k\to\infty} R_\nu(u_k) =\lambda[\Omega, \nu]. \]
\item[(iv)] 
Assume that 
there exists a convex function $u\in C(\overline{\Omega})\setminus \E(\Omega)$ with $u=0$ on $\p\Omega$ and
\[\mu_u =\lambda[\Omega, \nu] |u|^n \nu.\]
Then, the Monge--Amp\`ere eigenvalue problem \eqref{eqp1} with the measure $\nu$ and $p=n$ is not solvable in the energy class $\E(\Omega)$. Moreover, there are examples where it could have infinitely many families of eigenvalues and convex eigenfunctions.
\end{enumerate}
\end{thm}
 In Example \ref{exaRn}, we give an example showing that the insolvability situation in Theorem \ref{EVP1} (iv) can actually happen. The measure $\nu$ in Example \ref{exaH1} for nonuniqueness is of Hardy-type; that is, it is the square inverse of the distance to the boundary.
 
 \medskip
We will prove Theorem \ref{EVP1} in Section \ref{deg_sec}.

\subsection{Ingredients of the proofs}
We find nontrivial solutions of \eqref{eqp1} in the proof of Theorem \ref{Dist0pn} (i) by looking for minimizers of the functional
$\displaystyle \frac{\int_\Omega (-u) d\mu_u}{\big(\int_\Omega |v|^{p+1}\, d\nu\big)^{\frac{n+1}{p+1}}}$
among nonzero convex functions $u\in C(\overline{\Omega})$ 
in the finite energy class $\E(\Omega)$.
By now, it is more or less standard to show the existence of a nontrivial minimizer $u$ when $u\mapsto \int_\Omega |u|^{p+1}\, d\nu$ is continuous on bounded subsets of $\E(\Omega)$; see \cite[Theorem 11.5]{Lbook}. However, showing that $u$ (or one of its multiples) solves \eqref{eqp1} is more involved. For this purpose, one needs to use test functions $\varphi\in C_c^{\infty}(\Omega)$ and do variations involving $u+ t\varphi \,(t\in\R)$. These are not convex functions in general for the minimizing argument, so one needs to use the convex envelops $\Gamma_{u+ t\varphi}$ of $u+ t\varphi$ instead. This naturally leads to computing the derivative of $E(\Gamma_{t+ t\varphi})$ at $t=0$.
 \begin{thm}[Variational derivative of  
Monge--Amp\`ere energy of convex envelopes] 
\label{envED}
Let $\Omega$ be a bounded convex domain in $\R^n$, $u\in \E(\Omega)$ and  $v\in \E(\Omega) \cap C^{0, 1}(\overline{\Omega})$. Let $E$ be as in \eqref{Edef}. For $t\in\R$, let $\Gamma_{u+ tv}$ be the convex envelope of $u+ tv$ in $\overline\Omega$; that is, 
\[\Gamma_{u+ tv}(x):=\sup\big\{\varphi(x): \varphi\leq u+ tv\text{ in }\overline{\Omega},\, \varphi \text{ is convex in }\overline{\Omega}\big\}.\]
Then 
\[\frac{d}{dt}\mid_{t=0} E (\Gamma_{u + tv})=(n+ 1)\int_\Omega (-v) \, d\mu_u.\]
\end{thm}
Theorem \ref{envED} has its origin in the complex Monge--Amp\`ere equations in the work of  Berman--Boucksom \cite[Theorem B]{BB} on compact, complex manifolds. See also Lu--Nguyen \cite[Theorem 6.13]{LN} and Guedj--Zeriahi \cite[Theorem 11.11]{GZ}.
Moreover, it is the real Monge--Amp\`ere version of Theorem 4.11 in Lu \cite{Lu} for complex Hessian equations on domains in ${\mathbb C}^n$.  In these contexts, it is referred to as the {\it Projection Theorem}.

Theorem \ref{envED} will be proved in Section \ref{env_sec}.

\medskip
As in the case of Borel measures with finite total masses, the solvability of the Dirichlet problem with singular Borel measures uses certain Aleksandrov-type maximum principles. 
For Theorem \ref{Dist0pn} (ii) and (iii), we need the following uniform Aleksandrov--Jerison maximum principle, improving the classical ones 
 due to Aleksandrov and Jerison (see Theorem \ref{Al_E}).
\begin{thm}[Uniform Aleksandrov--Jerison maximum principle]\label{AJ_thm} Let  $\Omega\subset\R^n$ be a bounded convex domain. Let $u,\tilde u\in C(\overline{\Omega})$ be convex functions
with $u=\tilde u=0$ on $\p \Omega$ and $\mu_u\geq\mu_{\tilde u}$ in $\Omega$. Then for all $\alpha\in [0, 1]$ and all $x_0\in\Omega$, we have
\[
|\tilde u(x_0)-u(x_0)|^{n}\le C(n)[\diam (\Omega)]^{n-1}\dist^{\alpha}(x_0,\partial \Omega)\int_{\Omega}\dist^{1-\alpha}(\cdot,\partial \Omega)\,(d\mu_u-d\mu_{\tilde u}).
\]
\end{thm}
Theorem \ref{AJ_thm} will be proved in Section \ref{maxDir_sec1}.

\medskip
It is worth mentioning the lack of compactness in general of
convex solutions $u_m$ to the Dirichlet problem problem $\mu_{u_m}=\nu_m$ on a fixed bounded convex domain $\Omega\subset\R^n$ where the nonmonotonic Borel measures $\nu_m$ have uniformly bounded mass when weighted against the distance function; see Remark \ref{lackrem}. 
When a fixed measure $\nu$ is truncated to $\nu_m= \chi_{\{x\in\Omega: \dist(x, \p\Omega)>1/m\}}\nu$, the monotonicity of $\nu_m$ transfers to that of $\mu_{u_m}$ (and also $u_m$ due to the classical comparison principle) and we can use  the  Aleksandrov--Jerison maximum principle in Theorem \ref{AJ_thm} to obtain the convex solution to 
the Dirichlet problem $\mu_u=\nu$ on $\Omega$.  

\medskip
Even for the truncated measures $\nu_m$, then due to the dependence of the right-hand side on the solution itself for the equation $\mu_{u_m}=|u_m|^p\nu_m$ when $p>0$, it is not obvious if the monotonicity of $\nu_m$ transfers to that of $\mu_{u_m}$ or $u_m$.
We show that this is in fact the case when the equation is scaling subcritical; that is, when $0<p<n$.
In the proof of Theorem \ref{Dist0pn} (iii), we also will need the following comparison principle.
\begin{thm}[Comparison principle for degenerate subcritical Monge--Amp\`ere equations]
 \label{compnu} Let $p\in (0, n)$, $\Omega$ be a bounded convex domain in $\R^n$, and $\nu$ be a Borel measure in $\Omega$.
 Assume 
 \[\mbox{either (i) $\nu=f \, d\L^n$ where $f\in C(\overline{\Omega})$ with $f>0$ in $\overline\Omega$,\,
 or (ii) $\nu$ is compactly supported in  $\Omega$}.\]
 Let $u, v\in C(\overline{\Omega})$ be convex functions 
 satisfying 
 \begin{enumerate}
 \item[$\bullet$] $v<0$ in $\Omega$, $v\leq 0$ on $\p\Omega$,  and $v$ is a subsolution in the sense that
 $ \mu_v\geq |v|^{p}\nu$ in $\Omega$,
 \item[$\bullet$] $u= 0$ on $\p\Omega$, and $u$ is a supersolution in the sense that
 $ \mu_u\leq |u|^{p}\nu$ in $\Omega$.
 \end{enumerate}
 Then $u\geq v$ in $\Omega$, and  consequently, $\mu_u\leq \mu_v$ in $\Omega$.
  \end{thm}
 The condition $v<0$ in $\Omega$ in Theorem \ref{compnu} could not be relaxed. Indeed, when $\nu=d\L^n$ is the Lebesgue measure, $u$ is a nonzero solution with zero boundary condition as given by \cite[Theorem 4.2]{L} and $v\equiv 0$, we have $u< v$ in $\Omega$. The first case of Theorem \ref{compnu} is a refinement of the comparison principle in \cite[Lemma 2.1]{L_AFST} where the functions involved were required to be $C^2$. The second case of Theorem \ref{compnu} for a general compactly supported measure $\nu$ is new.
 
  Theorem \ref{compnu} will be proved in Section \ref{com_sec}.

\medskip
All of the above results, except Theorem \ref{compnu}, will be proved using the mixed Monge--Amp\`ere measures. They fully linearize the Monge--Amp\`ere operator. We give here a quick definition; see also Definition \ref{mMAdef} where their multi-linearity can be deduced.
\begin{defn}[Mixed Monge--Amp\`ere measure]
Define the mixed Monge--Amp\`ere measure $\mu_n[u_1,\ldots, u_n]$ of $n$ convex functions $u_1, \cdots, u_n$  on $\Omega\subset\R^n$ 
in term of the Monge--Amp\`ere measures by the polarization formula
\begin{equation} \label{polf}\mu_n[u_1,\ldots, u_n]=\frac{1}{n!}\sum_{k=1}^n\sum_{1\leq i_1<\cdots<i_k\leq n}(-1)^{n-k} \mu_{u_{i_1}+\cdots + u_{i_k}}.
 \end{equation}
 The mixed Monge--Amp\`ere measure is a nonnegative Borel measure that is invariant under permutation. 
 Moreover,   the mixed Monge--Amp\`ere measures are multi-linear; that is 
 \[\mu_n[u_1,\ldots, u_{n-1},\alpha v_1 + \beta v_2]=\alpha\mu_n[u_1,\ldots, u_{n-1}, v_1 ] + \beta\mu_n[u_1,\ldots, u_{n-1}, v_2] \]
 for convex functions $u_1,\cdots, u_{n-1}, v_1, v_2$ and positive numbers $\alpha,\beta$.
 
  If $u$ is convex, then its Monge--Amp\`ere and mixed Monge--Amp\`ere measures are related by
 \begin{equation}
 \label{MixMA1}
 \mu_u= \mu_n[u,\ldots, u].
 \end{equation}
\end{defn}

Many of our arguments are based on the following energy estimates which are refinements of Blocki \cite[Theorem 2.1]{Bl} for the mixed complex Monge--Amp\`ere measures. 
\begin{thm}[Blocki-type maximum principles for mixed Monge--Amp\`ere measures] 
\label{BlemR}
Let $\Omega$  be a bounded (not necessarily convex) domain in $\R^n$. Let $v, w\in C(\overline{\Omega})$ be convex functions satisfying $v\leq w$ in $\Omega$, and $v=w$ on $\p\Omega$. Let $u_1, \tilde u_1, \cdots, u_n, \tilde u_n$ be nonpositive bounded convex functions in $\Omega$. 
\begin{enumerate}
\item[(i)] Then 
\[\int_\Omega (w-v)^n \,d\mu_n[u_1, \cdots, u_n] \leq n! \|u_1\|_{L^{\infty}(\Omega)}\cdots \|u_{n-1}\|_{L^{\infty}(\Omega)} \int_\Omega |u_n| \, (d\mu_v-d\mu_w).\]
\item[(ii)] 
Assume $v\leq 0$ in $\Omega$. Let \[d_k:= \|u_k-\tilde u_k\|_{L^{\infty}(\Omega)} ,\quad m_k:= \prod_{i=1}^{k-1} \|u_i\|_{L^{\infty}(\Omega)}\cdot
\prod_{j=k+1}^n\|\tilde u_j\|_{L^{\infty}(\Omega)}.\] Then  
\[\int_\Omega (w-v)^{n+1} \,(d\mu_n[u_1, \cdots, u_n]-d\mu_n[\tilde u_1, \cdots, \tilde u_n])\\ \leq 2(n+1)!\sum_{k=1}^n 
d_k m_k
 \int_\Omega |v| (d\mu_v-d\mu_w).\]
\end{enumerate}
\end{thm}
In \cite{Bl}, the term $\int_\Omega |u_n| \, (d\mu_v-d\mu_w)$ was replaced by the larger expression $\int_\Omega |u_n| \, d\mu_v$. See also Wan \cite[Theorem 1.2]{WD} for similar estimates for the mixed Hessian measures.

Theorem \ref{BlemR} will be proved in Section \ref{maxDir_sec1}.
\medskip

 We will repeatedly use the generalized Cauchy--Schwarz inequality and integration by parts with respect to mixed measures induced by convex functions in the energy class.
\begin{thm}[Generalized Cauchy--Schwarz inequality and Integration by parts] 
\label{gCSIBP}
Let $\Omega\subset\R^n$ be a bounded convex domain. 
Let $u_0, u_1,\cdots, u_n\in \E(\Omega)$, and $E$ be as in \eqref{Edef}. 
\begin{enumerate}
\item[(i)] The following generalized Cauchy--Schwarz inequality holds:
\[\int_\Omega |u_0| \,d\mu_n[u_1, \cdots, u_n] \leq [E(u_0)]^\frac{1}{n+1}\cdots [E(u_n)]^\frac{1}{n+1}.\]
\item[(ii)] The following integration by parts formula holds:
\[\int_\Omega u_0 \,d\mu_n[u_1, \cdots, u_n]= \int_\Omega u_n \,d\mu_n[u_0, u_1, \cdots, u_{n-1}].\]
\end{enumerate}
\end{thm}
Theorem \ref{gCSIBP} (i) (respectively Theorem \ref{gCSIBP} (ii)) is the real Monge--Amp\`ere version of Theorem 3.4 in Persson \cite{Per} (respectively Theorem 3.2 in Cegrell \cite{Ceg}) for complex Monge--Amp\`ere equations.
When $u_0$, $u_1=\cdots= u_n$ are $C^2$ convex functions vanishing at the boundary of smooth and uniformly convex domain $\Omega\subset\R^n$, Theorem \ref{gCSIBP} (i) can be found in 
Verbitsky \cite[Theorem 3.1]{V}.  In Huang \cite[Lemma 4.1]{H}, the smoothness in \cite{V} was removed.
\medskip

Theorem \ref{gCSIBP} will be proved in Section \ref{pol_sec}. 
\medskip

For uniqueness results of the Monge--Amp\`ere eigenvalue problem, we combine integrations by parts with mixed Monge--Amp\`ere inequalities.
\begin{thm}[Mixed Monge--Amp\`ere inequality]
\label{mMAnu0}
Let $\Omega$ be a bounded convex domain in $\R^n$, $\nu$ be a Borel measure in $\Omega$, and
 $0\leq f_i\in L_{\text{loc}}^1(\Omega, d\nu)$ $(1\leq i\leq  n)$. Let $u_1,\cdots u_n \in C(\Omega)$ be convex functions satisfying 
 \[\mu_{u_i}\geq f_i\nu\quad (1\leq i\leq n).\]
Then their mixed Monge--Amp\`ere measure satisfies
 \[\mu_n[u_1, \cdots, u_n] \geq \Big(\prod_{i=1}^{n}f_i^{\frac{1}{n}}\Big) \nu.
\]
Consequently,
\[\mu_{u_1+ u_2}\geq (f_1^{1/n} + f_2^{1/n})^n\nu,\]
and if equality occurs, it is necessary that $\mu_{u_i}= f_i\nu$ for $i=1, 2$.
\end{thm}
Theorem \ref{mMAnu0}, to be proved in Section \ref{mMA_sec2},  is a real mixed Monge--Amp\`ere version of the complex mixed Monge--Amp\`ere inequality established by Dinew \cite[Theorem 1.3]{D09}.

\medskip
\noindent
{\bf Notation.} Throughout, we use the following notation for partial derivatives: $D_i u =\frac{\p u}{\p x_i}$ and $D_{ij} u= \frac{\p^2 u}{\p x_i\p x_j}$  for $1\leq i, j\leq  n$.
 For a Lebesgue measurable set $K \subset \R^n$,  we denote its $n$-dimensional Lebesgue measure by $|K|$. The characteristic function of a set $A$ is denoted by $\chi_A$.   The Dirac measure giving the unit mass to $z$ is denoted by  $\delta_z$. 
 For a Borel measure $\nu$ on a Borel set $\Omega\subset\R^n$ and $p\in [1,\infty)$, we denote $\|u\|_{L^p(\Omega, d\nu)}= \big(\int_\Omega|u|^p\, d\nu\big)^{1/p}$.
 We use $\dist(\cdot, S)$ to denote the distance from a closed set $S$. The set of positive integers is denoted by $\N$. We use $c=(\ast,\ldots,\star)$ and $C= C(\ast,\ldots,\star)$ to denote positive constants $c, C$ depending on the quantities appearing in the parentheses; they may change from line to line.

\medskip
The rest of the paper is organized as follows. In Section \ref{mMA_sec}, we recall some basic results in the Monge--Amp\`ere equations, study properties of the mixed Monge--Amp\`ere measures, and prove the mixed Monge--Amp\`ere measure inequality.
In Section \ref{maxDir_sec}, we prove maximum principles and apply them to establish solvability of the Dirichlet problem. We also prove a comparison principle for degenerate subcritical Monge--Amp\`ere equations in Section \ref{maxDir_sec}. We prove the generalized Cauchy--Schwarz inequality and establish an integration by parts formula in the energy class in Section \ref{pol_sec}.
We study convex envelopes and variational derivatives in Section \ref{env_sec}. In Section \ref{deg_sec}, we use results in previous sections to study degenerate Monge--Amp\`ere equations, including the Monge--Amp\`ere eigenvalue problem.

\section{Mixed Monge--Amp\`ere measures and inequalities}
\label{mMA_sec}
In this section, we study several basic properties of the mixed Monge--Amp\`ere measures and energies. We will prove the mixed Monge--Amp\`ere inequality in Theorem \ref{mMAnu0} which is restated here as Theorem \ref{mMAnu}.

\subsection{Mixed Monge--Amp\`ere measures and energies}
For (not necessarily convex) functions $u_1,\cdots, u_n\in C^2(\Omega)$ where $\Omega\subset\R^n$, we introduce the mixed Monge--Amp\`ere operator  $\tilde{M}[u_1,\ldots, u_n] $ by the explicit representation formula
 \begin{equation}
 \label{tMpol1}
 \tilde{M}[u_1,\ldots, u_n] = \frac{1}{n!} \sum \delta^{i_1,\cdots, i_n}_{j_1,\cdots, j_n} D_{i_1 j_1} u_1 \cdots  D_{i_n j_n} u_n, \end{equation}
 where $\delta^{i_1,\cdots, i_n}_{j_1,\cdots, j_n}$ denotes the generalized Kronecker delta, which vanishes if $\{i_1,\cdots, i_n\}\neq \{j_1,\cdots, j_n\}$ and equals $\pm 1$ depending whether $(i_1,\cdots, i_n)$ is an even or odd permutation of $(j_1,\cdots, j_n)$. 
 
 The mixed Monge--Amp\`ere operator  $\tilde{M}[u_1,\ldots, u_n] $ can be given in term of the Monge--Amp\`ere operators by the polarization formula
\begin{equation} \label{MApolf}\tilde{M}[u_1,\ldots, u_n] =\frac{1}{n!}\sum_{k=1}^n\sum_{1\leq i_1<\cdots<i_k\leq n}(-1)^{n-k} \det D^2(u_{i_1}+\cdots + u_{i_k}).
 \end{equation}
 
 Moreover, we have the following expansion
 \[\det (t_1 D^2 u_1 +\cdots+ t_n D^2 u_n)=\sum_{i_1, \cdots, i_n =1}^n t_{i_1}\cdots t_{i_n} \tilde{M}[u_{i_1},\ldots, u_{i_n}]\quad \text{for all } t_1,\cdots, t_n\in\R.\]
 
 For  $u_1,\cdots, u_n\in C^2(\Omega)$, we denote
\begin{equation}
\label{tMcof1}
\tilde M^{ij} [u_1, \cdots, u_{n-1}] = \frac{\p }{\p (D_{ij} u_n)}\tilde{M}[u_1,\ldots, u_{n-1}, u_n]. \end{equation}
 If $A$ is an $n\times n$ matrix, then let $\text{Cof } A$ be its cofactor matrix whose
 the $(i, j)$ entry $(\text{Cof } A)_{ij}$ is $(-1)^{i+ j}\det A(i\mid j)$ where $A(i\mid j)$ is the $(n-1)\times (n-1)$ matrix obtained by deleting
the $i$-th row and $j$-th column of $A$.
Then 
\begin{equation}
\label{tMcof2} \tilde M^{ij} [u_1, \cdots, u_{n-1}]=\frac{1}{n!}\sum_{k=1}^n\sum_{1\leq i_1<\cdots<i_k= n}(-1)^{n-k}[\text{Cof}\, D^2(u_{i_1}+\cdots + u_{i_k})]_{ij}. \end{equation}
In particular, for a $C^2$ function $u$,
\[n\tilde M^{ij}[\underbrace{u, \cdots, u}_{(n-1)-\text{times}}]=[\text{Cof}\, D^2 u]_{ij},\]
and 
$n\tilde M^{ij}[\underbrace{u, \cdots, u}_{(n-1)-\text{times}}] D_{ij}$
is the linearized operator of the Monge--Amp\`ere operator $\det D^2 u$.

\medskip
Since the cofactor matrix of the Hessian matrix of a $C^3$ function is divergence-free (see \cite[Lemma 2.56]{Lbook}), the matrix
$(\tilde M^{ij} [u_1, \cdots, u_{n-1}])_{1\leq i, j\leq n}$ is divergence-free; that is
\begin{equation}
\label{div0}
\sum_{i=1}^nD_i \tilde M^{ij} [u_1, \cdots, u_{n-1}]=0\quad \text{for each } j=1,\cdots, n.\end{equation}
We thus write $\tilde M[u_1,\ldots, u_n] $ in the divergence form when $u_1,\cdots, u_n\in C^3(\Omega)$:
\begin{equation}
\label{divf}
\begin{split}\tilde M [u_1, \cdots, u_{n-1}, u_n]&=\sum_{i, j=1}^n D_{ij} u_n\tilde M^{ij} [u_1, \cdots, u_{n-1}] \\
&=\sum_{i, j=1}^n D_i\big(\tilde M^{ij} [u_1, \cdots, u_{n-1}] D_j u_n\big).
\end{split}
\end{equation}
 
 \begin{rem}
 \label{mMpos}
 When $u_1, \cdots, u_n$ are $C^2$ convex functions, then we have these crucial facts.
 \begin{enumerate}
 \item By G\r{a}rding's inequality \cite[Theorem 5]{Gd}, we have
 \[\tilde M [u_1, \cdots, u_{n-1}, u_n]\geq \prod_{i=1}^n \Big(\tilde M [u_i, \cdots, u_i, u_i]\Big)^{1/n}= \prod_{i=1}^n  (\det D^2 u_i)^{1/n}\geq 0.\]
  \item In the sense of symmetric matrices
 \begin{equation}
 \label{tMpos}
 \Big(\tilde M^{ij} [u_1, \cdots, u_{n-1}]\Big)_{1\leq i, j\leq n}\geq 0.\end{equation}
 To see this, let $\xi=(\xi_1,\cdots, \xi_n)\in\R^n$ and $v(x) =e^{\xi\cdot x}>0$. Then $v$ is convex with $D^2 v= v\xi\otimes \xi $. The assertion \eqref{tMpos} now follows from
 \[0\leq \tilde M [u_1, \cdots, u_{n-1}, v] = v \tilde M^{ij} [u_1, \cdots, u_{n-1}]\xi_i\xi_j.\]
 \end{enumerate}
 \end{rem}
 Following Trudinger--Wang \cite[Lemma 2.3 and Theorem 2.4]{TW02} and Passare--Rullg\r{a}rd \cite[formula (14)]{PR}, we can give another definition of the mixed Monge--Amp\`ere measures of $n$ convex functions as follows.
 \begin{defn}[Mixed Monge--Amp\`ere measures]
 \label{mMAdef}
Let $\Omega$ be a bounded domain in $\R^n$. 
 Let $u_1, \cdots, u_n$ be convex functions on $\Omega$. Then, 
 there exists a Borel measure $\mu_n[u_1, \cdots, u_n]$, called {\it the mixed Monge--Amp\`ere measure} of $u_1,\cdots, u_n$, such that
 \[\mu_n[u_1, \cdots, u_n]=\tilde M[u_1,\cdots, u_n] \quad \text{when }u_1,\cdots, u_n\in C^2(\Omega),\] and 
 if $\{u_{1, k}\}_{k=1}^{\infty}, \cdots$, $\{u_{n, k}\}_{k=1}^{\infty}$ are $n$ sequences of convex functions 
 converging locally uniformly in $\Omega$ to convex functions
  $u_1, \cdots, u_n$, respectively, then $\mu_n[u_{1,k},\cdots, u_{n,k}]$ converges weakly to $\mu_n[u_1,\cdots, u_n]$; that is,
  \[\lim_{k\to\infty} \int_\Omega\varphi\, d\mu_n[u_{1,k},\cdots, u_{n,k}]=\int_\Omega\varphi\, d\mu_n[u_1,\cdots, u_n]\quad\text{for all } \varphi \in C_c(\Omega).\]
  \end{defn} 
 
\begin{rem}
Trudinger--Wang \cite{TW02} analyzed the general case of mixed Hessian measures. For the mixed Monge--Amp\`ere measures, the 
polarization formula \eqref{polf} allows us to deduced most of their convergence properties from those of Monge--Amp\`ere measures as exposed in \cite[Chapter 3]{Lbook}.
\end{rem}
An important concept in our analysis is that of mixed Monge--Amp\`ere energy.
\begin{defn}[Mixed Monge--Amp\`ere energy] 
\label{mEdef}
Let $\Omega$ be a bounded convex domain in $\R^n$.
For $u_0, u_1, \cdots, u_n\in \E(\Omega)$, we define their mixed Monge--Amp\`ere energy by
\[E(u_0, u_1, \cdots, u_n)=\int_\Omega -u_0\, d\mu_n[u_1,\cdots, u_n]. \]
Clearly,
\[E(u)=E(\underbrace{u, \cdots, u}_{(n+1)-\text{times}})\quad \text{for } u\in \E(\Omega).\]
\end{defn}
The generalized Cauchy--Schwarz inequality and integration by parts formula in Theorems \ref{gCS} and \ref{IBPf} will confirm the finiteness of $E(u_0, u_1, \cdots, u_n)$ and its symmetry in all arguments.

\begin{exa} Here is an example computing the mixed Monge--Amp\`ere measures.
 Consider the following convex functions in $\R^n$: 
\[u(x)=|x|^2/2, \quad v(x)=|x|.\]
To compute $\mu_{u+ v}$, we first compute $\det D^2 w_\e$ where $w_\e(x) =|x|^2/2 +\sqrt{|x|^2 +\e}$, and then take the weak limit of $\det D^2 w_\e$ when $\e\to 0^+$. 

Since $w_\e(x)= W_\e(r)$ where $r=|x|$ and $W_\e(r) = r^2/2 +\sqrt{r^2 +\e}$, we have
\[\det D^2 w_\e = W''_\e(r)(W'_\e/r)^{n-1}=[1+ \e(r^2 +\e)^{-3/2}] [1 + (r^2 +\e)^{-1/2}]^{n-1}.\]
Taking the limit $\e\to 0^+$, we find
\[\mu_{u+ v} = (1+r^{-1})^{n-1}d\L^n + |B_1(0)|\delta_0.\]
Since 
\[\mu_{u+ v}=\mu_n[u+ v, \cdots, u+ v]= \sum_{k=0}^n {n \choose k}\mu_n[\underbrace{u,\cdots,u}_{(n-k)-\text{times}}, \underbrace{v, \cdots, v}_{k-\text{times}}],\]
we obtain
\[\mu_n[\underbrace{u,\cdots,u}_{(n-k)-\text{times}}, \underbrace{v, \cdots, v}_{k-\text{times}}]= {n \choose k}^{-1}{n-1\choose k}r^{-k}\, d\L^n= \frac{n-k}{n} r^{-k}\, d\L^n\quad (0\leq k\leq n-1).\]
\end{exa}

\begin{rem}[Finite masses of mixed Monge--Amp\`ere measures for convex functions with finite Monge--Amp\`ere masses]
\label{MLiptot}
 Let $\Omega\subset\R^n$ $(n\geq 1)$ be a bounded convex domain.
\begin{enumerate}
 \item If  $u\in C^{0, 1}(\overline{\Omega})$ is convex with Lipschitz constant not greater than $L$ then $\p u(\Omega)\subset B_L(0)$ and thus the total measure $\mu_u(\Omega)$ satisfies \[\mu_u(\Omega) =|\p u(\Omega)|\leq |B_L(0)|=C(n, L).\]
 \item If  $u_1, \cdots, u_n\in  C^{0, 1}(\overline{\Omega})$ where each convex function $u_i$ has a Lipschitz constant not greater than $L_i$, then using the polarization formula \eqref{polf}, we also have \[\mu_n[u_1,\cdots, u_n](\Omega)\leq C(n, L_1, \cdots, L_n)<\infty.\]
 \item If  $u_1, \cdots, u_n\in  C(\overline{\Omega})$ are convex functions with $u_1=\cdots=u_n=0$ on $\p\Omega$ and \[\mu_{u_1}(\Omega)+\cdots+ \mu_{u_n}(\Omega)<\infty,\]
 then the polarization formula \eqref{polf} and the first assertion of Lemma \ref{coneMA} give
 \[\mu_n[u_1,\cdots, u_n](\Omega)\leq C(n)[\mu_{u_1}(\Omega)+\cdots+ \mu_{u_n}(\Omega)]<\infty.\]
 \end{enumerate}
\end{rem}

\medskip
Convex functions with finite Monge--Amp\`ere masses form a cone. Cegrell \cite{Ceg1} seems to be the first to observe this phenomenon in the complex Monge--Amp\`ere setting. 
\begin{lem}[Cone property]
\label{coneMA}
 Let $\Omega\subset\R^n$ be a bounded convex domain. Let $u, v\in C(\overline{\Omega})$ be convex functions with $u=v=0$ on $\p\Omega$. 
 Then
 \[\mu_{u+ v}(\Omega)\leq 3^n[\mu_u(\Omega)+\mu_v(\Omega)]. \] 
 More generally, if $w\in C(\overline{\Omega})$ is a nonnegative convex function, then
 \[\int_\Omega (-w)\, d\mu_{u+ v} \leq 3^n \Big[\int_\Omega (-w)\, d\mu_u  + \int_\Omega (-w)\, d\mu_ v \Big].\]
\end{lem}
\begin{proof} The proof here is inspired by Cegrell \cite[Section 2]{Ceg1}.
It suffices to consider the case when both $u$ and $v$ are not identically zero. Then, by convexity, $u, v<0$ in $\Omega$ and therefore
\[\Omega=\{x\in\Omega: u(x)<v(x)\}\cup \{x\in\Omega: u(x)>2v(x)\}\equiv \Omega_1 \cup\Omega_2.\]
We have 
$u+ v> 2u$ in $\Omega_1$ while $u+ v=2u$ on $\p\Omega_1$. By the maximum principle for the Monge--Amp\`ere equation (\cite[Lemma 3.11]{Lbook}), we have 
$\p (u+ v)(\Omega_1)\subset \p (2u)(\Omega_1)$. Therefore
\[\mu_{u+ v}(\Omega_1)\leq 2^n \mu_u(\Omega_1).\]
Similarly, we have $u+ v> 3v$ in $\Omega_2$ while $u+ v=3v$ on $\p\Omega_2$, and
\[\mu_{u+ v}(\Omega_2)\leq 3^n \mu_v(\Omega_2).\]
Adding these two inequalities proves the first assertion of the lemma.

Assume now $w\in C(\overline{\Omega})$.  Then, by Corollary \ref{mpcor}, we have
\[\int_{\Omega_1} (-w) \, d\mu_{u+ v} \leq \int_{\Omega_1} (-w) \, d\mu_{2u}= 2^n \int_{\Omega_1} (-w) \, d\mu_{u}, \]
and 
\[\int_{\Omega_2} (-w) \, d\mu_{u+ v} \leq \int_{\Omega_2} (-w) \, d\mu_{3v}= 3^n \int_{\Omega_1} (-w) \, d\mu_{v}. \]
The second assertion of the lemma follows.
\end{proof}

In \cite[Proposition 11.1]{Lbook}, we proved the following convergence result for integrals of Monge--Amp\`ere measures: 
 \begin{prop}
 \label{I_conti}
 Let $\Omega$ be a bounded convex domain in $\R^n$.
 Let $\{u_k\}_{k=1}^{\infty}, \{f_k\}_{k=1}^{\infty} \subset C(\overline\Omega)$ be sequences of functions converging uniformly on $\overline\Omega$ to
  $u, f\in C(\overline{\Omega})$, respectively. 
 Furthermore, assume that $f=0$ on $\p\Omega$, $u_k$ and $u$ are convex with
 $\sup_{k\geq 1}\mu_{u_k}(\Omega)<\infty$.
Then 
  \[\lim\limits_{k \to \infty} \int_\Omega f_k\, d \mu_{u_k} =\int_\Omega f\, d\mu_u.\]
 \end{prop}

 Using the first assertion of Lemma \ref{coneMA} and the polarization formula \eqref{polf}, we have the following consequence of Proposition \ref{I_conti} for mixed Monge--Amp\`ere measures.
 \begin{prop}
 \label{I_conti2}
 Let $\Omega$ be a bounded convex domain in $\R^n$.
 Let $\{u_{0, k}\}_{k=1}^{\infty}$, $\{u_{1, k}\}_{k=1}^{\infty}, \cdots$, $\{u_{n, k}\}_{k=1}^{\infty}$ be sequences of functions in $C(\overline\Omega)$
 converging uniformly on $\overline\Omega$ to
  $u_0, u_1, \cdots, u_n\in C(\overline{\Omega})$, respectively. 
 Furthermore, assume that $u_0=0$ on $\p\Omega$, $u_{i, k}$ and $u_i$ ($1\leq i\leq n, \, k\geq 1$) are convex, and there is a constant $C$ such that
 $\sup_{k\geq 1,1\leq i\leq n}\mu_{u_{i, k}}(\Omega)\leq C<\infty$.
Then 
  \[\lim\limits_{k \to \infty} \int_\Omega u_{0, k}\, d \mu_n[u_{1, k},\cdots, u_{n, k}] =\int_\Omega u_0\, d\mu_n[u_1,\cdots, u_n].\]
 \end{prop}

\medskip
We now state a null Lagrangian property of the mixed Monge--Amp\`ere operator. 
\begin{lem} 
\label{IBPC3}
Let $\Omega$ be a bounded domain in $\R^n$ and $u_0, u_1,\cdots, u_n\in C^3(\overline{\Omega})$. 
\begin{enumerate}
\item[(i)] Assume $u_0$ vanishes in a neighborhood of  $\p\Omega$. Then
\[\int_\Omega u_0\tilde M[u_1, \cdots, u_n]\,dx=\int_\Omega u_n\tilde M[u_1, \cdots, u_{n-1}, u_0]\,dx.\]
\item[(ii)] Assume $u_0=u_1$ in a neighborhood of  $\p\Omega$. Then
\[\int_\Omega \tilde M[u_0, u_2, \cdots, u_n]\, dx=\int_\Omega \tilde M[u_1, u_2, \cdots, u_n]\, dx.\]
\end{enumerate}
\end{lem}
\begin{proof} We prove part (i) where $u_0=0$ and $Du_0=0$ on $\p\Omega$. We use \eqref{div0} and \eqref{divf} and integrating by parts twice to obtain
\begin{equation*}
\begin{split}
\int_\Omega u_0\tilde M[u_1, \cdots, u_n]\,dx &=\int_\Omega \sum_{i, j=1}^n D_{ij}\big(\tilde M^{ij} [u_1, \cdots, u_{n-1}] u_n\big) u_0\, dx\\
&=\int_\Omega \sum_{i, j=1}^n \tilde M^{ij} [u_1, \cdots, u_{n-1}]  u_n  D_{ij} u_0\, dx \\
&=\int_\Omega u_n\tilde M[u_1, \cdots, u_{n-1}, u_0]\,dx.
\end{split}
\end{equation*}
The proof of part (ii) is similar so we skip it. 
\end{proof}
We have the following monotonicity result.
\begin{lem}[Monotonicity]
\label{Exchcon}
 Let $\Omega\subset\R^n$ be a bounded domain. Let $u, v, w\in C^3(\overline{\Omega})$ be convex functions with $w\geq v$ in $\Omega$ and $w=v$ in a neighborhood of $\p\Omega$.
Then 
 \begin{equation}
 \label{monoB}
 \int_{\Omega}  \tilde M [v, \cdots, v, w] u\,dx\leq  \int_{\Omega}  \tilde M [w, \cdots, w, w] u\,dx.\end{equation}
\end{lem}
\begin{proof} The difference $B$ between the right-hand side and the left-hand sides of \eqref{monoB} equals 
\[B= \int_{\Omega}  \tilde M [w, \cdots, w-v, w] u\,dx +\int_{\Omega}  \big(\tilde M [w, \cdots, w, v, w] - \tilde M [v, \cdots, v, v, w]\big)u\,dx.\]
Using the exchange result in Lemma \ref{IBPC3} (i) for  $w-v=0$ in a neighborhood of $\p\Omega$ and $w\geq v$ together with Remark \ref{mMpos} (i), we have 
 \begin{equation*}
\begin{split}
B &=   \int_{\Omega}  \tilde M [w, \cdots, u, w] (w-v)\,dx +\int_{\Omega}  \big(\tilde M [w, \cdots, w, v, w] - \tilde M [v, \cdots, v, v, w]\big)u\,dx\\
&\geq \int_{\Omega}  \big(\tilde M [w, \cdots, w, v, w] - \tilde M [v, \cdots, v, v, w]\big)u\,dx.
\end{split}
\end{equation*}
Continuing, we find
 \begin{equation*}
\begin{split}
B&\geq \int_{\Omega}  \big(\tilde M [w, v, \cdots, v, v, w] - \tilde M [v, v, \cdots, v, v, w]\big)u\,dx \\&=   \int_{\Omega}  \tilde M [u, v, \cdots,v, v, w] (w-v)\,dx\geq 0.
\end{split}
\end{equation*}
Therefore, \eqref{monoB} is proved.
\end{proof}
As a consequence of the maximum principle and the locality of the normal mapping \cite[Lemma 3.11 and Remark 2.31]{Lbook}, we have the following locality property of the mixed Monge--Amp\`ere measures.
\begin{lem}[Locality of mixed Monge--Amp\`ere measure] 
\label{mMAloc}
Let $\Omega\subset\R^n$ be a bounded  
domain. Let $u_1, \bar u_1, u_2, \cdots, u_n\in C(\overline{\Omega})$ be convex functions where $u_1= \bar u_1$ in 
a compact subset $K\subset \overline\Omega$. 
\begin{enumerate}
\item[(i)] Then, in the interior $\text{int}(K)$ of $K$, we have
\[ \mu_n[u_1,u_2, \ldots, u_n](\text{int}(K)) = \mu_n[\bar u_1,u_2, \ldots, u_n](\text{int}(K)).\]
\item[(ii)] Moreover, if $K$ contains a neighborhood of $\p\Omega$, then
\[ \mu_n[u_1,u_2, \ldots, u_n](\Omega) = \mu_n[\bar u_1,u_2, \ldots, u_n](\Omega).\]
\end{enumerate}
\end{lem}
\begin{proof} By \cite[Lemma 3.11 and Remark 2.31]{Lbook}, at any $x\in \text{int}(K)$, the normal mapping of any sum of convex functions $u_{i_1} + \cdots+ u_{i_k}$ is the same when $u_1$ is replaced by $\bar u_1$. Thus, using the polarization formula \eqref{polf}, we easily deduce the conclusion of part (i). 

\medskip
For part (ii), assume that there is an open set $\Omega'\Subset\Omega$ such that $u_1=\bar u_1 \quad \text{in }\overline{\Omega}\setminus \Omega'$. Choose an open set $\Omega''$ such that $\Omega'\Subset\Omega''\Subset\Omega$ and
\begin{equation}
\label{locp0}
\mu_n[u_1,u_2, \ldots, u_n](\p\Omega'') = \mu_n[\bar u_1,u_2, \ldots, u_n](\p\Omega'').\end{equation}
For each $0<\e<\min\{\dist(\p\Omega'', \p\Omega), \dist(\p\Omega', \p\Omega'')\}$, let $u_{i, \e}$ be the convolution of $u_i$ with the standard mollifier $\varphi_\e\in C_c^\infty (B_\e(0))$. We define $\bar u_{1,\e}$ similarly. Then $u_{1,\e}= \bar u_{1,\e}$ in a neighborhood of $\p\Omega''$. By Lemma \ref{IBPC3} (ii), we have
\[ \mu_n[u_{1,\e},u_{2,\e}, \ldots, u_{n,\e}](\Omega'') = \mu_n[\bar u_{1,\e},u_{2,\e}, \ldots, u_{n,\e}](\Omega'').\]
Letting $\e\to 0$ and using \eqref{locp0}, we obtain
\[\mu_n[u_1,u_2, \ldots, u_n](\Omega'') = \mu_n[\bar u_1,u_2, \ldots, u_n](\Omega'').\]
By part (i), we have
\[\mu_n[u_1,u_2, \ldots, u_n](\Omega\setminus\overline{\Omega''}) = \mu_n[\bar u_1,u_2, \ldots, u_n](\Omega\setminus\overline\Omega'').\]
Combining these two equalities, we obtain the conclusion of part (ii).
\end{proof}
\begin{rem}
In many arguments in this paper, we will extensively use approximations. However,
it is generally hard to directly carry out the approximation procedure since we do not know how to control the Monge--Amp\`ere masses near the boundary $\p\Omega$. This difficulty can be overcome if we know more information in a neighborhood of $\p\Omega$ as in Lemma \ref{mMAloc} (ii).
{\it Here is a useful trick for working with integrands involving $w-v$ where $w=v$ on $\p\Omega$ and $w\geq v$ in $\Omega$;} see Blocki \cite{Bl}.
For any $\e>0$, let $w_\e=\max\{v, w-\e\}$. Then $w_\e$ is convex, $w_\e\geq v$,  and $w_\e=v$ in a neighborhood of $\p\Omega$. This way, we can do approximations using locality. 
\end{rem}
We recall the basic existence and uniqueness result for solutions to the Dirichlet problem with convex boundary data for the Monge--Amp\`ere equation;  see Blocki \cite[Proposition 2.1]{Bl} and Hartenstine \cite[Theorem 1.1]{Har1}  (and also  
 \cite[Theorem 3.39]{Lbook}).
\begin{thm}[The Dirichlet problem] 
\label{Dir_thm}
Let $\Omega$ be a bounded convex domain in $\R^n$, and let $\nu$ be a Borel measure in $\Omega$ with $\nu(\Omega)<\infty$.
Let $\varphi\in C(\overline{\Omega})$ be a convex function.
Then there exists a unique convex function $u\in C(\overline{\Omega})$ that is an Aleksandrov solution of
\begin{equation*}
   \det D^{2} u=\nu \quad\text{in} ~\Omega, \quad
u =\varphi\quad\text{on}~\p \Omega.
\end{equation*}
\end{thm}
\begin{rem}[Standard approximants] \label{Stappro}
 Let $w\in C(\overline{\Omega})$ be a convex function on a bounded convex domain $\Omega\subset\R^n$.
For each  $m\in \N$, by Theorem \ref{Dir_thm}, there exists a unique convex function $w_m\in C(\overline{\Omega})$ that is an Aleksandrov solution of
\begin{equation*}
   \det D^{2} w_m=\chi_{\{x\in\Omega: \dist(x, \p\Omega)>1/m\}} \mu_w\quad\text{in} ~\Omega, \quad
w_m =w\quad\text{on}~\p \Omega.
\end{equation*}
Let $h\in C(\overline{\Omega})$ be the harmonic function in $\Omega$ with boundary value $w$. 
By the maximum and comparison principles, we have \[h\geq w_m\geq w\quad\text{in}\quad\Omega.\] The sequence $\{w_m\}_{m=1}^{\infty}$ is nonincreasing and uniformly bounded. It converges locally uniformly to $w$ in $\Omega$.
By Dini's theorem, $\{w_m\}_{m=1}^{\infty}$ converges uniformly to $w$ in $\overline{\Omega}$. We call the functions $w_m$ {\it standard approximants}  of $w$.
\end{rem}
For later references, we state  the celebrated Aleksandrov maximum principle, an energy estimate \cite[Theorem 11.4]{Lbook}, and Jerison's estimate \cite[Lemma 7.3]{J}, together with the comparison principle for the Monge--Amp\`ere equation \cite[Theorem 3.21]{Lbook}.
\begin{thm}
\label{Al_E}
Let $\Omega\subset \R^n$ be a bounded convex domain. Let $u\in C(\overline{\Omega})$ be a convex function 
with $u=0$ on $\p\Omega$. Then the following statements hold for all $x_0\in\Omega$.
\begin{enumerate}
\item[(i)] (Aleksandrov's maximum principle)
\[
|u(x_0)|^{n}\le C(n)[\diam\Omega)]^{n-1}\emph{dist}(x_0,\partial \Omega)\mu_u(\Omega).
\]
\item[(ii)] (Energy estimate)
\[|u(x_0)|^{n+1}\leq C(n) [\diam(\Omega)]^{n-1} \text{dist} (x_0, \p\Omega)\int_\Omega |u|\, d\mu_u.\]
\item[(iii)] (Jerison's estimate) Assume further that $\Omega$ is normalized; that is, 
$B_1(b)\subset \Omega\subset B_n(b)$ for some $b\in\R^n$. Then for each $\alpha\in (0, 1]$, there exists a constant $C(n,\alpha)$ 
such that
\[
|u(x_0)|^{n}\le C(n,\alpha)\delta^{\alpha}(x_0,\Omega)\int_{\Omega}\delta^{1-\alpha}(x, \Omega)\,d\mu_u,
\]
where  $\delta(x,\Omega)$ is the normalized distance from $x\in\Omega$ to the boundary:
\[\delta(x,\Omega) =\min\Big\{ |x-y|/|x-z|: \quad y, z\in\p\Omega\text{ and } x, y, z \text{ are collinear}\Big\}.\]
\end{enumerate}
\end{thm}
By Theorem \ref{AJ_thm}, the constant $C(n,\alpha)$ in Theorem \ref{Al_E} (iii) can be taken to be independent of $\alpha$ when $\alpha\to 0^+$.
\begin{thm}[Comparison principle]
\label{compa1}
 Let $\Omega\subset \R^n$ be a bounded convex domain. Let $u, v\in C(\overline{\Omega})$ be convex functions such that
\[\mu_u\leq \mu_v  \quad\text{in }\Omega \quad\text{and } \quad u\geq v \quad\text{on }\Omega.\]
Then $u\geq v$ in $\Omega$.
\end{thm}
Another tool to compare two convex functions is the domination principle.
\begin{lem}[Domination principle] 
\label{domi_prin}
Let $\Omega$ be a bounded convex domain in $\R^n$. Let $u, v\in C(\overline{\Omega})$ be convex functions satisfying
\[u\geq v\quad\text{on }\p\Omega,\quad \mu_u(\{u<v\})=0.\]
Then $u\geq v$ in $\Omega$.
\end{lem}
\begin{proof} Let $w(x)=|x|^2-M$ where $M$ is large so that $w\leq -1$ in $\overline{\Omega}$. For $\e>0$, let $v_\e= v+ \e w$. Then $\{u<v_\e\}\Subset\{u<v\}\subset\Omega$ (due to $u\geq v$ on $\p\Omega$). It follows that
\[\mu_u(\{u<v_\e\})= \mu_u(\{u<v\})=0.\]
Since $u=v_\e$ on $\p \{u<v_\e\}\Subset\Omega$, the maximum principle gives
\[\mu_{v_\e}(\{u<v_\e\}) \leq \mu_u(\{u<v_\e\})=0.\]
Since
$\mu_{v_\e} \geq \mu_v + (2\e)^n$ (see \cite[Lemma 3.10]{Lbook}), we find 
$|\{u<v_\e\}|=0$.
This implies that $\{u<v_\e\}$ is empty, so $u\geq v_\e= v+ \e w$ in $\Omega$. Letting $\e\to 0$ gives $u\geq v$ in $\Omega$, as desired.
\end{proof}
\subsection{Mixed Monge--Amp\`ere inequalities}
\label{mMA_sec2}
We are now ready to prove the mixed Monge--Amp\`ere inequality in Theorem \ref{mMAnu0} that is restated here.
\begin{thm}[Mixed Monge--Amp\`ere inequality] 
\label{mMAnu}
Let $\Omega$ be a bounded convex domain in $\R^n$, $\nu$ be a Borel measure in $\Omega$,
and $0\leq f_i\in L_{\text{loc}}^1(\Omega, d\nu)$ $(1\leq i\leq n)$. Let $u_1,\cdots u_n \in C(\Omega)$ be convex functions satisfying 
 $\mu_{u_i}\geq f_i\nu\quad (1\leq i\leq n)$.
Then their mixed Monge--Amp\`ere measure satisfies
 \begin{equation}
 \label{mixedineq1}
 \mu_n[u_1, \cdots, u_n] \geq \Big(\prod_{i=1}^{n}f_i^{\frac{1}{n}}\Big) \nu.
\end{equation}
Consequently,
\[\mu_{u_1+ u_2}\geq (f_1^{1/n} + f_2^{1/n})^n\nu,\]
and if equality occurs, it is necessary that $\mu_{u_i}= f_i\nu$ for $i=1, 2$.
\end{thm}
Our proof follows the strategy of the proof of the complex mixed Monge--Amp\`ere inequality in Dinew \cite[Theorem 1.1]{D09}.
An approximation procedure for measures is needed.
 Consider a bounded domain $B$ in $\R^n$. 
Let $\mu$ be a Borel measure on $\R^n$ with support in $\overline{B}$ and\[\mu(\p B)=0,\quad \mu(B)<\infty.\] Let $K\subset\R^n$ be a cube  containing $B$. For each $m\in\N$, consider a subdivision $\mathcal{D}_m$ of $K$ into $2^{mn}$ congruent semi open cubes $K_m^j$ $(1\leq j\leq 2^{mn})$. Without loss of generality, we can assume that $\mu(\bigcup_{K^j_m\in\mathcal{D}_m} \p K^j_m)=0$; otherwise we can shift the boundaries a bit for each $K^j_m$.  Define
\begin{equation}\label{canoapp}\mu^m:= \sum_{j=1}^{2^{mn}} \frac{\mu(K^j_m\cap B)}{|K^j_m\cap B|}\chi_{K^j_m}.\end{equation}
Then $\mu^m$ converges weakly to $\mu$ with $\mu^m(B)\leq \mu(B)<\infty$ and each $\mu^m$ is a bounded function. Following Dinew \cite{D09}, we call the sequence $\{\mu^m\}_{m=1}^{\infty}$ a {\it canonical approximation} of $\mu$ in $B$.

\begin{proof}[Proof of Theorem \ref{mMAnu}] Since the inequality \eqref{mixedineq1} is of local nature, it suffices to prove it 
in any ball $B$ such that $B\Subset K\Subset\Omega$ and $\mu_{u_i}(\p B)=0$ for all $1\leq i\leq n$, where $K$ is a cube. 

\medskip
\noindent
{\it Step 1. } Consider the case $\nu$ is the Lebesgue measure and $\mu_{u_i}= f_i\, d\L^n\in L^1_{\text{loc}}(\Omega)$. Redefine $\mu_{u_i}=0$ outside $B$. For $\e\in (0, 1)$, let $\varphi_\e\in C_c^{\infty}(B_\e(0))$ be a standard mollifier and set
\[f_{i, \e}= \mu_{u_i}\ast \varphi_\e +\e.\]
Restricting our discussion to $\overline{B}$, we have $f_{i,\e}\in C^\infty(\overline{B})$, and
\begin{equation}
\label{fie}
\e\leq f_{i, \e},\quad \|f_{i,\e}\|_{L^1(B)} \leq  \|\mu_{u_i}\|_{L^1(B)}+\e\leq C,\quad f_{i,\e}\overset{\e\to 0}\longrightarrow \mu_{u_i}\quad\text{in } L^1_{\text{loc}}(B),\end{equation}
where $C$ is independent of $\e$.
Let $\{g_{i,\e}\}\subset C^{\infty}(\p B)$ be smooth functions such that $g_{i,\e}\to u_i\vert_{\p B}$ uniformly on $\p B$. 
Let $u_{i,\e}\in C(\overline{B})$
be the convex solution to 
\begin{equation*}
   \mu_{u_{i,\e}}~=f_{i, \e} \quad \text{in} ~B, \quad
u_{i,\e} =g_{i,\e}\quad \text{on}~\p B.
\end{equation*}
Then, $u_{i,\e}\in C^\infty(B)$ (see, \cite[Theorem 4.12]{Lbook}) and
$\mu_{u_{i,\e}}(B)=\|f_{i,\e}\|_{L^1(B)}\leq C$.
By the compactness property of the Monge--Amp\`ere equation (see, \cite[Theorem 3.35]{Lbook}), $u_{i,\e}$ converges locally uniformly in $B$ to $u_{i}$ when $\e\to 0$.

By G\r{a}rding's inequality, we have
\[\mu_n[u_{1,\e}, \cdots, u_{n,\e}]=\tilde M[u_{1,\e}, \cdots, u_{n,\e}] \geq \prod_{i=1}^{n} (\det D^2 u_{i,\e})^{1/n}= \prod_{i=1}^{n} f_{i,\e}^{1/n}.\]
  In view of \eqref{fie}, letting $\e\to 0$ gives 
  $\mu_n[u_1, \cdots, u_n] \geq \prod_{i=1}^{n}\mu_{u_i}^{\frac{1}{n}}$, which establishes \eqref{mixedineq1}.
  
 \medskip
\noindent 
{\it Step 2.} Consider the general case. Let $\{\mu_i^m\}_{m=1}^{\infty}$ be a canonical approximation of $\mu_{u_i}$ in $B$, as in \eqref{canoapp}. 
Let $u_{i, m}\in C(\overline{\Omega})$ be the convex solution to 
\[\mu_{u_{i, m}}=\mu_{u_i}^m\quad\text{in } B,\quad u_{i, m}=u_i\quad\text{on } \p B.\]
The functions $u_{i, m}$ exist by Theorem \ref{Dir_thm}, and converge locally uniformly to $u_i$ in $B$ when $m\to\infty$. Therefore, we have, weakly as Borel measures, 
\[\mu_n[u_1, \cdots, u_n] =\lim_{m\to\infty} \mu_n[u_{1,m}, \cdots, u_{n,m}].\]
By the case of $L^\infty$ density Monge--Amp\`ere measures in Step 1, we have
\[\mu_n[u_{1,m}, \cdots, u_{n,m}]\geq \prod_{i=1}^{n}\mu_{u_{i, m}}^{\frac{1}{n}}\geq \sum_{j=1}^{2^{mn}} \frac{\Big(\prod_{i=1}^{n}\int_{K^j_m} d\mu_{u_{i}}\Big)^{1/n}}{|K^j_m\cap \Omega|}\chi_{K^j_m}.\]
Using $\mu_{u_i}\geq f_i\nu$ and the H\"older inequality to the last term, we find 
\begin{equation*}
\mu_n[u_{1,m}, \cdots, u_{n,m}]\geq \sum_{j=1}^{2^{mn}} \frac{\Big(\prod_{i=1}^{n}\int_{K^j_m} f_i d\nu\Big)^{1/n}}{|K^j_m\cap \Omega|}\chi_{K^j_m} 
\geq \sum_{j=1}^{2^{mn}} \frac{\int_{K^j_m} \Big(\prod_{i=1}^{n}f_i \Big)^{1/n}d\nu}{|K^j_m\cap \Omega|}\chi_{K^j_m}. 
\end{equation*}
Letting $m\to\infty$, we obtain \eqref{mixedineq1}.

\medskip
For the consequence, we recall $\mu_{u_1+ u_2}=\mu_n[u_1+ u_2, \cdots, u_1+ u_2]$ and use \eqref{mixedineq1} to deduce
\begin{equation*}
\mu_{u_1+ u_2}= \sum_{m=0}^n  {n\choose m}d\mu_n[\underbrace{u_1,\cdots, u_1}_{m\text{ times}}, \underbrace{u_2,\cdots, u_2}_{(n-m)\text{ times}}]
\geq \sum_{m=0}^n  {n\choose m} f_1^{\frac{m}{n}} f_2^{\frac{n-m}{n}}\nu =(f_1^{1/n} + f_2^{1/n})^n\nu.
\end{equation*}
If equality occurs, it is necessary from the above inequalities that we have equalities for $m=0$ and $m=n$, so $\mu_{u_i}= f_i\nu$ for $i=1, 2$. The theorem is proved.
\end{proof}

\section{Maximum and comparison principles and the Dirichlet problem}
\label{maxDir_sec}
In this section, we will prove the maximum principles in Theorems  \ref{AJ_thm} and \ref{BlemR}, and apply them to establish solvability to the Dirichlet problem in Theorem \ref{Dist0pn} (i). We will also prove a comparison principle for degenerate Monge--Amp\`ere equations in Theorem \ref{compnu}.
\subsection{Maximum principles} 
\label{maxDir_sec1}
In this subsection, we prove Theorems \ref{BlemR} and \ref{AJ_thm} and state some of their consequences.
\begin{proof}[Proof of Theorem \ref{BlemR}] We first prove 
part (i) using the strategy in the proof of the estimates for the complex Monge--Amp\`ere operator in Blocki \cite[Theorem 2.1]{Bl}. 

We begin by reducing the proof to the case of smoothness and equality near the boundary for $w$ and $v$.  Indeed, 
for any $\e>0$, let $w_\e=\max\{v, w-\e\}$. Then $w_\e$ is convex, $w_\e\geq v$,  $w_\e=v$ in a neighborhood of $\p\Omega$, and $w_\e\to w$ uniformly in $\overline{\Omega}$.
If we can prove part (i) for $w_\e$ replacing $w$; that is,
\begin{equation}
\label{Ble1}
\int_\Omega (w_\e-v)^n \,d\mu_n[u_1, \cdots, u_n] \leq n! \|u_1\|_{L^{\infty}(\Omega)}\cdots \|u_{n-1}\|_{L^{\infty}(\Omega)} \int_\Omega |u_n| \, (d\mu_v-d\mu_{w_\e}),\end{equation}
then part (i) is also true with the original $w$. For this, we just
 let $\e\to 0$ in \eqref{Ble1}, and then
 use Fatou's lemma
 for the left-hand side, while in the right-hand side $\int_\Omega |u_n| \, (d\mu_v-d\mu_{w_\e})$ increases to $\int_\Omega |u_n| \, (d\mu_v-d\mu_w)$; this is because $|u_n|\mu_{w_\e}\to |u_n|\mu_w$ weakly, so
 \[\int_\Omega |u_n|\, d\mu_w \leq \liminf_{\e\to 0} \int_\Omega |u_n|\,d\mu_{w_\e}.\]

For the proof of \eqref{Ble1}, by the locality property of the mixed Monge--Amp\`ere measure (Lemma \ref{mMAloc}), we can shrink $\Omega$ a bit. We can then assume by approximations that all functions are smooth and $w_\e=v$ in a neighborhood of $\p\Omega$. Below, we suppress $\e$ in $w_\e$.

\medskip
We use the convention that repeated indices are summed. Let
\begin{equation*}
\begin{split}
A:=\int_\Omega (w-v)^n \,d\mu_n[u_1, \cdots, u_n] 
&= \int_\Omega (w-v)^n \,\tilde M[u_1, \cdots, u_n]\,dx\quad(\text{due to smoothness}) \\&=\int_\Omega (w-v)^n D_{ij}\big(\tilde M^{ij} [u_2, \cdots, u_n] u_1\big)\, dx.
\end{split}
\end{equation*}
Using integration by parts, 
we have
\begin{equation*}
\begin{split}
A
&=\int_\Omega n(w-v)^{n-1} D_{ij}(w-v) \tilde M^{ij} [u_2, \cdots, u_n] u_1\, dx\\
&\quad + \int_\Omega n(n-1)(w-v)^{n-2} D_i(w-v) D_j(w-v) \tilde M^{ij} [u_2, \cdots, u_n] u_1\, dx.
\end{split}
\end{equation*}
Recalling that \[u_1\leq 0,\quad D_i(w-v) D_j(w-v) \tilde M^{ij} [u_2, \cdots, u_n] \geq 0,\quad D_{ij}w \tilde M^{ij} [u_2, \cdots, u_n]\geq 0, \] where Remark \ref{mMpos} (i) was used,  we deduce 
\begin{equation*}
\begin{split}
A
&\leq \int_\Omega n(-u_1)(w-v)^{n-1} D_{ij}v \tilde M^{ij} [u_2, \cdots, u_n] \, dx\\
&\leq n \|u_1\|_{L^{\infty}(\Omega)}\int_\Omega (w-v)^{n-1} \tilde M [v, u_2,\cdots, u_n]\,dx.
\end{split}
\end{equation*}
Continuing, if $u_1, \cdots, u_{n-1}$ are nonpositive, then we have
\begin{equation}
\label{Awvun}
A \leq n! \|u_1\|_{L^{\infty}(\Omega)}\cdots \|u_{n-1}\|_{L^{\infty}(\Omega)} \int_{\Omega} (w-v)  \tilde M [v, \cdots, v, u_n]\,dx.
\end{equation}
Note that, by the exchange result in Lemma \ref{IBPC3} and $w=v$ in a neighborhood of $\p\Omega$,
\begin{equation}
\label{Amuwv}
\begin{split}
\int_{\Omega} (w-v)  \tilde M [v, \cdots, v, u_n]\,dx&= \int_{\Omega}  \tilde M [v, \cdots, v, w-v] u_n\,dx \\&=  \int_{\Omega}  \tilde M [v, \cdots, v, w] u_n\,dx -\int_\Omega u_n \, d\mu_v\\
&\leq \int_\Omega u_n (\, d\mu_w-\, d\mu_v) =\int_\Omega |u_n| (\, d\mu_v-\, d\mu_w),
\end{split}
\end{equation}
where we use the monotonicity result in Lemma \ref{Exchcon} in the last inequality.

Combining \eqref{Amuwv} with \eqref{Awvun}, we obtain \eqref{Ble1}, completing the proof of part (i) 
of the theorem.

\medskip
\noindent
We now prove part (ii) and use shrinking and the smooth approximations as above. Write
\begin{equation*}
\begin{split}
\mu_n[u_1, \cdots, u_n]-\mu_n[\tilde u_1, \cdots, \tilde u_n]&=\sum_{k=1}^n \tilde M[u_1, \cdots, u_{k-1}, u_k-\tilde u_k, \tilde u_{k+1}, \cdots, \tilde u_n]\\
&=\sum_{k=1}^n \tilde M^{ij}[u_1, \cdots, u_{k-1}, \tilde u_{k+1}, \cdots, \tilde u_n] D_{ij} (u_k-\tilde u_k).
\end{split}
\end{equation*}
For each $k=1, \cdots, n$, let \[d_k:=\|u_k-\tilde u_k\|_{L^{\infty}(\Omega)},\quad \tilde M_k^{ij}:= \tilde M^{ij}[u_1, \cdots, u_{k-1}, \tilde u_{k+1}, \cdots, \tilde u_n].\]
Integrating by parts, and recalling \eqref{div0} together with Remark \ref{mMpos} (i), we find 
\begin{equation}
\label{Ble2}
\small
\begin{split} 
\int_\Omega \frac{(w-v)^{n+1}}{n+1} \tilde M_k^{ij} D_{ij} (u_k-\tilde u_k) \,dx&= \int_\Omega (w-v)^n D_{ij} (w-v) \tilde M_k^{ij} (u_k-\tilde u_k) \,dx \\&\,\, + n\int_\Omega (w-v)^{n-1} D_i(w-v) D_j(w-v) \tilde M_k^{ij} (u_k-\tilde u_k) \,dx\\
&\quad\leq  d_k \int_\Omega (w-v)^n D_{ij} (w+v) \tilde M_k^{ij} \,dx \\&\quad + n d_k\int_\Omega (w-v)^{n-1} D_i(w-v) D_j(w-v) \tilde M_k^{ij} \,dx.
\end{split}
\end{equation}
Observe that
\begin{equation}
\label{Ble3}
\begin{split}n\int_\Omega (w-v)^{n-1} D_i(w-v) D_j(w-v) \tilde M_k^{ij}\,dx&=\int_\Omega D_j[(w-v)^n D_i(w-v)\tilde M_k^{ij}]\,dx\\&\quad -\int_\Omega (w-v)^n D_{ij}(w-v)\tilde M_k^{ij}\,dx.
\end{split}
\end{equation}
The first term on the right-hand side of \eqref{Ble3} vanishes. 
Summing up  \eqref{Ble2} and  \eqref{Ble3} gives
\begin{equation*}
\begin{split} 
\int_\Omega \frac{(w-v)^{n+1}}{n+1} \tilde M_k^{ij} D_{ij} (u_k-\tilde u_k) \,dx &\leq 2d_k \int_\Omega (w-v)^n D_{ij} v \tilde M_k^{ij}\,dx \\
&=2d_k \int_\Omega (w-v)^n \tilde M[u_1, \cdots, u_{k-1}, \tilde u_{k+1}, \cdots, \tilde u_n, v]\,dx\\
&\leq 2n! d_k m_k \int_\Omega |v| (d\mu_v- d\mu_w),
\end{split}
\end{equation*}
where we used part (i) in the last inequality. Therefore
\begin{equation*}
\begin{split} 
\int_\Omega (w-v)^{n+1} \,(d\mu_n[u_1, \cdots, u_n]-d\mu_n[\tilde u_1, \cdots, \tilde u_n])&= \sum_{k=1}^n\int_\Omega (w-v)^{n+1} \tilde M_k^{ij} D_{ij} (u_k-\tilde u_k) \,dx \\ &\leq 2(n+1)!\sum_{k=1}^n d_k m_k \int_\Omega |v| (d\mu_v-d\mu_w).
\end{split}
\end{equation*}
The proof of part (ii) is complete. The theorem is proved.
\end{proof}

\medskip
A consequence of Theorem \ref{BlemR} is the following comparison principle. 
\begin{cor}
\label{mpcor}
Let $\Omega$  be a bounded (not necessarily convex) domain in $\R^n$.
Assume $u, v, w\in C(\overline{\Omega})$ are nonpositive convex functions such that $u=v$ on $\p\{u>v\}$. Then
\[\int_{\{u>v\}} (-w) \, d\mu_u \leq \int_{\{u>v\}} (-w) \, d\mu_v. \]
Consequently, if $0\geq u=v$ on $\p\Omega$ and $0\geq u>v$ in $\Omega$, then 
\[\int_\Omega (-u) \, d\mu_u \leq \int_\Omega (-u) \, d\mu_v\leq \int_\Omega (-v) \, d\mu_v.\]
\end{cor}
\begin{proof} Let $\Omega':=\{u>v\}$.
It suffices to consider $\|w\|_{L^{\infty}(\Omega')}>0$. 
 By Theorem \ref{BlemR} (i), we have
\[0\leq \int_{\Omega'} (u-v)^n \, d\mu_w \leq n! \|w\|^{n-1}_{L^{\infty}(\Omega')} \int_{\Omega'}(-w) \, (d\mu_v-d\mu_u).\]
The claimed inequality follows.
\end{proof}

We are now ready to prove the uniform Aleksandrov--Jerison maximum principle.
\begin{proof}[Proof of Theorem \ref{AJ_thm}] 
Let $w\in C^{0, 1}(\overline{\Omega})$ be the convex function whose graph is the cone with vertex $(x_0, -1)$ and the base $\Omega$, with $w=0$ on $\p\Omega$. Then 
\[\p w(\Omega) =\p w(x_0)\quad\text{and}\quad \mu_w= C_{x_0}\delta_{x_0}.\]
From the Aleksandrov maximum principle in Theorem \ref{Al_E} (i), we have
\[1=|w(x_0)|^{n}\leq C(n)[\diam (\Omega))]^{n-1}\dist(x_0,\partial \Omega) C_{x_0}.\]
Therefore, for some $c(n)>0$, 
\begin{equation}
\label{Cx0est}
\dist(x_0,\p\Omega)C_{x_0} \geq c(n)[\diam(\Omega)]^{1-n}.\end{equation}
Observe that, if $p\in \p w(\Omega)= \p w(x_0)$, then 
$|p| \leq [\dist(x_0,\p\Omega)]^{-1}$. Hence, the Lipschitz constant $L$ of $w$ satisfies
$L \leq [\dist(x_0,\p\Omega)]^{-1}$. 
Therefore, $|w| \leq \frac{\dist(\cdot,\p\Omega)}{\dist(x_0,\p\Omega)}$ in  $\Omega$.

Let
\[v:= -(-w)^{1-\alpha}\quad \text{in }\Omega\quad\text{if}\quad 0\leq \alpha<1,\quad \text{and } v:=w\quad \text{in }\Omega\quad\text{if }  \alpha=1.\]
Then $v\in C(\overline{\Omega})$ is convex in $\Omega$ with $v=0$ on $\p\Omega$. Since $v\leq w$ and $v(x_0)= w(x_0)=-1$, we have $\p w(x_0)\subset \p v(x_0)$. It follows that
\begin{equation}\label{muvest} \mu_v\geq C_{x_0}\delta_{x_0}\quad \text{and}\quad |v| \leq \dist^{\alpha-1}(x_0,\p\Omega) \cdot \dist^{1-\alpha}(\cdot,\p\Omega)\quad \text{in }\Omega.\end{equation}
Since $u=\tilde u=0$ on $\p \Omega$ and $\mu_u\geq\mu_{\tilde u}$ in $\Omega$, the comparison principle in Theorem \ref{compa1} gives $u\leq \tilde u$ in $\Omega$. Applying Theorem \ref{BlemR} (i), we have
\begin{equation*}
\begin{split}\int_\Omega |\tilde u-u|^n C_{x_0}\delta_{x_0}  \leq \int_\Omega |\tilde u-u|^n \,d\mu_v  \leq n! \|v\|^{n-1}_{L^{\infty}(\Omega)} \int_{\Omega} |v| \,(d\mu_u-d\mu_{\tilde u}).
  \end{split}
 \end{equation*}
 It follows from \eqref{Cx0est}, $\|v\|_{L^{\infty}(\Omega)} =1$, $\mu_u\geq\mu_{\tilde u}$ in $\Omega$, and \eqref{muvest} that
 \begin{equation*}
 |\tilde u(x_0)-u(x_0)|^n 
 \le C(n)[\diam (\Omega)]^{n-1}\dist^{\alpha}(x_0,\partial \Omega)\int_{\Omega}\dist^{1-\alpha}(\cdot,\partial \Omega)\,(d\mu_u-d\mu_{\tilde u}).
 \end{equation*}
 The theorem is proved.
\end{proof}

\begin{rem} 
\label{AJ_rem}
Some remarks concerning Theorem \ref{AJ_thm} are in order.
\begin{enumerate}
 \item A consequence of Theorem \ref{AJ_thm} is that $u\in C^{\alpha/n}(\overline{\Omega})$ whenever $u\in C(\overline{\Omega})$ is convex,
 \[u=0\quad\text{on }\p\Omega\quad\text{and}\quad \int_\Omega \dist^{1-\alpha}(\cdot,\p\Omega)d\mu_u<\infty.\] 
\item The exponent $\alpha/n$ is sharp when $n\geq 2$ and $0\leq\alpha\leq 1$. Indeed, fix $\alpha/n<\beta\leq 1$.
Consider $\alpha/n<a<\beta$, and 
 \[\Omega=\{(x', x_n): |x'|<1, 0<x_n < 1- |x'|^2\},\quad w= x_n -x_n^{a} (1-|x'|^2)^{1-a}.\]
Then, $w$ is smooth, convex in $\Omega$ with $w=0$ on $\p\Omega$, and $w\not \in C^{\beta} (\overline{\Omega})$, but \[\dist^{1-\alpha}(\cdot,\p\Omega)\det D^2 w\in L^1(\Omega).\] 
Indeed, computing as in \cite[Lemma 2.4]{LDCDS}, we see that $w$ is smooth and convex in $\Omega$ with $w=0$ on $\p\Omega$, and
\[\det D^2 w=a(1-a)(2-2a)^{n-1} x_n^{na-2}(1-|x'|^2)^{1-na}= C(n, a)x_n^{na-2}(1-|x'|^2)^{1-na}.
\]
For $x=(0', x_n)\in\R^n$, we have $|w(0', x_n)|= x_n^{a}- x_n\geq x_n^{a}/2$
for $x_n$ small, depending only on $n$ and $a$. This shows that $w\not \in C^{\beta} (\overline{\Omega})$. Note that 
\begin{eqnarray*}
\int_{\Omega} \dist^{1-\alpha}(\cdot,\p\Omega)\det D^2 w \,dx&\leq &\int_{|x'|<1}\int_0^{1-|x'|^2} C(n, a) x_n^{na-1-\alpha}(1-|x'|^2)^{1-na} \,dx_n dx'\\
&=& \int_{|x'|<1}\frac{C(n, a)}{na-\alpha} (1-|x'|^2)^{1-\alpha} \,dx'<\infty.
\end{eqnarray*}
Therefore, $\dist^{1-\alpha}(\cdot,\p\Omega)\det D^2 w\in L^1(\Omega)$.
\end{enumerate}
\end{rem}

\subsection{The Dirichlet problem}
\label{Dir_sec}
In this subsection, we establish the solvability of the Dirichlet problem for the Monge--Amp\`ere equation with right-hand side being a Borel measure which could be quite singular near the boundary. We will also discuss its consequences and optimality.

The following theorem is a restatement of Theorem \ref{Dist0pn} (ii) and it is  equivalent to Theorem 1.5 in Lu--Zeriahi \cite{LZ}. 
\begin{thm}[Solvability of the Dirichlet problem] 
\label{Dirdist}
Let $\varphi\in C(\overline{\Omega})$ be a convex function on a bounded convex domain $\Omega\subset\R^n$,
and $\nu$ be a Borel measure on $\Omega$ satisfying
\[\int_\Omega \dist(\cdot,\p\Omega)d\nu\leq C<\infty.\]
Then there exists a unique convex Aleksandrov solution 
$u\in C(\overline{\Omega})$ to
\begin{equation*}
   \det D^{2} u=\nu \quad\text{in} ~\Omega, \quad
u =\varphi\quad\text{on}~\p \Omega.
\end{equation*}
Moreover, we have
\begin{equation}
\label{Dirineq1}
\|(\varphi-u)^+\|^n_{L^{\infty}(\Omega)}  \leq C(n, \diam(\Omega))\int_\Omega \dist(\cdot,\p\Omega) d\nu.\end{equation}
\end{thm}

\begin{proof} For $\e>0$, let $\Omega_\e:= \{x\in\Omega: \dist(x,\p\Omega)>\e\}$ and $\nu_\e =\chi_{\Omega_\e} \nu$.
Then
\[\int_\Omega d\nu_\e =\e^{-1}\int_{\Omega_\e} \e\, d\nu\leq \e^{-1} \int_{\Omega_\e} \dist(\cdot,\p\Omega)d\nu \leq \e^{-1}C.\]

\noindent
{\it Step 1.} The case $\varphi\equiv 0$. By Theorem \ref{Dir_thm}, 
there is a unique convex solution $u_\e\in C(\overline{\Omega})$ to
\begin{equation}
\label{Dir1e}
   \mu_{u_\e}=\nu_\e \quad\text{in} ~\Omega, \quad
u_\e =0\quad\text{on}~\p \Omega.
\end{equation}
Pick $x_0\in\Omega$. By Theorem \ref{AJ_thm}, we have
\begin{equation}
\label{uedist}
|u_\e(x_0)|^n  \leq C(n) [\diam(\Omega)]^{n-1}\int_\Omega \dist(\cdot,\p\Omega) d\nu_\e\leq   C(n) [\diam(\Omega)]^{n-1}.\end{equation}
This gives the uniform bound independent of $\e$:
\[\|u_\e\|_{L^{\infty}(\Omega)}\leq C.\]
We now pass to the limit $\e\to 0$.
Observe that if $\e>\e'>0$, then $\nu_{\e'}> \nu_\e$, so by the comparison principle in Theorem \ref{compa1}, we have $u_{\e'}< u_{\e}$.  By Theorem \ref{AJ_thm}, we can estimate
\begin{equation*}
\begin{split}
 |u_\e(x_0)-u_{\e'}(x_0)|^n &\leq C(n) [\diam(\Omega)]^{n-1}  \int_{\Omega} \dist(\cdot,\p\Omega)(d\nu_{\e'}- d\nu_{\e})\\ 
 &\leq C(n) [\diam(\Omega)]^{n-1}  \int_{\Omega_{\e'}\setminus \Omega_{\e}} \dist(\cdot,\p\Omega)d\nu.
 \end{split}
 \end{equation*}
 Thus, $\{u_\e\}$ is monotone and uniformly Cauchy in $C(\overline{\Omega})$.
From this, we can let $\e\to 0$ in \eqref{Dir1e} to obtain a convex solution 
$u_0\in C(\overline{\Omega})$ to
\begin{equation*}
   \mu_{u_0}=\nu \quad\text{in} ~\Omega, \quad
u_0 =0\quad\text{on}~\p \Omega.
\end{equation*}
Here, we used the fact that $\mu_{u_\e}$ converges to $\mu_{u_0}$ weakly as measures when $\e\to 0$.

\medskip
\noindent
{\it Step 2.} The general case. Let $h\in C(\overline{\Omega})$ be a harmonic function in $\Omega$ with boundary value $\varphi$; see Gilbarg--Trudinger \cite[Theorem 2.14 and the discussion following it]{GT}.
By Theorem \ref{Dir_thm}, 
there exists a unique convex solution $\tilde u_\e\in C(\overline{\Omega})$ to
\begin{equation}
\label{Dir2e}
   \mu_{ \tilde u_\e}=\nu_\e \quad\text{in} ~\Omega, \quad
\tilde u_\e =\varphi\quad\text{on}~\p \Omega.
\end{equation}
Note that (see \cite[Lemma 3.10]{Lbook})
$\mu_{u_\e +\varphi}\geq \mu_{u_\e} +\mu_{\varphi}\geq \mu_{u_\e}=\mu_{\tilde u_\e}$.
By the comparison principle, 
\begin{equation}
\label{uvedist}
u_\e + \varphi\leq \tilde u_\e \leq h\quad\text{in }\Omega.\end{equation}
By letting $\e\to 0$ in \eqref{Dir2e}, and recalling that $u_\e$ converges uniformly to $u_0\in C(\overline{\Omega})$ with boundary value $0$, we obtain a convex Aleksandrov solution 
$u\in C(\overline{\Omega})$ solving
\begin{equation*}
   \det D^{2} u=\nu \quad\text{in} ~\Omega, \quad
u =\varphi\quad\text{on}~\p \Omega.
\end{equation*}
Moreover, from \eqref{uedist} and \eqref{uvedist}, we obtain \eqref{Dirineq1} as claimed.
The theorem is proved.
\end{proof}
We make several remarks pertinent to Theorem \ref{Dirdist}.
\begin{rem} 
\label{criticalrem}
Let $n\geq 2$ and $\Omega=\{(x', x_n): |x'|<1, 0<x_n < 1- |x'|^2\}$. The exponent $1$ of $\dist(\cdot, \p\Omega)$ in the condition on the measure $\nu$ in Theorem \ref{Dirdist} is critical in the following sense:
\begin{enumerate}
\item For  $\beta \in (0, 1)$, there exists a convex function $w\in C(\overline{\Omega})$ with $w=0$ on $\p\Omega$ such that
\[\int_\Omega \dist^\beta(\cdot, \p\Omega) \, d\mu_w=+\infty.\]
Indeed, it suffices to choose
 \[w(x)=x_n -x_n^a (1-|x'|^2)^{1-a},\quad a= (1-\beta)/n.\]
Computing as in \cite[Lemma 2.4]{LDCDS}, we have
\[\det D^2 w=a(1-a)(2-2a)^{n-1} x_n^{na-2}(1-|x'|^2)^{1-na}= C(n,\beta) x_n^{-1-\beta} (1-|x'|^2)^\beta.
\]
Thus, 
\begin{eqnarray*}
\int_\Omega \dist^\beta(\cdot, \p\Omega) \, d\mu_w&\geq& 
\int_{\{(x', x_n):\, |x'|<1/4, \,0<x_n < 1/4\}} \dist^\beta(\cdot, \p\Omega) \,  \det D^2 w \,dx\\&=&C(n,\beta) \Big( \int_{|x'|<1/4}  (1-|x'|^2)^{\beta}  dx'\Big) \Big (\int_0^{1/4}x_n^{-1}\,dx_n\Big)=+\infty.
\end{eqnarray*} 
\item For $n\geq 2$, and $\gamma \in (1, 2)$, there is 
a Borel measure $\nu$ such that 
\[\int_\Omega \dist^\gamma(\cdot, \p\Omega) \, d\nu<\infty,\,\, \nu\neq \mu_u\,\,\mbox{for any convex function $u\in C(\overline{\Omega})$ with $u=0$ on $\p\Omega$.}\]
Indeed, let $\nu= \dist^{-(3+\gamma)/2}(\cdot, \p\Omega)$ in a neighborhood of $\p\Omega$ and $\nu=0$ elsewhere. To see that $\nu\neq \mu_u$ for any convex function $u\in C(\overline{\Omega})$ with $u=0$ on $\p\Omega$, it suffices to choose any $\alpha\in (1, (1+\gamma)/2)$ to contradict the following statement:

If $u\in C(\overline{\Omega})$ is a convex function with $u=0$ on $\p\Omega$, then we can show that 
\begin{equation}
\label{namu}
\int_{\{(x', x_n):\, |x'|<1/4, \,0<x_n < 1/4\}} \dist^{\alpha}(\cdot, \p\Omega) \, d\mu_u<\infty\quad\text{for all }\alpha\in (1, 2).
\end{equation}
To see this, consider the convex function
 \[v(x)=x_n -x_n^a (1-|x'|^2)^{1-a},\quad a=\alpha/n,\]
 which vanishes on $\p\Omega$.
Computing as in \cite[Lemma 2.4]{LDCDS}, we have
\[\det D^2 v=a(1-a)(2-2a)^{n-1} x_n^{na-2}(1-|x'|^2)^{1-na}.
\]
Therefore $\det D^2 v\in L^1(\Omega)$.
By Theorem \ref{BlemR}, we have
\[\int_\Omega |v|^n \, d\mu_u \leq n! \|u\|^n_{L^{\infty}(\Omega)}\int_\Omega d\mu_v<\infty.\]
Now, \eqref{namu} follows from the above inequality and 
 \[|v(x)|\geq c(n,\alpha)\dist^{\alpha/n}(\cdot, \p\Omega)\quad\text{in } \{(x', x_n):\, |x'|<1/4, \,0<x_n < 1/4\}.\]
\end{enumerate}
\end{rem}
\begin{rem} 
\label{criticalrem2}
The proof of \eqref{namu} indicates that in order to replace the weight $\dist^{\alpha}(\cdot, \p\Omega)$ by a larger one so that the left-hand side is still finite, we need to construct a convex function $v\in C(\overline{\Omega})$ satisfying $v=0$ on $\p\Omega$ with 
$\mu_v(\Omega)<\infty$ 
 such that
$|v|$ grows faster than $\dist^{\alpha/n}(\cdot, \p\Omega)$ from the flat boundary part. It would be ideal to replace $\alpha>1$ by $1$. When $n=1$, this just follows from the convexity of $v$. 
It is not clear if this is possible when $n\geq 2$. 
\begin{enumerate}
\item[(i)] However, using a construction in the proof of Lemma 3.2 in Jerrard \cite{Jd}
for $n\geq 2$ and $\Omega=\{(x', x_n): |x'|<1, 0<x_n < 1- |x'|\}$,
we can replace
$\dist^{\alpha/n}(\cdot, \p\Omega)$ by $\frac{\dist^{1/n}(\cdot, \p\Omega)}{|\log \dist(\cdot, \p\Omega)|^{\alpha/n}}$ for any $\alpha>1$. This improvement corresponds to choosing
\[f(s) = -s^{1/n}/|\log s|^{\alpha/n}\quad\text{for } 0<s\leq 1/2,\]
after equation (3.8) in \cite{Jd}. Jerrard chose $\alpha=n>1$, but his proof just needs the finiteness of $\int_0^{1/s} [f(s)]^n\frac{1}{s^2}\, ds$, as indicated at the end of his proof. Clearly, the above choice of $f$ with any $\alpha>1$ satisfies this finiteness condition. In particular, the above discussion implies the following statements.
\item[(ii)] If $\Omega$ is a bounded interval in $\R$ and $u\in C(\overline{\Omega})$ is a convex function with $u=0$ on $\p\Omega$, then 
$\int_\Omega  \dist(\cdot, \p\Omega) \, d\mu_u<\infty$.

If $n\geq 2, \Omega=\{(x', x_n): |x'|<1, 0<x_n < 1- |x'|\}$, 
and $u\in C(\overline{\Omega})$ is a convex function with $u=0$ on $\p\Omega$, then 
\begin{equation}
\label{namuJ}
\int_{\{(x', x_n):\, |x'|<1/4, \,0<x_n < 1/4\}} \frac{\dist(\cdot, \p\Omega)}{|\log \dist(\cdot, \p\Omega)|^\alpha} \, d\mu_u<\infty\quad\text{for all }\alpha>1.
\end{equation}
\end{enumerate}
\end{rem}
\begin{rem}
\label{lackrem}
We give an example concerning the lack of compactness for convex solutions $u_m\in C(\overline{\Omega})$ to 
\begin{equation*}
   \det D^{2} u_m=\nu_m \quad\text{in} ~\Omega, \quad
u_m =0\quad\text{on}~\p \Omega,    
\end{equation*}
where $\nu_m$ are Borel measures satisfying
$\int_\Omega \dist(\cdot,\p\Omega)d\nu_m\leq C<\infty$. 
\begin{enumerate}
\item Consider for each positive integer $m$
\[\Omega=(-1, 1)\subset\R,\quad u_m(x)= |x|^{2m}-1,\quad \nu_m= 2m(2m-1) x^{2m-2}\,d\L^1.\]
Then
\begin{equation*}
\int_{-1}^1 \dist(x, \p \Omega) u_m''(x)\, dx\leq \int_{-1}^1 (1-|x|^2) 2m(2m-1) x^{2m-2}\,dx= \frac{8m}{2m+1}\leq 4.
\end{equation*}
We note that $\nu_m$ converges weakly to $\nu=0$, $u_m$ converges to $-1$ uniformly on compact subsets of $(-1, 1)$ while $u_m(\pm 1)=0$, so $u_m$ does not converges to the zero function which is the solution to the Dirichlet problem on $(-1,1)$ with the measure $\nu$.

\item The difference between $\nu_m$ and the measures in the proof of Theorem \ref{Dirdist} is that:  $\nu_m$ is not monotone.
\end{enumerate} 
\end{rem}
We identify a class of measures satisfying the hypothesis of Theorem \ref{Dirdist}.
\begin{lem}
\label{Distex}
Let $\Omega\subset\R^n$ be a bounded convex domain, $u\in \E(\Omega)$, and $v\in C(\overline{\Omega})$ be a nonpositive convex function. Let $E$ be Monge--Amp\`ere energy as in \eqref{Edef}. Then 
\begin{equation}
\label{Direxa1}
\int_\Omega  \dist(\cdot,\p\Omega) |u|^n d\mu_v\leq C(n,\Omega)\|v\|^n_{L^{\infty}(\Omega)} [E(u)]^{\frac{n}{n+1}}.\end{equation}
In this case, there exists a unique convex Aleksandrov solution 
$w\in \E(\Omega)$ to
\begin{equation}
\label{DirE1}
   \mu_w~=|u|^n \mu_v \h~\text{in} ~\Omega, \quad
w =0\h~\text{on}~\p \Omega.
\end{equation}
with
\begin{equation} 
\label{Dirineq2}
E(w) \leq C(n) \|v\|^{n+1}_{L^{\infty}(\Omega)} E(u).\end{equation}
\end{lem}
The last statement is a real version of a related result of Cegrell \cite[Theorem 5.1]{Ceg1}.  
\begin{proof} Fix $x_0\in\Omega$. Let $w_0\in C^{0, 1}(\overline{\Omega})$ be the convex function whose graph is the cone with vertex $(x_0, -1)$ and the base $\Omega$, with $w_0=0$ on $\p\Omega$. Then
$w_0\in\E(\Omega)\cap C^{0,1}(\overline{\Omega})$ and
\begin{equation} 
\label{Dirineq3}\frac{\dist(x,\p\Omega)}{\diam(\Omega)}\leq |w_0(x)| \leq \frac{\dist(x,\p\Omega)}{\dist(x_0,\p\Omega)}\quad\text{for all } x\in\Omega.\end{equation}
By the H\"older inequality and Theorem \ref{BlemR} (ii),
\begin{equation}
\label{Dirineq4}
\begin{split}\int_\Omega  |w_0||u|^n d\mu_v &\leq \Big(\int_\Omega  |w_0|^{n+1} d\mu_v\Big)^{\frac{1}{n+1}} \Big(\int_\Omega  |u|^{n+1} d\mu_v\Big)^{\frac{n}{n+1}}\\
& \leq C(n)\|v\|^n_{L^{\infty}(\Omega)} [E(w_0)]^{\frac{1}{n+1}} [E(u)]^{\frac{n}{n+1}}\\
& \leq C(n,\Omega, x_0)\|v\|^n_{L^{\infty}(\Omega)} [E(u)]^{\frac{n}{n+1}}.
\end{split}
\end{equation}
The estimate \eqref{Direxa1} follows from \eqref{Dirineq3} and \eqref{Dirineq4}.

\medskip
Now, we consider the Dirichlet problem \eqref{DirE1}. 
Let $v_m\in C(\overline{\Omega})$ be standard approximants of $v$; see Remark \ref{Stappro}. So $0\geq v_m\geq v $ in $ \Omega$ and 
the  convex function $v_m$ solves
\begin{equation*}
   \mu_{v_m}=\chi_{\{x\in\Omega: \dist(x, \p\Omega)>1/m\}}\mu_v \quad\text{in} ~\Omega, \quad
v_m =v\quad\text{on}~\p \Omega.
\end{equation*}
Moreover, by Theorem \ref{Dir_thm}, there exists a unique convex function $w_m\in C(\overline{\Omega})$ satisfying
\begin{equation*}
   \mu_{w_m}=|u|^n\mu_{v_m} \quad\text{in} ~\Omega, \quad
w_m=0\quad\text{on}~\p \Omega.
\end{equation*}
 Clearly $w_m\in \E(\Omega)$. The proof of Theorem \ref{Dirdist} shows that $w_m$ converges uniformly  in $\overline{\Omega}$ to a convex function $w\in C(\overline{\Omega})$ solving \eqref{DirE1}.  We compute as above
\[E(w_{m}) =\int_\Omega |w_m| |u|^n\, d\mu_{v_m} \leq C(n)\|v_m\|^n_{L^{\infty}(\Omega)} [E(w_m)]^{\frac{1}{n+1}} [E(u)]^{\frac{n}{n+1}}.\]
It follows that
\[E(w_m)\leq C(n) \|v_m\|^{n+1}_{L^{\infty}(\Omega)} E(u) \leq C(n) \|v\|^{n+1}_{L^{\infty}(\Omega)} E(u).\]
Now, \eqref{Dirineq2} follows from the lower-semicontinuity of the energy $E$ (see Proposition \ref{I_lsc2}). 
\end{proof}
Now next consider the Dirichlet problem where the measure involved satisfies a Poincar\'e-type inequality in $\E(\Omega)$. 
\begin{lem} \label{Distex2}
Let $\Omega\subset\R^n$ be a bounded convex domain. Let $\nu$ be a Borel measure with $\nu(\Omega)>0$ such that there is a constant $\lambda_0>0$ satisfying
\[\int_\Omega |v|^{n+1}\, d\nu\leq \lambda_0^{-1} E(v)\quad \text{for all } v\in \E(\Omega).\]
Let $u\in\E(\Omega)$.
Then, there exists a unique convex Aleksandrov solution 
$w\in \E(\Omega)$ to
\begin{equation*}
   \mu_w~=|u|^n \nu \h~\text{in} ~\Omega, \quad
w =0\h~\text{on}~\p \Omega.
\end{equation*}
with the Monge--Amp\`ere energy (see \eqref{Edef})
$E(w) \leq \lambda_0^{-1/n} \|u\|^{n+1}_{L^{n+1}(\Omega,d\nu)}$.
\end{lem}
\begin{proof} The proof is the same as that of Lemma \ref{Distex} with  $\lambda_0^{-1}$ playing the role of the quantity $C(n) \|v\|^n_{L^{\infty}(\Omega)}$ there.
\end{proof}

It is worth recording that mixed Monge--Amp\`ere measures are also Monge--Amp\`ere measures. This shows that mixed measures, though very useful,  are not more general than unmixed ones.
\begin{lem} Let $\Omega\subset\R^n$ be a bounded convex domain. Let $u_1,\cdots, u_n\in C(\overline{\Omega})$ be convex functions. Then there exists a convex function $u\in C(\overline{\Omega})$ such that \[\mu_u =\mu_n[u_1, \cdots, u_n].\]
\end{lem}
\begin{proof} Let $w:=u_1 + \cdots + u_n$.
The measure $\nu:=\mu_n[u_1, \cdots, u_n]$ has finite local masses in $\Omega$ and satisfies $\nu\leq \mu_w$.
By Theorem \ref{Dir_thm}, for each $m\in \N$,
there exists a unique convex function $u_m\in C(\overline{\Omega})$ satisfying
\begin{equation*}
   \mu_{u_m}=\chi_{\{x\in\Omega: \dist(x, \p\Omega)>1/m\}} \nu \quad \text{in} ~\Omega, \quad
u_m =w\quad\text{on}~\p \Omega.
\end{equation*}
Let $h\in C(\overline{\Omega})$ be a harmonic function in $\Omega$ with boundary value $w$.
By the comparison principle, we have $h\geq u_m\geq w$ in $\Omega$. Letting $m\to\infty$, we obtain a convex 
function $u\in C(\overline{\Omega})$ such that $\mu_u=\nu$ in $\Omega$ with boundary value $w$.
\end{proof}

 \subsection{A comparison principle for degenerate Monge--Amp\`ere equations}
 Using maximum principle techniques and rescaling, we prove a comparison principle for degenerate Monge--Amp\`ere equations involving general Borel measures.
\label{com_sec}
\begin{proof}[Proof of Theorem \ref{compnu}] It suffices to prove $u\geq v$ in $\Omega$. We consider several cases.

\medskip
\noindent
{\bf Case 1: $\nu=f \, d\L^n$ where $f\in C(\overline{\Omega})$ with $f\geq c>0$ in $\overline\Omega$.}

  Suppose otherwise that  $v-u$ is positive somewhere in $\Omega$.  Note that $v-u\leq 0$ on $\p\Omega$.
By changing coordinates, we can assume that $0\in\Omega$ satisfies
$v(0)- u(0)=\tau>0.$
Since $v<0$ in $\Omega$, we have  \[ u(0)/v(0) =[v(0)-\tau]/v(0) = 1+\e\quad \text{for some }\e>0.\]
From the uniform continuity of $f\geq c>0$ in $\overline{\Omega}$, we can find $\gamma_0\in (1, 2)$ such that 
\begin{equation}
\label{fga1}
f(\gamma y)\leq (1+\e)^{\frac{n-p}{2}} f(y)\quad \text{whenever } \gamma\in (1,\gamma_0],\, y\in \Omega/\gamma:=\{x/\gamma: x\in\Omega\}.\end{equation}
Let us choose \[\gamma:= \Big[\min\Big\{\frac{(1+\e)^{\frac{n-p}{2}} +1}{2}, \gamma_0\Big\}\Big]^{\frac{1}{2n}}\in (1, 2).\] 
Consider for $x\in\Omega$, \[v_{\gamma}(x)= v(x/\gamma),\quad
\text{and }\eta_{\gamma}(x) = u(x)/v_{\gamma}(x).\] If $\dist(x,\p\Omega)\rightarrow 0$, then $\limsup u(x)\geq 0$, so $\liminf \eta_{\gamma}(x)\leq 0$. Recall that $\eta_\gamma(0)= 1+ \e$. 
Therefore, $\eta_{\gamma}$ attains its maximum value in $\overline{\Omega}$ at $x_{\gamma}\in\Omega$ with $\eta_{\gamma}(x_{\gamma})=:a\geq 1+\e$.
We have
\begin{equation}
\label{uvga1}
\frac{|u(x)|}{|v_\gamma(x)|} = \frac{-u(x)}{-v_\gamma(x)} =\eta_\gamma(x)\leq a\quad\text{in }\Omega. \end{equation}
From our assumption that $v<0$ in $\Omega$, we find
$m:= \min_{\p\Omega/\gamma}|v|/4>0$.

Let
\[\delta:= m/[\diam(\Omega)]^2,\quad
w(x):= u(x)- av_\gamma(x) -\delta|x-x_\gamma|^2\equiv  u(x)+ a|v_\gamma(x)| -\delta|x-x_\gamma|^2.\]
Then
$w(x_\gamma)=0$,
while 
\[\min_{\p\Omega} w \geq a\min_{\p\Omega}|v_\gamma| -\delta [\diam(\Omega)]^2 =a\min_{\p\Omega/\gamma}|v| -\delta [\diam(\Omega)]^2\geq 4m-m>2m.   \]
Let
$S:= \{x\in\overline{\Omega}: w(x)<m\}$.
Then $S$ is an open subset of $\overline{\Omega}$ with positive Lebesgue measure and $w=m$ on $\p S$. The function $u$ is below $av_\gamma+ \delta |x-x_\gamma|^2 + m $ in $S$, but they coincide on $\p S$. 
Consequently, the maximum principle in \cite[Lemma 3.11]{Lbook} gives
\[\p u(S)\supset \p (av_\gamma+ \delta |x-x_\gamma|^2 + m) (S)= \p (av_\gamma+ \delta |x-x_\gamma|^2) (S).\]
It follows from $\mu_u(S)=|\p u(S)| $ and \cite[Lemma 3.10]{Lbook} that
\begin{equation}
\label{muuvS1}
\begin{split}
\mu_u(S)\geq |\p (av_\gamma+ \delta |x-x_\gamma|^2) (S)| &\geq |\p [av_\gamma](S)| + |\p (\delta |x-x_\gamma|^2) (S)|\\&= |\p [av_\gamma](S)| + (2\delta)^n |S|> a^n|\p v_\gamma(S)|.
\end{split}
\end{equation}
We compute from $\mu_u\leq |u|^p fd\L^n$ with the help of \eqref{uvga1} and \eqref{fga1}
\begin{equation*}
\begin{split}
\mu_u(S) \leq \int_S |u|^pf\, dx\leq a^p  \int_S |v_\gamma(x)|^p f(x)\, dx&= a^p \gamma^n \int_{S/\gamma} |v(y)|^p f(\gamma y)\, dy\\&\leq a^p(1+\e)^{\frac{n-p}{2}} \gamma^n \int_{S/\gamma} |v(y)|^p f( y)\, dy. 
\end{split}
\end{equation*}
Since $\p v_\gamma(x) = \gamma^{-1} \p v (x/\gamma)$, we have from $\mu_v\geq |v|^p fd\L^n$
\[ |\p v_\gamma(S)|=\gamma^{-n} |\p v (S/\gamma)|\geq \gamma^{-n}\int_{S/\gamma}|v(y)|^p f(y)\, dy.
  \]
Putting all these together, we find from $\mu_u(S)> a^n|\p v_\gamma(S)|$ that
\[a^p(1+\e)^{\frac{n-p}{2}} \gamma^n \int_{S/\gamma} |v(y)|^p f( y)\, dy>
a^n\gamma^{-n}\int_{S/\gamma}|v(y)|^p f(y)\, dy.\]
This is impossible since 
$\gamma^{2n}(1+\e)^{\frac{n-p}{2}} <(1+\e)^{n-p}\leq a^{n-p}$.
Therefore, $u\geq v$ in $\Omega$.

\medskip
\noindent
{\bf Case 2: $\nu$ is a Borel measure supported in  $\Omega'\Subset\Omega$.}

We first localize the problem. Observe that if $u\geq v$ in $\Omega'$, then from $u\leq 0$ in $\Omega$, we have \[\mu_u \leq |u|^p d\nu \leq |v|^p d\nu \leq \mu_v\quad \text{in }\Omega,\] so the standard comparison principle in Theorem \ref{compa1} gives 
$u\geq v$ in $\Omega$. 

Suppose now  there exists a point, say $0\in \Omega'$, such that 
$u(0)< v(0)<0$.
From the assumption $v<0$ in $\Omega$, we can find a constant $\alpha>0$ such that $|v|\geq \alpha$ in $\Omega'$.
From $n>p$, we can find constants $c>1$ and $m>0$ such that
\[c^n\geq (c+4m/\alpha)^p, \quad u(0)< cv(0)-4m.\]
Let
\[\tilde v:= cv-4m.\]
Then \[\tilde v\leq -4m<0 \quad\text{on }\p\Omega,\quad\text{and}\quad 
 c^n|v|^p\geq (c|v|+4m)^p=|\tilde v|^p\quad \text{in }\Omega'.\]
It follows that
\[\mu_{\tilde v}=c^n\mu_v \geq c^n|v|^p\nu \geq |\tilde v|^p \nu \quad \text{in }\Omega.\]
The above inequalities clearly hold in $\Omega'$ but since $\nu$ is supported in  $\Omega'$, they also hold in $\Omega$.
From $u(0)<\tilde v(0)<0$, we have  $u(0)/\tilde v(0)  = 1+\e$ for some $\e>0$.
Let $\eta = u/\tilde v$.
Since $\eta(0)= 1+ \e$ and $\eta=0$ on $\p\Omega$, $\eta$ attains its maximum value in $\overline{\Omega}$ at $x_{0}\in\Omega$ with $\eta(x_0)=:a\geq 1+\e$.
We have
\begin{equation}
\label{uvga2}
\frac{|u(x)|}{|\tilde v(x)|} = \frac{-u(x)}{-\tilde v(x)} =\eta(x)\leq a\quad\text{in }\Omega. \end{equation}
Let
\[\delta:= m/[\diam(\Omega)]^2,\quad
w(x):= u(x)- a\tilde v(x) -\delta|x-x_0|^2\equiv  u(x)+ a|\tilde v(x)| -\delta|x-x_0|^2.\]
Then
$w(x_0)=0$,
while $\min_{\p\Omega} w \geq a\min_{\p\Omega}|\tilde v| -\delta [\diam(\Omega)]^2 \geq 4m-m>2m. $

Let $S:= \{x\in\overline{\Omega}: w(x)<m\}$.
Then $S$ is an open subset of $\overline{\Omega}$ with positive Lebesgue measure and $w=m$ on $\p S$.
Arguing as in \eqref{muuvS1}, we have
$\mu_u(S)> a^n|\p \tilde v(S)|$.

We compute with the help of \eqref{uvga2}
\[\mu_u(S) \leq \int_S |u|^p\, d\nu\leq a^p  \int_S |\tilde v|^p\, d\nu, \quad
\text{and } 
 |\p \tilde v(S)|\geq \int_{S }|\tilde v|^p\, d\nu. \]
Putting all these together, we find 
\[a^p  \int_S |\tilde v|^p\, d\nu >a^n\int_{S}|\tilde v|^p\, d\nu.\]
This is impossible since $a>1$ and $p<n$.
Therefore, $u\geq v$ in $\Omega$. The theorem is proved.
\end{proof}

\section{Mixed Monge--Amp\`ere energies}
\label{pol_sec}
In this section, we will prove the generalized Cauchy--Schwarz inequality and the integration by parts formula in Theorem  \ref{gCSIBP}.
They are restated as Theorems \ref{gCS} and \ref{IBPf}, respectively.

\subsection{Generalized Cauchy--Schwarz inequality}
For the proof of the generalized Cauchy--Schwarz inequality, we need a couple of preparatory results.
The first result is a polarized version of \cite[Proposition 11.1(ii)]{Lbook}.
\begin{lem}
 \label{I_lsc1}
 Let $\Omega$ be a bounded convex domain in $\R^n$.
 Let $\{u_{0, k}\}_{k=1}^{\infty}$, $\{u_{1, k}\}_{k=1}^{\infty}, \cdots$, $\{u_{n, k}\}_{k=1}^{\infty}$ be sequences of nonpositive functions in $C(\overline\Omega)$
 converging uniformly on $\overline\Omega$ to
  $u_0$, $u_1, \cdots, u_n\in C(\overline{\Omega})$, respectively. 
 Furthermore, assume that  $u_{i, k}$ and $u_i$ ($1\leq i\leq n, \, k\geq 1$) are convex with
 \begin{equation}
\label{massbLk}
\sup_{k\geq 1}\mu_n[u_{1, k},\cdots, u_{n, k}] (\Omega)\leq C<\infty.\end{equation}
 This is satisfied if  $u_{i, k}$ have Lipschitz constants bounded from above by $L<\infty$.
 Then 
  \[\liminf\limits_{k \to \infty} \int_\Omega -u_{0, k}\, d \mu_n[u_{1, k},\cdots, u_{n, k}] \geq \int_\Omega -u_0\, d\mu_n[u_1,\cdots, u_n].\]
\end{lem}
\begin{proof}
We first show, without appealing to \eqref{massbLk}, that for all $\varphi \in C_c(\Omega)$, 
\begin{equation}\label{Icont1}
\lim_{k \to \infty} \int_{\Omega} \varphi u_{0, k}\, d \mu_n[u_{1, k},\cdots, u_{n, k}]  = \int_{\Omega} \varphi u_0\, d\mu_n[u_1,\cdots, u_n].
\end{equation}
Indeed, write 
\begin{equation*}
\begin{split}
&\int_{\Omega} \varphi u_{0, k}\, d \mu_n[u_{1, k},\cdots, u_{n, k}]  - \int_{\Omega} \varphi u_0\, d\mu_n[u_1,\cdots, u_n] \\
&=  \int_{\Omega} \varphi u_{0}\, d \mu_n[u_{1, k},\cdots, u_{n, k}]  - \int_{\Omega} \varphi u_0\, d\mu_n[u_1,\cdots, u_n]- \int_{\Omega} \varphi(u_0- u_{0, k}) d \mu_n[u_{1, k},\cdots, u_{n, k}].
\end{split}
 \end{equation*}
If $\varphi \in C_c(\Omega)$, then the uniform convergences of $u_{i, k}$ to $u_i$, 
and  Definition \ref{mMAdef} give
 $$ \int_{\Omega} \varphi u_{0}\, d \mu_n[u_{1, k},\cdots, u_{n, k}]  - \int_{\Omega} \varphi u_0\, d\mu_n[u_1,\cdots, u_n]
\rightarrow 0~\quad\text{as } k\rightarrow \infty.$$
Moreover, $u_{i, k}$'s are uniformly Lipschitz on the support $\text{spt }\varphi$ of $\varphi$, so by Remark \ref{MLiptot}, 
\[\sup_{k\geq 1}\mu_n[u_{1, k},\cdots, u_{n, k}] (\text{spt }\varphi))\leq C(\varphi)<\infty,\]
from which we find
\[\big| \int_{\Omega} \varphi (u_0- u_{0, k}) d \mu_n[u_{1, k},\cdots, u_{n, k}] \big|\leq  \|\varphi(u_0 - u_{0, k})\|_{L^{\infty}(\Omega)} \mu_n[u_{1, k},\cdots, u_{n, k}] (\text{spt }\varphi)\to 0 
\]
as  $k\rightarrow \infty.$ Therefore,  (\ref{Icont1}) holds. 

\medskip
On the other hand, the uniform convergence of $u_{i, k}$ to $u_i$ implies that $\mu_n[u_{1, k},\cdots, u_{n, k}]$ converges weakly to $\mu_n[u_{1},\cdots, u_{n}]$. Hence,  the hypothesis  \eqref{massbLk} implies
\begin{equation}
\label{massbL1}
\mu_n[u_{1},\cdots, u_{n}] (\Omega)\leq C.\end{equation}
Let us denote
\[A_k:=\int_\Omega -u_{0, k}\, d \mu_n[u_{1, k},\cdots, u_{n, k}] + \int_\Omega u_0\, d\mu_n[u_1,\cdots, u_n].\]
 Let us fix $\e > 0$ and let $\Omega_\e:=\{x\in\Omega:\dist(x, \p\Omega)>\e\}$.
 Let $\varphi_{\e} \in C_c(\Omega)$ be such that $0 \leq \varphi_{\e} \leq 1$ in $\Omega$ and $\varphi_{\e} \equiv 1$ on $\overline{\Omega_{\e}}$. Then, for any $k \geq 1$, we have
\begin{equation*}
\begin{split}
A_k  &= \int_{\Omega} \varphi_{\e} u_0\, d\mu_n[u_1,\cdots, u_n]  - \int_{\Omega} \varphi_{\e} u_{0, k}\, d \mu_n[u_{1, k},\cdots, u_{n, k}]  \\ &\quad+ \int_{\Omega  \setminus \overline{\Omega_{\e}}} (1 - \varphi_{\e}) u_0\, d\mu_n[u_1,\cdots, u_n] + \int_{\Omega  \setminus \overline{\Omega_{\e}}} (\varphi_{\e}-1) u_{0, k}\, d \mu_n[u_{1, k},\cdots, u_{n, k}].
\end{split}
\end{equation*}
Note that the last term is nonnegative. Thus, invoking (\ref{Icont1}), we find
\begin{equation*}
\begin{split}
\liminf_{k\to\infty} A_k \geq \int_{\Omega  \setminus \overline{\Omega_{\e}}} (1 - \varphi_{\e}) u_0\, d\mu_n[u_1,\cdots, u_n].
\end{split}
\end{equation*}
Upon letting $\e\to 0$ and using the Dominated Convergence Theorem while recalling \eqref{massbL1}, we obtain the lower semicontinuity inequality of the lemma.
\end{proof}
The following lower semicontinuity result is a polarized version of \cite[Proposition 11.2]{Lbook}.
\begin{prop}[Lower semicontinuity of mixed Monge--Amp\`ere energy]
 \label{I_lsc2}
 Let $\Omega$ be a bounded convex domain in $\R^n$.
 Let $\{u_{0, k}\}_{k=1}^{\infty}$, $\{u_{1, k}\}_{k=1}^{\infty}, \cdots$, $\{u_{n, k}\}_{k=1}^{\infty}$ be uniformly bounded sequences of nonpositive convex functions in $C(\Omega)$
 converging uniformly on compact subsets of $\Omega$ to
  $u_0, u_1, \cdots, u_n\in C(\Omega)$, respectively. 
Then 
  \[\liminf\limits_{k \to \infty} \int_\Omega -u_{0, k}\, d \mu_n[u_{1, k},\cdots, u_{n, k}] \geq \int_\Omega -u_0\, d\mu_n[u_1,\cdots, u_n].\]
\end{prop}
\begin{proof}
Let $K$ be any compact subset of $\Omega$. Choose an open convex set $U$ such that $K\subset U\subset \overline{U}\subset\subset \Omega$. Let $M:= \sup_{i, k} \|u_{i, k}\|_{L^{\infty}(\Omega)}<\infty$. Note that, on $U$, each $u_{i, k}$ is Lipschitz with Lipschitz constant bounded from above by $L:= \frac{2M}{\dist (\overline{U},\p \Omega)}$. By Lemma \ref{I_lsc1}, we have
\begin{equation*}
 \int_U -u_0\, d\mu_n[u_1,\cdots, u_n]
 \leq \liminf\limits_{k \to \infty} \int_U -u_{0, k}\, d \mu_n[u_{1, k},\cdots, u_{n, k}].
\end{equation*}
It follows that
 \[\int_K -u_0\, d\mu_n[u_1,\cdots, u_n] \leq \liminf\limits_{k \to \infty} \int_\Omega -u_{0, k}\, d \mu_n[u_{1, k},\cdots, u_{n, k}].\]
Taking the supremum of the left-hand side over $K$ gives the claimed lower-semicontinuity property. 
\end{proof}
The proof of Proposition \ref{I_lsc2} in combination with Egorov's theorem gives the next result.
\begin{prop}
 \label{I_lsc3}
 Let $\Omega$ be a bounded convex domain in $\R^n$.
 Let $\{u_{0, k}\}_{k=1}^{\infty}$, $\{u_{1, k}\}_{k=1}^{\infty}, \cdots$, $\{u_{n, k}\}_{k=1}^{\infty}$ be uniformly bounded sequences of nonpositive convex functions in $C(\Omega)$
 converging uniformly on compact subsets of $\Omega$ to
  $u_0, u_1, \cdots, u_n\in C(\Omega)$, respectively. 
Let $\{v_k\}_{k=1}^{\infty}$ be a sequence of nonnegative,  locally uniformly bounded functions that converges almost everywhere to $v\in L_{\text{loc}}^{\infty}(\Omega)$. 
Then 
  \[\liminf\limits_{k \to \infty} \int_\Omega -u_{0, k} v_k\, d \mu_n[u_{1, k},\cdots, u_{n, k}] \geq \int_\Omega -u_0 v\, d\mu_n[u_1,\cdots, u_n].\]
\end{prop}

We recall \cite[Lemma 11.8]{Lbook} which is based on Lemma 3.1 in Hartenstine \cite{Har1}.
\begin{lem}
\label{lemH}
Let $\Omega$ be a bounded convex domain and let $u\in \E(\Omega)$. Then for any $\e>0$, there exists 
 a convex function $v\in C(\overline{\Omega})\cap C^2(\Omega)$ with $v=0$ on $\p\Omega$ and $D^2 v>0$ in $\Omega$ such that $$\|v-u\|_{L^{\infty}(\Omega)}+ |E(v)-E(u)|<\e.$$
 If $\p \Omega$ is uniformly convex and $C^{4,\alpha}$ for some $\alpha\in (0, 1)$, we can choose $v\in C^{4,\alpha}(\overline{\Omega})$ with $D^2 v>0$ in $\overline{\Omega}$. Here, $E$ is the Monge--Amp\`ere energy defined in \eqref{Edef}.
\end{lem}
The following result on abstract Cauchy--Schwarz inequalities is from Persson \cite[Theorem 4.1]{Per} whose proof is included for the reader's convenience.
\begin{thm} 
\label{Pethm}
Let $X$ be a set and $F$ be a nonnegative real-valued function on $X^n$ where $n\geq 2$ such that $F$ is symmetric with respect to the last $(n-1)$ variables. Assume that
\[F(u_1, u_2, u_3, \cdots, u_n) \leq [F(u_1, u_1, u_3, \cdots, u_n) ]^{\frac{1}{2}}  [F(u_2, u_2, u_3, \cdots, u_n) ]^{\frac{1}{2}}\quad\text{in } X^n.\]
Then
\[F(u_1, u_2, u_3, \cdots, u_n)\leq \prod_{i=1}^n [F(\underbrace{u_i, \cdots, u_i}_{n-\text{times}}) ]^{\frac{1}{n}} \quad\text{in } X^n.\]
\end{thm}
\begin{proof} We prove by induction on $n\geq 2$. The case $n=2$ is obvious. 
Assume the theorem is verified for $(n-1)\geq 2$. We prove it for $n$.  For a fixed $u_n$, we apply the induction hypothesis to the function
$(u_1, \cdots, u_{n-1})\mapsto F(u_1, \cdots, u_{n-1}, u_n)$ to obtain
\begin{equation}
\label{Pen}F(u_1, u_2, u_3, \cdots, u_n)\leq \prod_{i=1}^{n-1} [F(\underbrace{u_i, \cdots, u_i}_{(n-1)-\text{times}}, u_n) ]^{\frac{1}{n-1}}.
\end{equation}
By the symmetry and our assumption,
\begin{equation}
\label{Pe1}
F(\underbrace{u, \cdots, u}_{(n-1)-\text{times}}, v) =F(u, v, \underbrace{u, \cdots, u}_{(n-2)-\text{times}}) \leq [F(\underbrace{u, \cdots, u, u}_{n-\text{times}})  ]^{\frac{1}{2}}  [F(v, v,
\underbrace{u, \cdots, u}_{(n-2)-\text{times}}) ]^{\frac{1}{2}}, \end{equation}
and together with the induction hypothesis
\begin{equation}
\label{Pe2}F(v, v,
\underbrace{u, \cdots, u}_{(n-2)-\text{times}}) =F(v, 
\underbrace{u, \cdots, u}_{(n-2)-\text{times}},v) \leq [F( 
\underbrace{v, \cdots, v}_{n-\text{times}})]^{\frac{1}{n-1}} [F( 
\underbrace{u, \cdots, u}_{(n-1)-\text{times}}, v)]^{\frac{n-2}{n-1}}. \end{equation}
Substituting \eqref{Pe2} into \eqref{Pe1}, we find
\begin{equation}
\label{Pe3}
F(\underbrace{u, \cdots, u}_{(n-1)-\text{times}}, v) \leq [F(\underbrace{u, \cdots,  u}_{n-\text{times}})  ]^{\frac{n-1}{n}}  [F(\underbrace{v, \cdots, v}_{n-\text{times}})  ]^{\frac{1}{n}}. 
\end{equation}
Substituting \eqref{Pe3} into \eqref{Pen}, we have the conclusion of the theorem for $n$.
\end{proof}

We are now ready to prove the following  generalized Cauchy--Schwarz inequality, which is Theorem \ref{gCSIBP} (i). It asserts in particular the finiteness of the mixed Monge--Amp\`ere energy (see Definition \ref{mEdef}) in the energy class. 
\begin{thm}[Generalized Cauchy--Schwarz inequality] 
\label{gCS}
Let $\Omega\subset\R^n$ be a bounded convex domain.
Let $u_0, u_1,\cdots, u_n\in \E(\Omega)$ and $E$ be the Monge--Amp\`ere energy as in \eqref{Edef}. Then
\begin{equation}
\label{gCSu}
\int_\Omega |u_0| \,d\mu_n[u_1, \cdots, u_n] \leq [E(u_0)]^\frac{1}{n+1}\cdots [E(u_n)]^\frac{1}{n+1}.\end{equation}
\end{thm}

\begin{proof}
We divide the proof into several steps.

\medskip
\noindent
{\it Step 1. } Consider the case where  $\p\Omega\in C^3$, $u_0, \cdots, u_n\in C^3(\overline{\Omega})$ and vanish on the boundary $\p\Omega$. 
We recall Definition \ref{mEdef}, use \eqref{div0} and \eqref{divf} and integrating by parts to obtain
\begin{equation*}E(u_0, u_1, \cdots, u_n) =\int_\Omega -u_0\tilde M[u_1, \cdots, u_n]\,dx =\int_\Omega \sum_{i, j=1}^n \tilde M^{ij} [u_1, \cdots, u_{n-1}] D_ju_n D_i u_0\, dx.\end{equation*}
From \eqref{tMpos} and the Cauchy--Schwarz inequality, we have
\begin{equation*}
\begin{split}
\int_\Omega \sum_{i, j=1}^n \tilde M^{ij} [u_1, \cdots, u_{n-1}] D_j u_n  D_i u_0\, dx
&\leq \Big(\int_\Omega \sum_{i, j=1}^n \tilde M^{ij} [u_1, \cdots, u_{n-1}] D_j u_0  D_i u_0\, dx\Big)^{\frac{1}{2}}\\
&\quad \times  \Big(\int_\Omega \sum_{i, j=1}^n \tilde M^{ij} [u_1, \cdots, u_{n-1}] D_j u_n  D_i u_n\, dx\Big)^{\frac{1}{2}}.
\end{split}
\end{equation*}
This means 
\[E(u_0, u_1, \cdots, u_n)\leq  \big[E(u_0, u_0, u_1,\cdots, u_{n-1})\big]^{\frac{1}{2}}\cdot   \big[E(u_n, u_n, u_1,\cdots, u_{n-1})\big]^{\frac{1}{2}}.\]
Now, using Theorem \ref{Pethm}, we obtain
\[E(u_0, u_1,\cdots, u_n) \leq \prod_{i=0}^{n} \big[E(\underbrace{u_i, \cdots, u_{i}}_{(n+1)-\text{times}})\big]^{\frac{1}{n+1}}=[E(u_0)]^\frac{1}{n+1}\cdots [E(u_n)]^\frac{1}{n+1}.\]
The theorem is proved for this case.

\medskip
\noindent
{\it Step 2.} Consider the case  $u_i\in \E(\Omega)$ where $\Omega$ is a bounded uniformly convex domain with $\p\Omega \in C^{4,\alpha}$ for some $\alpha\in (0, 1)$. By Lemma \ref{lemH}, for each $0\leq i\leq n$ and $k\in\N$, we can find a convex function $u_{i, k}\in C^{4,\alpha}(\overline{\Omega})$ with $u_{i, k}=0$ on $\p\Omega$ and $D^2 u_{i, k}>0$ in $\overline{\Omega}$ such that 
$$\|u_{i, k}-u_i\|_{L^{\infty}(\Omega)}+ |E(u_{i, k})-E(u_i)|<1/k.$$
From Step 1, we have
\[\int_\Omega |u_{0, k}| \,d\mu_n[u_{1, k}, \cdots, u_{n, k}] \leq [E(u_{0, k})]^\frac{1}{n+1}\cdots [E(u_{n, k})]^\frac{1}{n+1}.\]
Letting $k\to\infty$ and using the lower semicontinuity result in Proposition \ref{I_lsc2}, we obtain \eqref{gCSu}. 

\medskip
\noindent
{\it Step 3.} The general case. Consider a sequence of smooth, uniformly convex domains $\Omega_k\subset\Omega$ such that $\bigcup_{k=1}^\infty\Omega_k =\Omega$. For each $i$, let $u_{i, k}$ solves
\[\det D^2 u_{i, k}=\chi_{\Omega_k}\mu_{u_i} \quad\text{in }\Omega_k,\quad u_{i, k}=0  \quad\text{on }\p\Omega_k.\]
By the comparison principle (Theorem \ref{compa1}), we have $u_i\leq u_{i, k}\leq 0$ in $\Omega_k$.
Hence
\begin{equation}
\label{g0} M:= \sup_{i, k} \|u_{i, k}\|_{L^{\infty}(\Omega_k)}\leq \sup_{0\leq i\leq n} \|u_i\|_{L^{\infty}(\Omega)}<\infty.\end{equation}
By Step 2, we have for each $k$
\begin{equation}
\label{g1}
\int_{\Omega_k} |u_{0, k}| \,d\mu_n[u_{1, k}, \cdots, u_{n, k}] \leq [E(u_{0, k};\Omega_k)]^\frac{1}{n+1}\cdots [E(u_{n, k};\Omega_k)]^\frac{1}{n+1}.\end{equation}
Here we emphasize the domain dependence of the Monge--Amp\`ere energy $E$.
Fix $0\leq i\leq n$. We now show that $u_{i, k}$ converges locally uniformly in $\Omega$ to $u_i$ (this comes from \eqref{g0}) and 
\begin{equation}
\label{g2}\lim_{k\to\infty} E(u_{i, k};\Omega_k) = E(u_i;\Omega).\end{equation}
Assume  \eqref{g2} holds. For each open convex set $U\Subset\Omega$, there exists a positive integer $k_U$ such that for all $k>k_U$, we have $U\subset \Omega_k$. By Proposition \ref{I_lsc2}, we have
\begin{equation}
\label{g3} \int_U |u_0|\, d\mu_n[u_1,\cdots, u_n] \leq \liminf_{k\to\infty}\int_{\Omega_k} |u_{0, k}|\, d \mu_n[u_{1, k},\cdots, u_{n, k}].\end{equation}
Combining \eqref{g1}--\eqref{g3}, we obtain \eqref{gCSu}.

\medskip
It remains to prove \eqref{g2}. For each fixed $\Omega'\Subset\Omega$, and $k$ large such that $\Omega'\subset\Omega_k$, we have
\begin{equation*}
\begin{split}
|E(u_{i, k};\Omega_k) - E(u_i;\Omega)| \leq \int_{\Omega\setminus\Omega_k}|u_i| d\mu_{u_i}  + |\int_{\Omega_k\setminus\Omega'} (u_i-u_{i, k}) d\mu_{u_i}| + |\int_{\Omega'} (u_i-u_{i, k}) d\mu_{u_i}|.
 \end{split}
\end{equation*}
Since $ \int_{\Omega}|u_i| d\mu_{u_i}<\infty$ and $u_{i, k}$ converges locally uniformly in $\Omega$ to $u_i$, the first and third terms go to $0$ when $k\to\infty$. For the second term, we have
\[ |\int_{\Omega_k\setminus\Omega'} (u_i-u_{i, k}) d\mu_{u_i}|  =  \int_{\Omega_k\setminus\Omega'} (u_{i, k}-u_i) d\mu_{u_i} \leq \int_{\Omega\setminus\Omega'} |u_i| d\mu_{u_i}.\]
Therefore, we have
\[\limsup_{k\to\infty} |E(u_{i, k};\Omega_k) - E(u_i;\Omega)| \leq \int_{\Omega\setminus\Omega'} |u_i| d\mu_{u_i}.\]
Letting $\Omega'$ tend to $\Omega$, we obtain \eqref{g2}.
\end{proof}
A simple consequence of Theorem \ref{gCS} is the convexity of the cone $\E(\Omega)$ of convex functions and the convexity of the $(n+1)$-th root of the Monge--Amp\`ere energy $E$ defined in \eqref{Edef}. 
\begin{prop}
\label{cvxlem}
 Let $\Omega$ be a bounded convex domain in $\R^n$ and $u, v\in \E(\Omega)$. Then
\[[E(u+ v)]^{\frac{1}{n+1}} \leq  [E(u)]^{\frac{1}{n+1}} + [E(v)]^{\frac{1}{n+1}}.\]
Consequently, $u+ v\in \E(\Omega)$ and $E^{\frac{1}{n+1}}$ is convex on $\E(\Omega)$.
\end{prop}
\begin{proof} Using the convention that  ${n\choose k}=0$ when $k<0$ or $k>n$, we have
\begin{equation*}
\begin{split}
E(u+ v) &=\int_\Omega -(u+ v) d\mu_n[u+ v, \cdots, u+ v]\\ &=\sum_{m=0}^n (-u) {n\choose m}d\mu_n[\underbrace{u,\cdots, u}_{m\text{ times}}, \underbrace{v,\cdots, v}_{(n-m)\text{ times}}] + 
\sum_{k=0}^n (-v) {n\choose k} d\mu_n[\underbrace{u,\cdots, u}_{k\text{ times}}, \underbrace{v,\cdots, v}_{(n-k)\text{ times}}]\\
&\leq \sum_{m=0}^n  {n\choose m} [E(u)]^{\frac{m+1}{n+ 1}}  [E(v)]^{\frac{n-m}{n+ 1}} +  \sum_{k=0}^n {n\choose k} [E(u)]^{\frac{k}{n+ 1}}  [E(v)]^{\frac{n-k+1}{n+ 1}} \\
&= \sum_{k=0}^{n+1}  {n\choose k-1} [E(u)]^{\frac{k}{n+ 1}}  [E(v)]^{\frac{n-k+1}{n+ 1}} +  \sum_{k=0}^{n+1} {n\choose k} [E(u)]^{\frac{k}{n+ 1}}  [E(v)]^{\frac{n-k+1}{n+ 1}}\\
&= \sum_{k=0}^{n+1} {n+1\choose k} [E(u)]^{\frac{k}{n+ 1}}  [E(v)]^{\frac{n-k+1}{n+ 1}}= \Big[ [E(u)]^{\frac{1}{n+1}} + [E(v)]^{\frac{1}{n+1}}\Big]^{n+1}.
 \end{split}
\end{equation*}
The convexity of $E^{\frac{1}{n+1}}$ follows from homogeneity and the above inequality.
The proposition is proved.
\end{proof}
We have the following lattice property of the finite energy class $\E(\Omega)$.
\begin{prop}[Lattice and monotonicity property of energy class]
\label{latticelem}
 Let $\Omega$ be a bounded convex domain in $\R^n$ and $u\in \E(\Omega)$. Assume that $\tilde u \in C(\overline{\Omega})$ is a 
convex function with $\tilde u=0$ on $\p\Omega$ and $\tilde u\geq u$ in $\Omega$. Then $\tilde u\in \E(\Omega)$ with Monge--Amp\`ere energy (see \eqref{Edef}) $E(\tilde u)\leq E(u)$.
\end{prop}
\begin{proof} Pick $x_0\in\Omega$. Let $w\in C^{0, 1}(\overline{\Omega})$ be the convex function whose graph is the cone with vertex $(x_0, -1)$ and the base $\Omega$, with $w=0$ on $\p\Omega$. Then
$w\in\E(\Omega)$ and $w<0$ in $\Omega$. For any $\e>0$,  $\tilde u> u+\e w$ in $\Omega$, so by Corollary \ref{mpcor} and Proposition \ref{cvxlem}, we have 
\[E(\tilde u)\leq E(u+\e w)\leq \big( [E(u)]^{\frac{1}{n+1}} + \e [E(w)]^{\frac{1}{n+1}}\big)^{n+1}.\]
Letting $\e\to 0$ concludes the proof of the proposition.
\end{proof}

\subsection{Integration by parts in energy class}
For the proof of the integration by parts formula in Theorem \ref{gCSIBP} (ii), we need several preparatory results. We have the following exchange result:
\begin{lem}[Exchange]
\label{Exlem}
Let $u_1,\cdots, u_n, v_1, v_2\in C(\overline{\Omega})$ be convex functions on a bounded domain $\Omega\subset\R^n$ where $v_1=v_2$ in a neighborhood of $\p\Omega$. Then
\begin{equation}
\label{Exch1}
\int_\Omega (v_1-v_2) \,d\mu_n[u_1, \cdots, u_n]= \int_\Omega u_n \,\big(d\mu_n[v_1, u_1, \cdots, u_{n-1}]- d\mu_n[v_2, u_1, \cdots, u_{n-1}]\big).\end{equation}
Consequently, if $\Omega$ is a bounded convex domain, $u_0, u_1,\cdots, u_n\in \E(\Omega)$ and $u_0=u_n$ in a neighborhood of $\p\Omega$, then
\begin{equation}
\label{Exch1c}\int_\Omega u_0 \,d\mu_n[u_1, \cdots, u_n]= \int_\Omega u_n \,d\mu_n[u_0, u_1, \cdots, u_{n-1}].\end{equation}
\end{lem}
\begin{proof} 
We first prove \eqref{Exch1}. 
Approximate $u_i$, $v_1$ and $v_2$ by smooth functions $u_i^\e$, $v_1^\e$ and $v_2^\e$ using standard regularizations for $\e\to 0^+$. If we shrink $\Omega$ a little to $\Omega'$, we have $v_1=v_2$ and $v_1^\e= v_2^\e$ in a neighborhood $N$ of  $\p\Omega'$ for small $\e>0$. By Lemma \ref{IBPC3} (i), we have
\[\int_{\Omega'} (v_1^\e- v_2^\e)\tilde M[u^\e_1, \cdots, u^\e_n]\,dx= \int_{\Omega'} u^\e_n\Big(\tilde M[v_1^\e, u^\e_1, \cdots, u^\e_{n-1}]-\tilde M[v_2^\e, u^\e_1, \cdots, u^\e_{n-1}]\Big)\,dx.
 \]
 The above right-hand side is unchanged if we replace $u_n^\e$ by $u_n^\e\varphi$ for any $\varphi\in C^3_c(\Omega')$ with $\varphi \equiv 1$ in $\Omega'\setminus N$. Consequently, we can 
 let $\e\to 0$ and use  the weak convergence of mixed Monge--Amp\`ere measures 
  to obtain
\begin{equation}
\label{Exch2}
\int_{\Omega'} (v_1-v_2)\, d\mu_n[u_1, \cdots, u_n]
=\int_{\Omega'} u_n\, \big(d\mu_n[v_1, u_1, \cdots, u_{n-1}]-d\mu_n[v_2, u_1, \cdots, u_{n-1}]\big).
\end{equation}
By  the locality result in Lemma \ref{mMAloc}, we can replace $\Omega'$ by $\Omega$ in  \eqref{Exch2}. Thus \eqref{Exch1} holds. 

\medskip
For $u_0, u_1,\cdots, u_n\in \E(\Omega)$ and $u_0=u_n$ in a neighborhood of $\p\Omega$,  \eqref{Exch1} gives
\[\int_\Omega (u_0-u_n) \,d\mu_n[u_1, \cdots, u_n]= \int_\Omega u_n \,\big(d\mu_n[u_0, u_1, \cdots, u_{n-1}]-d\mu_n[u_n, u_1, \cdots, u_{n-1}]\big).\]
This gives  \eqref{Exch1c}, since we can cancel the nonnegative finite term $\int_\Omega -u_n \,d\mu_n[u_1, \cdots, u_n]$
by an application of Theorem \ref{gCS}.
\end{proof}

We next study the  Monge--Amp\`ere energy of convex functions truncated from below.
\begin{lem} [Cut-off]
\label{cutlem}
Let $u\in\E(\Omega)$ where $\Omega$ is a bounded convex domain in $\R^n$. For $\e>0$, let $u_\e=\max\{u, -\e\}$. Then the Monge--Amp\`ere energy $E(u_\e)\to 0$ when $\e\to 0$.
\end{lem}
\begin{proof} First, Proposition \ref{latticelem} implies $u_\e\in\E(\Omega)$. We give here another simple proof.
Let $A_\delta=\{x\in\Omega: u(x)<-\delta\}$. Then $u_\e= u$ in $\Omega\setminus A_\e$. By the locality of the Monge--Amp\`ere measure (Lemma \ref{mMAloc}), we have
\[\int_\Omega |u_\e| \, d\mu_{u_\e} 
= \int_{A_{\e/2}} |u_\e| \, d\mu_{u_\e}  + \int_{\Omega\setminus \overline{A_{\e/2}}} |u| \, d\mu_u \leq \int_{A_{\e/2}} |u_\e| \, d\mu_{u_\e}  + E(u).\]
In $A_{\e/2}$, $u_\e$ is globally Lipschitz so $\int_{A_{\e/2}} |u_\e| \, d\mu_{u_\e} <\infty$. Hence $E(u_\e)<\infty$ and $u_\e\in\E(\Omega)$.

\medskip
Similar to Corollary \ref{mpcor}, we now prove that
\begin{equation}
\label{uuecomp}
E(u_\e)=\int_\Omega -u_\e \,d\mu_{u_\e} \leq \int_\Omega -u_\e \,d\mu_u.\end{equation}
The setting here is a bit different because $u_\e$ is not strictly greater than $u$, so we give the detailed proof.
We use the consequence of the exchange result in Lemma \ref{Exlem} to estimate
\begin{equation*}
\begin{split}
D:=\int_\Omega -u_\e \,d\mu_u + \int_\Omega u_\e \,d\mu_{u_\e} 
&=\int_\Omega (u_\e-u) \,d\mu_{n}[u_\e,\cdots, u_\e]\\&\quad\quad + \int_\Omega u_\e\,d\mu_{n}[u_\e,\cdots, u_\e, u] - \int_\Omega u_\e \,d\mu_{n}[u,\cdots, u]  \\
&\geq  \int_\Omega u_\e\,d\mu_{n}[u_\e,\cdots, u_\e, u] - \int_\Omega  u_\e\,d\mu_{n}[u,\cdots, u].
\end{split}
\end{equation*}
Continuing this process, we find
\begin{equation*}
D\geq \int_\Omega u_\e\,d\mu_{n}[u_\e, u\cdots, u] - \int_\Omega u_\e \,d\mu_{n}[u,\cdots, u]  = \int_\Omega (u_\e-u)\,d\mu_{n}[u_\e, u\cdots, u] \geq 0.
\end{equation*}
Since $|u_\e|=\min\{ |u|, \e\}$ and $\int_\Omega |u|\, d\mu_u<\infty$, 
the Dominated Convergence Theorem tells us from \eqref{uuecomp} that $E(u_\e)\to 0$ when $\e\to 0$.
\end{proof}

We are now ready to prove the integration by parts formula in Theorem \ref{gCSIBP} (ii).  
\begin{thm}[Integration by parts] 
\label{IBPf}
Let $u_0, u_1,\cdots, u_n\in \E(\Omega)$ where  $\Omega$ is a bounded convex domain in $\R^n$. Then
\[\int_\Omega u_0 \,d\mu_n[u_1, \cdots, u_n]= \int_\Omega u_n \,d\mu_n[u_0, u_1, \cdots, u_{n-1}].\]
\end{thm}
\begin{proof}
For $\e>0$, let $v_\e:=\max\{u_n, -\e\}$. Then $v_\e$ is convex and $v_\e=v_n$ in a neighborhood of $\p\Omega$. 
By Lemma \ref{Exlem} 
\[\int_\Omega u_0 \,\big(d\mu_n[u_1, \cdots, u_n]-d\mu_n[u_1, \cdots, v_\e]\big)= \int_\Omega (u_n-v_\e) \,d\mu_n[u_0, u_1, \cdots, u_{n-1}].\] 
Using Theorem \ref{gCS}
and the cut-off Lemma \ref{cutlem}, we conclude that, as $\e\to 0$, 
\[\int_\Omega v_\e \,d\mu_n[u_0, u_1, \cdots, u_{n-1}]\to 0\quad \text{and}\quad \int_\Omega u_0 \,d\mu_n[u_1, \cdots, v_\e]\to 0.\]
The conclusion of the theorem follows.
\end{proof}
We give some useful consequences of the integration by parts formula in Theorem \ref{IBPf}. The first one
allows us to factorize the difference of Monge--Amp\`ere energies defined in \eqref{Edef}.
\begin{lem} [Factorization formula]
\label{cocyc}
Let $\Omega$ be a bounded convex domain in $\R^n$. Then,
for $u, v\in\E(\Omega)$, we have 
%the following identity concerning the difference of their Monge--Amp\`ere energies
\[E(u)-E(v)=\int_\Omega v\, d\mu_v- \int_\Omega u  \, d\mu_u= \sum_{k=0}^n \int_\Omega (v-u) \, d\mu_n[v,\cdots, v, \underbrace{u,\cdots, u}_{k\, \text{times}}].\]
\end{lem}
\begin{proof}
As a consequence of the integration by parts formula in Theorem \ref{IBPf}, we have 
\begin{equation*}
\begin{split}
E(u)-E(v)
&=\int_\Omega (v-u)\,d\mu_n[v, \cdots, v]  +\int_\Omega u \,d\mu_n[v, \cdots, v]- \int_\Omega u \,d\mu_n[u, \cdots, u]\\
&=\int_\Omega (v-u) \,d\mu_n[v, \cdots, v] +\int_\Omega v \,d\mu_n[v, \cdots, v, u]- \int_\Omega u \,d\mu_n[u, \cdots, u].
\end{split}
\end{equation*}
Continuing this process, we obtain the conclusion of the lemma.
\end{proof}

A second  consequence of Theorem \ref{IBPf} is the monotonicity  of the Monge--Amp\`ere energy.
\begin{cor}[Monotonicity] 
\label{monof}
Let $\Omega$ be a bounded convex domain in $\R^n$.
If $u_n\geq \tilde u_n$ in $\Omega$ and $u_0, u_1,\cdots, u_n, \tilde u_n\in\E(\Omega)$, then
\[\int_\Omega -u_0 \,d\mu_n[u_1, \cdots, u_{n-1}, u_n]\leq \int_\Omega -u_0 \,d\mu_n[u_1, \cdots, u_{n-1}, \tilde u_n].\]
\end{cor}
\begin{proof} By Theorem \ref{IBPf}, the difference between the left-hand side and the right-hand side of the claimed inequality is
\[ \int_\Omega -u_n \,d\mu_n[u_1, \cdots, u_{n-1}, u_0]+
 \int_\Omega \tilde u_n \,d\mu_n[u_1, \cdots, u_{n-1}, u_0]\leq 0.\]
The corollary is proved.
\end{proof}
The third consequence of Theorem \ref{IBPf} is the nonlinear integration by parts inequality.
 \begin{cor}[Nonlinear integration by parts inequality]
 \label{NIBP30} Let $\Omega$ be a bounded convex domain in $\R^n$ and
 $u_0, u_1, \cdots, u_n\in \E(\Omega)$ with
 $\mu_{u_i}= f_i\nu\,(0\leq i\leq n)$
 where $f_0, f_1, \cdots, f_n\in C(\Omega)$ are nonnegative functions and $\nu$ is a Borel measure in $\Omega$.
Then
\begin{equation*} \int_{\Omega} |u_0|\, d\mu_n[u_1, \cdots, u_n] \geq \int_{\Omega} |u_n|\Big(\prod_{i=0}^{n-1}f_i^{\frac{1}{n}}\Big) \,d\nu.
\end{equation*}
\end{cor}
\begin{proof}
By Theorem \ref{IBPf}
 and the mixed Monge--Amp\`ere inequality in Theorem \ref{mMAnu}, 
 \[ \int_{\Omega} |u_0|\, d\mu_n[u_1, \cdots, u_n] =  \int_{\Omega} |u_n|\, d\mu_n[u_1, \cdots, u_{n-1}, u_0] \geq \int_{\Omega} |u_n|\Big(\prod_{i=0}^{n-1}f_i^{\frac{1}{n}}\Big) \,d\nu. \]
 The corollary is proved. 
\end{proof}
Another consequence of Theorem \ref{IBPf} is the reverse Aleksandrov estimate.
\begin{prop}[Reverse Aleksandrov estimate]
\label{ReA}
 Let $\Omega$ be a bounded convex domain in $\R^n$.  Let $\lambda[\Omega]$ be the Monge--Amp\`ere eigenvalue of $\Omega$ and let $w\in C(\overline{\Omega})\cap C^{\infty}(\Omega)$ be a nonzero Monge--Amp\`ere eigenfunction of $\Omega$; that is, 
  \begin{equation*}
   \det D^{2} w=\lambda[\Omega] |w|^{n} \quad\text{in} ~\Omega, \quad
w =0\quad\text{on}~\p \Omega.
\end{equation*}
Assume that $u\in C(\overline{\Omega})$ is a convex function in $\Omega$ with $\mu_u\in L_{\text{loc}}^1(\Omega)$. Then
\begin{equation}
\label{ReA2}
\int_\Omega (\lambda[\Omega])^{1/n} |u| |w|^n\,dx\geq \int_\Omega \mu_u^{1/n} |w|^n\,dx.
\end{equation}
\end{prop}

\begin{proof}
For each $\e\in (0, 1)$, let 
$u_{\e}\in C(\overline{\Omega})$
be the convex solution to 
\begin{equation*}
   \mu_{u_{\e}}=\chi_{\{x\in\Omega: \dist(x, \p\Omega)>\e\}}\mu_u \quad\text{in} ~\Omega, \quad
u_{\e} =0\quad\text{on}~\p \Omega.
\end{equation*}
Clearly, $u_\e\in \E(\Omega)$ and $|u_\e|\leq |u|$.
By Theorem \ref{IBPf}, we have
\[\int_\Omega  |u_\e| \det D^2 w\,dx=\int_\Omega  |u_\e| \, d\mu_n[w,\cdots, w]=\int_\Omega|w|\, d\mu_n[u_\e, w, \cdots, w].\]
Then, using $\det D^2 w= \lambda[\Omega] |w|^n$ together with Theorem \ref{mMAnu}, we get
\begin{eqnarray*}\int_\Omega  |u| \lambda[\Omega] |w|^n\, dx\geq \int_\Omega  |u_\e| \det D^2 w\,dx&\geq& \int_\Omega |w|\mu_{u_\e}^{1/n} \mu_w^{\frac{n-1}{n}}\,dx
\\&=&  \int_\Omega ( \lambda[\Omega])^{\frac{n-1}{n}} (\mu_u\chi_{\{\dist(\cdot, \p\Omega)>\e\}} )^{1/n} |w|^n\,dx.
\end{eqnarray*}
Letting $\e\to 0$, and
dividing 
the above estimates by $  ( \lambda[\Omega])^{\frac{n-1}{n}}$, we obtain (\ref{ReA2}).
\end{proof}
\begin{rem}
Compared to Theorem \ref{Al_E} (i), the function $u$ appears on the dominating side in (\ref{ReA2}). For this reason,
(\ref{ReA2}) can be viewed as a sort of reverse Aleksandrov estimate. When $u$ is a Monge--Amp\`ere eigenfunction of $\Omega$, (\ref{ReA2}) is an equality, so it  is sharp.
\end{rem}

\section{Convex envelopes and variational derivatives} 
\label{env_sec}
In this section, we study convex envelopes of continuous functions and variational derivatives of  
Monge--Amp\`ere energies of convex envelopes. We will prove Theorem \ref{envED}, restated as Theorem \ref{envelopeED}.
 
  \begin{defn}[Convex envelope of a continuous function]
\label{cvx_env_defn}
Let $v\in C(\overline{\Omega})$ be a continuous function on a bounded domain $\Omega$ in $\R^n$. 
The {\it convex envelope} $\Gamma_v$ of $v$ in $\overline{\Omega}$ is defined to be
\begin{equation*}\Gamma_v(x):=\sup\big\{\varphi(x): \varphi\leq v\text{ in }\overline{\Omega},\, \varphi \text{ is convex in }\overline{\Omega}\big\}.
\end{equation*}
\end{defn}
\subsection{Derivative of the Monge--Amp\`ere energy of convex envelopes }
One of our key technical ingredients in the variational approach to \eqref{eqp1} is concerned with the variational derivative at $0$ of $E\circ \Gamma$ for variations of the form $u+ tv$ where $u\in \E(\Omega)$ and $v$ is a globally Lipschitz convex function and $t\in\R$.
\begin{thm} 
\label{envelopeED}
 Let $\Omega$ be a bounded convex domain in $\R^n$,
$u\in \E(\Omega)$ and $v\in \E(\Omega) \cap C^{0, 1}(\overline{\Omega})$. Let $E$ be the Monge--Amp\`ere energy as in \eqref{Edef}. Then 
\[\frac{d}{dt}\mid_{t=0} E (\Gamma_{u + tv})=(n+ 1)\int_\Omega (-v) \, d\mu_u.\]
\end{thm}
 
The rest of this subsection is devoted to proving Theorem \ref{envelopeED}. Our proof is inspired by that of Theorem 4.11 in Lu \cite{Lu} for the case of complex Hessian energies; see also Lu--Nguyen \cite[Lemma 6.12]{LN} for an exposition using the convexity of Monge--Amp\`ere energies along affine curves (which can be deduced from Proposition \ref{cvxlem})  and the orthogonal relation in Proposition \ref{envelopeD}.
\medskip

Observe that the convex envelope has zero Monge--Amp\`ere measure on the set where it is below the original function.
\begin{lem}
\label{Enlinear}
Let $v\in C(\overline{\Omega})$ be a continuous function on a bounded domain $\Omega$ in $\R^n$. Then
\[\int_{\{\Gamma_v<v\}} d\mu_{\Gamma_v} =0.\]
\end{lem}
\begin{proof} This is a consequence of Lemmas 6.7 and 2.45 in \cite{Lbook}. Indeed, 
if $p\in\p \Gamma_v(x_0)$ where $x_0\in\{\Gamma_v<v\}$, then by Lemma 6.7 in \cite{Lbook},  
$p$ is also the slope of a supporting hyperplane to $v$ at more than one point. By Aleksandrov's lemma (see Lemma 2.45 in \cite{Lbook}), the set of all such slopes has Lebesgue measure zero. Hence
\[\int_{\{\Gamma_v<v\}} d\mu_{\Gamma_v} =|\p\Gamma_v(\{\Gamma_v<v\})|=0,\]
completing the proof of the lemma.
\end{proof}

The following consequence of the maximum principle will be useful.
 \begin{lem}
 \label{Lipcom}
 Let $u, v\in C(\overline{\Omega})$ be convex functions on a bounded domain $\Omega$ in $\R^n$.  If $u=v$ on $\p\Omega$, $v\geq u$ in $\Omega$, and $u\in C^{0, 1}(\overline{\Omega})$, then
 $v\in C^{0, 1}(\overline{\Omega})$.
 \end{lem}
 \begin{proof} This is a consequence of the maximum principle which gives $\p v(\Omega)\subset \p u(\Omega)$ and the fact that $u\in C^{0, 1}(\overline{\Omega})$ if and only if $\p u(\Omega)$ is bounded.
 \end{proof}

To compute the left derivative of $E(\Gamma_{u+ tv})$ at $0$, we need the following proposition. 
\begin{prop}[Convex envelope and orthogonal relation]
\label{envelopeD}
 Let $\Omega$ be a bounded convex domain in $\R^n$, $u\in\E(\Omega)$, $v\in \E(\Omega) \cap C^{0, 1}(\overline{\Omega})$, and $0\leq k\leq n$. Then
\begin{equation}
\label{GD1}
\lim_{t\to 0^-} \int_{\Omega} \frac{\Gamma_{u+ tv}- u-tv}{t} \, d\mu_n[\Gamma_{u+ tv,\cdots}, \Gamma_{u+ tv}, \underbrace{u,\cdots, u}_{k\, \text{times}}] =0.\end{equation}
As a consequence,
\begin{equation}
\label{GD2}\lim_{t\to 0^-} \int_{\Omega} \frac{\Gamma_{u+ tv}- u}{t} \, d\mu_n[\Gamma_{u+ tv,\cdots}, \Gamma_{u+ tv}, \underbrace{u,\cdots, u}_{k\, \text{times}}]  =\int_{\Omega} v\,d\mu_u.\end{equation}
\end{prop}
\begin{proof} Since $v\in C^{0, 1}(\overline{\Omega})$ and $v=0$ on $\p\Omega$,  there exists a positive constant $C$ such that $|v|\leq C\dist (\cdot,\p\Omega)$  in $ \overline{\Omega}$.
The convexity of $u$ gives 
$|u(x)|\geq \frac{\text{dist}(x,\p\Omega)}{\text{diam}(\Omega)}\|u\|_{L^{\infty}(\Omega)}$ for all  $x\in\Omega$.

It follows that for $-c<t< 1$ where $c=\|u\|_{L^{\infty}(\Omega)}/(C \cdot \text{diam}(\Omega))>0$, we have
\[u(x) + t v(x) =-|u(x)| - t|v(x)| \leq 0 \quad \text{for all  }x\in\Omega.\]
For $-c<t<0$, we have $0\geq u+ tv\geq u$ so $0\geq u+ tv\geq\Gamma_{u + tv}\geq u$.  Hence $\Gamma_{u + tv}\in\E(\Omega)$, by Proposition \ref{latticelem}.
Denote
\[w(t) = \frac{\Gamma_{u+ tv}-(u+ tv)}{t},\quad -c<t<0. \]
Then $w(t)$ is decreasing in $t$ and $0\leq w(t) \leq -v=|v|$. We have $w(t)=0$ on $\p\Omega$, and $\Gamma_{u+ tv}$ converges uniformly to $u$ on $\overline{\Omega}$ when $t\to 0^-$.
Note that $0\leq \Gamma_{u+ tv} -u \leq |t||v|$.

Fixing $-c<s<0$, we have $-w(s)\geq v$ so $-w(s)\in \E(\Omega)$, and
\begin{equation}
\label{GD1a}
\begin{split}
0&\leq \limsup_{t\to 0^-} \int_{\Omega} w(t) \, d\mu_n[\Gamma_{u+ tv,\cdots}, \Gamma_{u+ tv}, \underbrace{u,\cdots, u}_{k\, \text{times}}] \\
&\leq \limsup_{t\to 0^-} \int_{\Omega} w(s) \, d\mu_n[\Gamma_{u+ tv,\cdots}, \Gamma_{u+ tv}, \underbrace{u,\cdots, u}_{k\, \text{times}}]  
= \int_\Omega w(s) \, d\mu_n[u, \cdots, u]=:I[u,v, s], 
\end{split}
\end{equation}
 using the multi-linearity of $\mu_n$ and the generalized Cauchy--Schwarz inequality (Theorem \ref{gCS}). 
 
 Because $w(s)>0$ if and only if $\Gamma_{u+ sv}-sv<u$, we use Corollary \ref{mpcor} to obtain
 \begin{equation*}
I[u, v, s]\leq \int_{\{\Gamma_{u+ sv}-sv<u\}} |v| \, d\mu_u
\leq \int_{\{\Gamma_{u+ sv}-sv<u\}}  |v| \, d\mu_{\Gamma_{u+ sv}- sv}.
\end{equation*}
It follows from the multi-linearity of $\mu_n$, Lemma \ref{Enlinear}, 
\[\mu_{\Gamma_{u+ sv}- sv}=\mu_n [\Gamma_{u+ sv}- sv, \cdots, \Gamma_{u+ sv}- sv]\quad\text{and}\quad  \int_{\{\Gamma_{u+ sv}<u+ sv\}}  |v| \,d \mu_{\Gamma_{u+ sv}}=0,\]
that 
 \begin{equation}
 \label{GD1b}
I[u, v, s]\leq\sum_{k=1}^n \int_{\{\Gamma_{u+ sv}<u+ sv\}} C(n, k) |v| (-s)^k \, d\mu_n [\Gamma_{u+ sv}, \cdots, \Gamma_{u+ sv}, \underbrace{v,\cdots, v}_{k\,\text{times}}].
\end{equation}
For each $1\leq k\leq n$,  the generalized Cauchy--Schwarz inequality in Theorem \ref{gCS} easily gives
\begin{equation}
\label{GD1c}
 \int_{\{\Gamma_{u+ sv}<u+ sv\}}  |v| (-s)^k \, d\mu_n [\Gamma_{u+ sv}, \cdots, \Gamma_{u+ sv}, \underbrace{v,\cdots, v}_{k\,\text{times}}]\to 0\quad\text{when } s\to 0^{-}.\end{equation}
Combining \eqref{GD1a}--\eqref{GD1c}, we obtain \eqref{GD1}.

\medskip

To prove \eqref{GD2}, we just use \eqref{GD1} and  
the generalized Cauchy--Schwarz inequality in 
Theorem \ref{gCS}.
The proposition is proved.
\end{proof}

For the proof of Theorem \ref{envelopeED}, we will also use the following lemma, which extends a result due to Krylov \cite{K76} for globally Lipschitz functions.
\begin{lem}
\label{Kr_lem}
Let $\Omega\subset\R^n$ be a bounded convex domain. Let $u, v\in \E(\Omega)$.
Then
$$\frac{d}{dt}\mid_{t=0^+} E(u+ tv)= (n+ 1) \int_{\Omega} -v\, d\mu_u.$$
\end{lem}
\begin{proof} Compute using the factorization formula (see Lemma \ref{cocyc})
 \begin{equation*}
\begin{split}\frac{E(u + tv)- E(u)}{t}& = \sum_{k=0}^n \int_\Omega -v \, d\mu_n[u+ tv,\cdots, u+ tv, \underbrace{u,\cdots, u}_{k\, \text{times}}]\\
& =(n+ 1) \int_{\Omega} -v\, d\mu_u -  \sum_{k=0}^{n-1} \int_\Omega C(n, k)tv \, d\mu_n[\underbrace{u+ tv,\cdots, u+ tv}_{n-k-1\, \text{times}}, v, \underbrace{u,\cdots, u}_{k\, \text{times}}].
\end{split}
\end{equation*}
We now let $t\to 0^+$ using 
the generalized Cauchy--Schwarz inequality in 
Theorem \ref{gCS}
to conclude.
\end{proof}
We are now ready to give the proof of Theorem \ref{envelopeED}.
\begin{proof}[Proof of Theorem \ref{envelopeED}] If $t>0$, then $\Gamma_{u+ tv}= u+ tv$. Applying Lemma \ref{Kr_lem}
we have
\[\frac{d}{dt}\mid_{t=0^+} E (\Gamma_{u + tv}) = \frac{d}{dt}\mid_{t=0^+} E(u+ tv) =(n+ 1)\int_\Omega (-v) \, d\mu_u.\]
Consider now $t<0$ and the difference quotient (see Lemma \ref{cocyc})
\[\frac{E (\Gamma_{u + tv}) - E(u)}{t} = \sum_{k=0}^n \int_\Omega \frac{u-\Gamma_{u + tv}}{t} \, d\mu_n[\Gamma_{u + tv},\cdots, \Gamma_{u + tv}, \underbrace{u,\cdots, u}_{k\, \text{times}}].\]
By Proposition \ref{envelopeD}, we have
\[\frac{d}{dt}\mid_{t=0^-} E (\Gamma_{u + tv}) =\lim_{t\to 0^-} \frac{E (\Gamma_{u + tv}) - E(u)}{t}  =(n+ 1)\int_\Omega (-v) \, d\mu_u. \]
The theorem is proved.
\end{proof}

 \subsection{Test functions} Theorem \ref{envelopeED} allows us to compute the variational derivative of the Monge--Amp\`ere energy in the direction of globally Lipschitz convex functions $v$.
 In applications, we need to consider variations with compactly supported nonzero test functions. These are not convex, but interestingly, they are differences of globally Lipschitz convex functions.

\medskip
 The following lemma, motivated by Cegrell \cite[Lemma 3.1]{Ceg}, allows us to consider test functions in the class of globally Lipschitz convex functions. 
\begin{lem} 
 \label{2cvx}
 Let $\varphi\in C^\infty_c(\Omega)$ where $\Omega$ is a bounded convex domain in $\R^n$. Then there exist $\varphi_1,\varphi_2\in \E(\Omega)\cap C^{0, 1}(\overline{\Omega})$ such that 
 $\varphi=\varphi_1-\varphi_2$.
 \end{lem}
 \begin{proof} Fix $x_0\in\Omega$. Let $v\in C(\overline{\Omega})$ be a convex function whose graph is the cone with vertex $(x_0, -1)$ and the base $\Omega$, with $v=0$ on $\p\Omega$. Then
 $v\in \E(\Omega)\cap C^{0, 1}(\overline{\Omega})$. 
 
 First, we choose a large positive constant $C_0$ such that $\varphi + C_0|x|^2$ is convex. Next, we choose positive constants $C_1, C_2$ satisfying
 \[C_1=\max_{\overline{\Omega}} \big( |\varphi| + C_0|x|^2 + 1\big),\quad C_2 v<-2C_1  \quad\text{on the support of  } \varphi.\]
 The conclusion of the lemma holds for
 \[\varphi_1 =\max\{\varphi + C_0|x|^2-C_1, C_2v\},\quad \varphi_2= \max\{ C_0|x|^2-C_1, C_2 v\}.\]
 Indeed, it is easy to verify that $\varphi_1$ and $\varphi_2$ are convex functions in $\overline\Omega$, vanish on $\p\Omega$, and $\varphi=\varphi_1-\varphi_2$. Since $\varphi_1= \varphi_2=C_2 v=0$  on $\p\Omega$, $\min\{\varphi_1,\varphi_2\}\geq C_2 v$ in $\Omega$,  and $v\in C^{0, 1}(\overline{\Omega})$,  we also have $\varphi_1,\varphi_2\in C^{0, 1}(\overline{\Omega})$ by Lemma \ref{Lipcom}.  The lemma is proved.
 \end{proof}

   The following approximation result is inspired by Cegrell \cite[Theorem 2.1]{Ceg}.
      \begin{lem}
 \label{appro_lem}
  Let $\Omega$ be a bounded convex domain in $\R^n$ and let $u\in\K(\Omega)$ where 
  \[\K(\Omega) = \{  w \in C(\overline{\Omega}):   ~w~\text{is convex, nonzero in } \Omega,~ w=0~\text{on}~\p\Omega \}.\]
  Then there is a decreasing sequence of functions $\{u_k\}_{k=1}^\infty\subset  \K(\Omega)\cap C^{0, 1}(\overline{\Omega})$ such that $u_k$ converges uniformly to $u$ in $\overline{\Omega}$.
  \end{lem}
 \begin{proof} 
 Fix $v\in \K(\Omega)\cap C^{0, 1}(\overline{\Omega})$ as in the proof of Lemma \ref{2cvx}.
 
 For each $r>0$, we denote $\Omega_r=\{x\in\Omega: \dist(x,\p\Omega)>r\}$. We choose a decreasing sequence of positive numbers $\{r_k\}_{k=1}^\infty$ such that 
 $0< r_k< \dist \big(\big\{x\in \Omega: v(x)<-\frac{1}{2k^2}\big\}, \p\Omega\big)$.
 
  Let $\varphi \in C_c^{\infty}(\R^n)$ be a standard mollifier. For $\e>0$, let $\varphi_\e(x):= \e^{-n}\varphi(x/\e).$
 Let $u_{r_k}=u\ast \varphi_{r_k}$ be the mollification of $u$ in $\Omega_{r_k}$. Then $u_{r_k}$ is negative and convex in  $\Omega_{r_k}$. We extend it to be a negative and convex function in  $\Omega$. It suffices to choose 
 \[u_{k}=\sup_{k\leq m}\max\Big(u_{r_m}-1/m, mv\Big).\]
 
 If $\dist(x_0,\p\Omega)\leq r_m$ then $v(x_0)\geq -\frac{1}{2m^2}$ and thus $mv(x_0)>-\frac{1}{m}\geq u_{r_m}-1/m$. Hence 
 \[\max\Big(u_{r_m}-1/m, mv\Big) = mv\quad \text{in }\overline\Omega\setminus \Omega_{r_m}.\]
 It follows that $u_k\in\K(\Omega)$. Since $u_k=kv$ on $\p\Omega$, $u_k\geq kv$ in $\Omega$, and $v\in C^{0, 1}(\overline{\Omega})$, we have $u_k\in C^{0, 1}(\overline{\Omega})$ by Lemma \ref{Lipcom}.
It is easy to see that for each $x\in\overline{\Omega}$, $\{u_k(x)\}_{k=1}^\infty$ decreases to $u(x)$. By Dini's theorem, $\{u_k(x)\}_{k=1}^\infty$ converges uniformly to $u$ on $\overline\Omega$.
 \end{proof}
 \section{Degenerate Monge--Amp\`ere equations}
\label{deg_sec}
In this section, we study the solvability and uniqueness for several degenerate Monge--Amp\`ere equations including the Monge--Amp\`ere eigenvalue problem in real Euclidean spaces that involve singular Borel measures. 
In particular, we will prove Theorems \ref{Dist0pn} and \ref{EVP1}.
 \subsection{Solvability of degenerate Monge--Amp\`ere equations}
 Using variational method, we first establish the existence part of Theorem \ref{Dist0pn} (i), (iii) and (iv), but under some restrictions of the measure $\nu$ when $p\in (0, n)$. Let $E$ be the Monge--Amp\`ere energy as in \eqref{Edef}.
\begin{thm} 
\label{Dirdistp}
Let $\Omega\subset\R^n$ be a bounded convex domain, $p\in (-1, \infty)$, and
$\nu$ be a locally finite Borel measure on $\Omega$ satisfying
\begin{equation}
\label{Distnu}
\nu(\Omega)>0\quad \text{and}\quad \int_\Omega \dist^{\frac{p+1}{n+1}}(\cdot,\p\Omega)d\nu\leq C<\infty.\end{equation}
When $p=n$, we can assume more generally that 
 \begin{equation}
\label{Distnu2} 
\small
\int_{\{x\in\Omega: \dist(x, \p\Omega)\leq1/m\}}|v|^{n+1}\,d\nu\to 0 \,\mbox{when $m\to\infty$ uniformly in $v$ on bounded subsets of $\E(\Omega)$}.\end{equation}
Let $E$ be as in \eqref{Edef}. Denote
\[K(\Omega)=\bigg\{ v\in\E(\Omega),\quad \int_\Omega |v|^{p+1}\, d\nu=1\bigg\}.\]
Then, there exists a convex function $u\in K(\Omega)$ such that 
\[E(u)=\min_{v\in K(\Omega)} E(v):=\lambda_0.\]
Moreover, $\lambda_0>0$ and each minimizer $u$ is a nonzero Aleksandrov solution to
\begin{equation}
\label{Distnu3}
   \det D^{2} u=\lambda_0 |u|^p\nu \quad\text{in} ~\Omega, \quad
u =0\quad\text{on}~\p \Omega.
\end{equation}
\end{thm}
\begin{proof} We divide the proof into several steps.

\medskip
\noindent
{\it Step 1: Existence of minimizers with constraint.} We consider two cases.

{\bf General case of $p$ where \eqref{Distnu} holds.}
If $u\in \E(\Omega)$, then 
by Theorem \ref{Al_E}, \[|u|\leq C(n, \Omega, E(u))\dist^{\frac{1}{n+1}}(\cdot,\p\Omega) \quad \text{in}\quad  \Omega;\]
hence
\[\int_\Omega |u|^{p+1}\, d\nu \leq C^{p+1} \int_\Omega \dist^{\frac{p+1}{n+1}}(\cdot,\p\Omega)d\nu<\infty.\]
This together with $\nu(\Omega)>0$ implies that $K(\Omega)$ is nonempty. 

Let $\{u_k\}_{k=1}^\infty\subset K(\Omega)$ be a minimizing sequence for $E$ so that
\[\lim_{k\rightarrow \infty } E(u_k)= \inf_{v\in K(\Omega)} E(v)=:\lambda_0.\]
By taking a cone-like function $v\in K(\Omega)$ with vertex at $x_0\in\Omega$, we see that $\lambda_0<\infty$.
Then $\{E(u_k)\}_{k=1}^\infty$ is bounded from above by a constant $C_1=C_1(n, p, \nu,\Omega)$. Thus, by Theorem \ref{Al_E},
we find that $\{u_k\}_{k=1}^\infty$ is uniformly bounded in $C^{ \frac{1}{n+1}}(\overline{\Omega})$. By the Arzela--Ascoli theorem, 
there exists a subsequence $\{u_{k_i}\}_{i=1}^\infty$ that converges uniformly on  $\overline{\Omega}$ to a convex function 
$u\in C^{ \frac{1}{n+1}}(\overline{\Omega})$ with $u=0$ on $\p\Omega$. By 
the lower semicontinuity property of $E$ (see Lemma \ref{I_lsc1}),
we have
\[E(u) \leq \liminf_{i\rightarrow \infty} E(u_{k_i}).\]
Moreover, using 
the Dominated Convergence Theorem, we have
\[\int_\Omega |u|^{p+1}\, d\nu=\lim_{i\rightarrow \infty} \int_\Omega |u_{k_i}|^{p+1}\, d\nu=1.\]
Clearly $u\in K(\Omega)$.
As a result, $u$ is a minimizer of $E$ over $K(\Omega)$, so in fact
\[E(u)=   \min_{v\in K(\Omega)} E(v)=\lambda_0>0.\]

{\bf The case $p=n$ and $\nu$ is a locally finite Borel measure on $\Omega$ satisfying \eqref{Distnu2}.} Let $u_k$, $u_{k_i}$ and $u$ be as above. It suffices to show $u\in K(\Omega)$.
We claim that 
\begin{equation}
\label{Inucont}
\int_\Omega |u|^{n+1}\, d\nu=\lim_{i\rightarrow \infty} \int_\Omega |u_{k_i}|^{n+1}\, d\nu=1.\end{equation}
Let $ \nu_{m}=\chi_{\{x\in\Omega: \dist(x, \p\Omega)>1/m\}}\nu$.
For each $m\in\N$, it is clear that $\int_\Omega \dist(\cdot, \p\Omega) d\nu_m<\infty$. As in the general case above, by the Dominated Convergence Theorem, we have
\[\int_\Omega |u|^{n+1}\, d\nu_m=\lim_{i\rightarrow \infty} \int_\Omega |u_{k_i}|^{n+1}\, d\nu_m.\]
Observe that
\begin{equation*}
\begin{split}
\int_\Omega |u_{k_i}|^{n+1}\, d\nu - \int_\Omega |u|^{n+1}\, d\nu \leq \int_{\{x\in\Omega: \dist(x, \p\Omega)\leq 1/m\}} |u_{k_i}|^{n+1}\, d\nu + \int_\Omega (|u_{k_i}|^{n+1}-  |u|^{n+1}) \, d\nu_m.
\end{split}
\end{equation*}
On the right-hand side, we first let $i\to\infty$ and then $m\to\infty$ using \eqref{Distnu2} to discover
\[\limsup_{i\to\infty} \int_\Omega |u_{k_i}|^{n+1}\, d\nu \leq  \int_\Omega |u|^{n+1}\, d\nu.\]
This, together with the Fatou lemma,  implies \eqref{Inucont} as claimed. From this, we  have $E(u)=\lambda_0$.

\medskip

In all cases, $u$ is a minimizer of the functional
$F(u)= \frac{E(u)}{\Big(\int_\Omega |u|^{p+1}\, d\nu\Big)^{\frac{n+1}{p+1}}}$ over $\E(\Omega)$.

\noindent
We now show that a minimizer $u$ of $F$ solves \eqref{Distnu3} in the sense of Aleksandrov. For this, due to rescaling, we can assume that $\int_\Omega |u|^{p+1}\, d\nu=1$ and divide the proof into two steps.

\medskip
\noindent
{\it Step 2: We show that if  $v\in\E(\Omega)\cap C^{0, 1}(\overline{\Omega})$ then the following identity holds:}
\begin{equation}
\label{ELvK1}
\int_\Omega v [d\mu_u- \lambda_0(-u)^p\, d\nu]=0.\end{equation}
 Indeed, there is a constant $C>0$ such that $|v|\leq C\dist (\cdot,\p\Omega)$  in $ \overline{\Omega}$.
The convexity of $u$ gives 
\[|u(x)|\geq \frac{\text{dist}(x,\p\Omega)}{\text{diam}(\Omega)}\|u\|_{L^{\infty}(\Omega)} \quad \text{for all  }x\in\Omega.\]
It follows that for $-c<t< 1$ where $c=\|u\|_{L^{\infty}(\Omega)}/(C \cdot \text{diam}(\Omega))>0$, we have
\[u(x) + t v(x) =-|u(x)| - t|v(x)| \leq 0 \quad \text{for all  }x\in\Omega.\]
Let $\Gamma$ be the convex envelope as in Definition \ref{cvx_env_defn}. Consider the function
\[ f(t) =  \frac{E(\Gamma_{u+ tv})}{\Big(\int_\Omega (-u-tv)^{p+1}\, d\nu\Big)^{\frac{n+1}{p+1}}},\quad -c<t<1.\]
Then, by our assumption, $f(0)=\lambda$.
Since $\Gamma_{u + tv}\leq u + tv\leq 0$ in $\Omega$ when $-c<t< 1$, we have \[\int_\Omega (-u - tv)^{p+1}\, d\nu \leq \int_\Omega (-\Gamma_{u + tv})^{p+1}\, d\nu.\] Thus, recalling the minimality of $u$,
we have
\[f(t) \geq \frac{E (\Gamma_{u + tv})} {\Big(\int_\Omega (-\Gamma_{u + tv})^{p+1}\, d\nu\Big)^{\frac{n+1}{p+1}}} = F(\Gamma_{u+ tv})\geq F(u)=\lambda_0 \quad \text{when }-c<t< 1.\]
Therefore, $f$ attains its minimum value $\lambda_0 $ at $t=0$.

By Theorem \ref{envelopeED}, $f$ is differentiable at $0$ with
\[f'(0) =\frac{(n+ 1)\int_\Omega (-v)  \, d\mu_u- (n+ 1)\lambda_0 \Big(\int_\Omega (-u)^{p+1}\, d\nu\Big)^{\frac{n-p}{p+1}}\int_\Omega (-v) (-u)^p\, d\nu}{\Big(\int_\Omega (-u)^{p+1}\, d\nu\Big)^{\frac{n+1}{p+1}}}.\]
Since $f'(0)=0$ and $\int_\Omega |u|^{p+1}\, d\nu=1$, we obtain \eqref{ELvK1}.

\medskip
\noindent
{\it Step 3: $u$ solves the degenerate Monge--Amp\`ere equation \eqref{Distnu3} in the sense of Aleksandrov.}

Let $\varphi \in C_c^\infty(\Omega)$. Then, by Lemma \ref{2cvx}, we can find $\varphi_1,\varphi_2 \in\E(\Omega)\cap C^{0, 1}(\overline{\Omega})$ such that $\varphi=\varphi_1-\varphi_2$. Applying \eqref{ELvK1} to $\varphi_1$ and $\varphi_2$ and then subtracting, we deduce
\begin{equation*}
\int_\Omega \varphi [d \mu_u- \lambda_0 (-u)^p\, d\nu]=0.\end{equation*}
This holds for all $\varphi\in C_c^\infty(\Omega)$, so $u$ is an Aleksandrov solution of \eqref{Distnu3}.
The theorem is proved.
\end{proof}

We identify a general class of locally finite Borel measures satisfying the hypotheses of Theorem \ref{Dirdistp} for the case $p=n$. See also \cite[Proposition 3.16]{LZ}.
\begin{lem} \label{vamassho} Let $\Omega\subset\R^n$ be a bounded convex domain.
If $\nu=\mu_w$ where $w\in C(\overline{\Omega})$ is a nonzero convex function, then 
\eqref{Distnu2} holds.
\end{lem}
\begin{proof}
Subtracting a constant, we can assume $w\leq 0$ in $\Omega$.
Let $w_m\in C(\overline{\Omega})$ be standard approximants of $w$; see Remark \ref{Stappro}. Thus, the convex function $w_m$ solves
\begin{equation*}
   \mu_{w_m}=\chi_{\{x\in\Omega: \dist(x, \p\Omega)>1/m\}}\mu_w \quad\text{in} ~\Omega, \quad
w_m =w\quad\text{on}~\p \Omega,
\end{equation*}
$|w_m|\leq |w|$, and 
$\{w_m\}_{m=1}^{\infty}$ converges uniformly to $w$ in $\overline{\Omega}$.

Let $v\in \E(\Omega)$. Then,
Theorem \ref{BlemR} gives
\begin{equation*}
\begin{split}
\int_{\{x\in\Omega:\dist(x, \p\Omega)\leq 1/m\}} |v|^{n+1}\, d\nu &= \int_\Omega |v|^{n+1}\, (d\mu_w-d\mu_{w_m}) \\&\leq 
2(n+2)! \|w-w_m\|_{L^{\infty}(\Omega)} \|w\|^n_{L^{\infty}(\Omega)}  E(v).
\end{split}
\end{equation*}
When $m\to\infty$, the last quantity converges uniformly to $0$ on bounded subsets of $\E(\Omega)$.
\end{proof}
We can give another proof, without using convex envelopes, of Theorem \ref{Dirdistp} in the case $p=n$ that the minimizers with constraint are in fact Monge--Amp\`ere eigenfunctions.
\begin{rem}[Second proof of the energy characterization of the Monge--Amp\`ere eigenfunctions] This proof is an adaptation of the proof of \cite[Theorem 1.5]{L_scheme} from the case $\nu$ being the Lebesgue measure.
Consider $p=n$ in Theorem \ref{Dirdistp} and $u\in\E(\Omega)\setminus \{0\}$ such that
\[E(u) = \lambda[\Omega,\nu] \int_\Omega|u|^{n+1}\, d\nu,\]
for some $ \lambda[\Omega,\nu] >0$ defined by the variational formula
\[ \lambda[\Omega,\nu] = \min\bigg\{\frac{E(u)}{\int_\Omega |u|^{n+1}\, d\nu}: u\in \E(\Omega)\bigg\}.\]
By Lemma \ref{Distex2}, there exists a unique convex function $v\in \E(\Omega)$ satisfying
\begin{equation}\label{vsch}\mu_v=\lambda[\Omega,\nu] |u|^n\nu\quad \text{in }\Omega,\quad \text{and } v=0\quad\text{on }\p\Omega.\end{equation}
Multiplying  \eqref{vsch} by $|u|$, integrating over $\Omega$ and then using Theorem \ref{gCS}, we have
\[E(u)= \lambda[\Omega,\nu]\|u\|^{n+1}_{L^{n+1}(\Omega,d\nu)}=\int_{\Omega} |u|\, d \mu_v \leq 
[E(u)]^{\frac{1}{n+1}} [E(v)]^{\frac{n}{n+1}}.
\]
It follows that \[E(u)\leq E(v).\] 
Multiplying \eqref{vsch} by $|v|$, integrating over $\Omega$ and then using the H\"older inequality, we have
\begin{equation}
\label{uvEHol}
E(v)=\int_{\Omega} |v|\, d\mu_v = \lambda[\Omega,\nu]\int_{\Omega} |u|^n |v|d\nu\leq   \lambda[\Omega,\nu] \|u\|^n_{L^{n+1}(\Omega, d\nu)} \|v\|_{L^{n+1}(\Omega, d\nu)}.
\end{equation}
By the definition of $ \lambda[\Omega,\nu]$, we have 
$E(v) \geq  \lambda[\Omega,\nu] \|v\|_{L^{n+1}(\Omega, d\nu)}^{n+1}$.
It follows that 
\[\|v\|_{L^{n+1}(\Omega, d\nu)}\leq \|u\|_{L^{n+1}(\Omega, d\nu)}.\]
But then \eqref{uvEHol} implies
\[E(v) \leq \lambda[\Omega,\nu] \|u\|^{n+1}_{L^{n+1}(\Omega, d\nu)} =E(u).\]
Since $E(u)\leq E(v)$, the above inequality is only possible when $E(u)=E(v)$ and all the above inequalities are equalities. In particular, \eqref{uvEHol} is an equality and $\|v\|_{L^{n+1}(\Omega, d\nu)}= \|u\|_{L^{n+1}(\Omega, d\nu)}$. Thus, there must be a constant $c>0$ such that
$|v|= c|u|$ almost everywhere with respect to $\nu$. Now, we must have $c=1$ so $v=u$. Therefore, recalling \eqref{vsch}, we conclude that $u$ satisfies $\mu_u=\lambda[\Omega,\nu] |u|^n\nu$, so it is a Monge--Amp\`ere eigenfunction associated with $\nu$.
\end{rem}
Using the method of the proof of Theorem \ref{Dirdist} and the comparison principle in Theorem \ref{compnu}, we can improve the existence part of Theorem \ref{Dirdistp} when $p\in (0, n)$ to complete the proof of Theorem \ref{Dist0pn} (iii). In the following theorem, the measure $\nu$ can be more singular.
\begin{thm}\label{MApnu}
Let $\Omega\subset\R^n$ be a bounded convex domain. Let $p\in (0, n)$ and
$\nu$ be a Borel measure on $\Omega$ satisfying 
\[\nu(\Omega)>0,\quad \int_\Omega \dist(\cdot,\p\Omega)d\nu\leq C<\infty.\]
Then, there exists a nonzero convex Aleksandrov solution 
$u\in C(\overline{\Omega})$ to
\begin{equation}
\label{0pnnu}
   \det D^{2} u=|u|^p\nu \quad\text{in} ~\Omega, \quad
u =0\quad\text{on}~\p \Omega.
\end{equation}
\end{thm}

\begin{proof}
For $\e>0$, let \[\Omega_\e:= \{x\in\Omega: \dist(x,\p\Omega)>\e\},\quad \nu_\e =\chi_{\Omega_\e} \nu.
\]
Let $\e_0>0$ is small so that $\Omega':=\Omega_{2\e_0}\neq\emptyset$ and  $\nu(\Omega')>0$.
Consider now $0<\e<\e_0$. Then
\[\int_\Omega [\dist(\cdot,\p\Omega)]^{\frac{p+1}{n+1}}d\nu_\e=\int_{\Omega_\e} [\dist(\cdot,\p\Omega)]^{\frac{p+1}{n+1}} d\nu \leq  \e^{\frac{p-n}{n+1}}   \int_{\Omega_\e} \dist(\cdot,\p\Omega)d\nu \leq \e^{\frac{p-n}{n+1}} C.\]
By Theorem \ref{Dirdistp}, there exists
 a nonzero convex function $u_\e\in \E(\Omega)$ satisfying
\begin{equation}
\label{muep}
   \mu_{u_\e}=|u_\e|^p\nu_\e \quad\text{in} ~\Omega, \quad
u_\e =0\quad\text{on}~\p \Omega.
\end{equation}
We will let $\e\to 0$ in \eqref{muep} to obtain a nonzero convex Aleksandrov solution to \eqref{0pnnu}. For this, we need to bound $\|u_\e\|_{L^{\infty}(\Omega)}$ from above and below by positive constants independent of $\e$.

\medskip
\noindent
{\it Step 1. Uniform bound independent of $\e$.} 
We show that \[\|u_\e\|_{L^{\infty}(\Omega)}\leq C_0.\]
Pick $x_0\in\Omega$ where $u_\e$ attains its maximum value. 
By Theorem \ref{AJ_thm}, we have
\[|u_\e(x_0|^n  \leq C(n,\diam(\Omega)) \int_{\Omega} \dist(\cdot,\p\Omega) d\mu_{u_\e} \leq C(n,\diam(\Omega)) |u_\e(x_0)|^p\int_{\Omega} \dist(\cdot,\p\Omega) d\nu.\]
Now, recalling $0<p<n$, the uniform bound of $\|u_\e\|_{L^{\infty}(\Omega)}$ independent of $\e$ follows.

\medskip
\noindent
{\it Step 2. Lower bound for $\alpha_\e:=\|u_\e\|_{L^{\infty}(\Omega)}$ when $\e\leq \e_0$.}
Let $w_\e=u_\e/\alpha_\e$. Then \[\|w_\e\|_{L^{\infty}(\Omega)}=1,\quad \mu_{w_\e}=\alpha_\e^{p-n} |w_\e|^p\nu_\e.\]
By the convexity of $w_\e$ and the gradient estimate for convex functions, we have
\[|Dw_\e| \leq \frac{\|w_\e\|_{L^{\infty}(\Omega)}}{\dist(\Omega',\p\Omega)} \leq C_1,\quad |w_\e|\geq \frac{\dist(\cdot,\p\Omega)}{\text{diam}(\Omega)}\|w_\e\|_{L^{\infty}(\Omega)}\geq C_2\quad \text{in }\Omega'.\]
Thus, for $\e<\e_0$, we have
\[\alpha_\e^{p-n}= \frac{\mu_{w_\e}(\Omega')}{\int_{\Omega'} |w_\e|^p d\nu_\e} \leq \frac{|B_{C_1}(0)|}{C_2^p\int_{\Omega'}  d\nu} \leq C(\nu, p, n,\Omega)<\infty.\]
Again, recalling $0<p<n$, this gives a positive lower bound  \[\|u_\e\|_{L^{\infty}(\Omega)}=\alpha_\e\geq c(\nu, p, n,\Omega)>0.\]
\medskip
\noindent
{\it Step 3. Passage to the limit $\e\to 0$.}
Observe that if $\e>\e'$, then $\nu_{\e'}> \nu_\e$, so 
$\mu_{u_\e}\leq |u_\e|^p \nu_{\e'}$ in $\Omega$, and 
 the comparison principle in Theorem \ref{compnu} for \eqref{muep} gives $u_\e\geq u_{\e'}$ and $ \mu_{u_\e}\leq \mu_{u_{\e'}}$.

Let $x_0\in\Omega$. Then, by Theorem \ref{AJ_thm}, we have for $0<\e'<\e$
\begin{equation}
\label{ueen}
\begin{split}
 |u_\e(x_0)-u_{\e'}(x_0)|^n& \leq C(n, \diam(\Omega))\int_{\Omega} \dist(\cdot,\p\Omega)(d\mu_{u_\e'}-d\mu_{u_\e})\\&= C(n, \diam(\Omega))\int_{\Omega}  \dist(\cdot,\p\Omega) (|u_{\e'}|^p d\nu_{\e'}-|u_\e|^p d\nu_{\e}).
 \end{split}
 \end{equation}
By writing 
\[ |u_{\e'}|^p d\nu_{\e'}-|u_\e|^p d\nu_{\e}= |u_{\e'}|^p (d\nu_{\e'}-d\nu_{\e}) + (|u_{\e'}|^p -|u_\e|^p) d\nu_{\e},\]
and using the uniform bounds from below and above on $\|u_\e'\|_{L^{\infty}(\Omega)}$ and  $\|u_\e\|_{L^{\infty}(\Omega)}$, we find
\[
\int_{\Omega}  \dist(\cdot,\p\Omega) (|u_{\e'}|^p d\nu_{\e'}-|u_\e|^p d\nu_{\e})\leq  C\int_{\Omega_{\e'}\setminus\Omega_\e}  \dist(\cdot,\p\Omega) d\nu + C\int_{\Omega}  \dist(\cdot,\p\Omega) |u_{\e'} -u_\e| d\nu.
\]
 Clearly
 \[\int_{\Omega_{\e'}\setminus\Omega_\e}  \dist(\cdot,\p\Omega) d\nu\to 0\quad\text{as } 0<\e'<\e\to 0^+.\]
 Let $\alpha>0$. Then, there exists $\delta=\delta (\alpha, n,\nu,\Omega))>0$ such that 
 \[C\int_{\Omega\setminus \Omega_\delta}  \dist(\cdot,\p\Omega) |u_{\e'} -u_\e| d\nu \leq CC_0\int_{\Omega\setminus \Omega_\delta}  \dist(\cdot,\p\Omega) d\nu<\alpha/2.\]
 Thus, using the uniform convergence of $\{u_\e\}_{\e>0}$ on compact subsets of $\Omega$, we deduce the existence of $\tilde \e=\tilde\e(\alpha, n,\nu,\Omega)>0$ such that 
 \[C\int_{\Omega}  \dist(\cdot,\p\Omega) |u_{\e'} -u_\e| d\nu < \alpha/2+ C\int_{\Omega_\delta}  \dist(\cdot,\p\Omega) |u_{\e'} -u_\e| d\nu<\alpha\,\, \text{when } 0<\e'<\e<\tilde\e. \]
 It follows from \eqref{ueen} and the above estimates that $\{u_\e\}_{\e}$ is uniformly Cauchy in $C(\overline{\Omega})$. 
Thus, letting $\e\to 0$ in \eqref{muep} gives a convex solution 
$u\in C(\overline{\Omega})$ to \eqref{0pnnu}.
The theorem is proved.
\end{proof}

\subsection{Uniqueness of the Monge--Amp\`ere eigenvalue problem in energy class}

In this subsection, we study the uniqueness property of the Monge--Amp\`ere eigenvalue problem.

The following lemma guarantees the uniqueness of the Monge--Amp\`ere eigenvalue associated with eigenfunctions in the energy class.
\begin{lem}[Uniqueness of the Monge--Amp\`ere eigenvalue with finite energy  eigenfunctions]
\label{MAlu} Let $\nu$ be a Borel measure on a bounded convex domain $\Omega\subset\R^n$  with $\nu(\Omega)>0$.
 Assume $u, v\in \E(\Omega)$ are
 nonzero convex functions such that
 \[\mu_u= \lambda |u|^n\nu,\quad \mu_v \geq \Lambda |v|^n\nu\quad\text{in }\Omega.
 \]
 Then $\lambda\geq \Lambda$.
\end{lem}
\begin{proof} As in \cite[Proposition 5.6]{L}, we use integration by parts.
By Proposition \ref{IBPf}, we have
\[\int_\Omega \lambda |u|^n|v|\, d\nu=\int_\Omega |v|\, d\mu_n[u,\cdots, u]=\int_\Omega |u|\, d\mu_n[v,u,\cdots, u]. \]
By Theorem \ref{mMAnu0},
\[\mu_n[v,u,\cdots, u] \geq (\Lambda |v|^n)^{\frac{1}{n}} (\lambda |u|^n)^{\frac{n-1}{n}}\nu.\]
Thus
\[\int_\Omega \lambda |u|^n|v|\, d\nu\geq \int_\Omega |u| (\Lambda |v|^n)^{\frac{1}{n}} (\lambda |u|^n)^{\frac{n-1}{n}}\nu=  \int_\Omega \Lambda^{\frac{1}{n}}\lambda^{\frac{n-1}{n}} |v||u|^n d\nu.\]
It follows that $\lambda\geq\Lambda$.
\end{proof}

We show that the infimum of the Rayleigh quotient in the energy class is the monotone limit of Monge--Amp\`ere eigenvalues for truncated measures. This proves Theorem \ref{EVP1} (i).
\begin{thm}[Monotonicity]
\label{monolam} 
Let $\nu$ be a locally finite Borel measure on a bounded convex domain $\Omega\subset\R^n$ with $\nu(\Omega)>0$. 
 For each $m\in\N$, let 
$\nu_m= \chi_{\{x\in\Omega: \dist(x, \p\Omega)>1/m\}}\nu$.
Let $u_m\in\E(\Omega)$ be a nonzero solution to the Monge--Amp\`ere eigenvalue problem
\begin{equation}
\label{EVPlamm}
   \mu_{u_m}=\lambda_m |u_m|^n\nu_m \quad\text{in} ~\Omega, \quad
u_{m} =0\quad\text{on}~\p \Omega.
\end{equation}
Then, the sequence $\{\lambda_m\}_{m=1}^{\infty}$ is nonincreasing and 
\[\lim_{m\to\infty }\lambda_m=\lambda[\Omega,\nu]:=\inf\bigg\{\frac{E(u)}{\int_\Omega |u|^{n+1}\, d\nu}: u\in \E(\Omega)\bigg\}.\]
\end{thm}
\begin{proof} 
Since $\nu_m$ is compactly supported, combining Theorem \ref{Dirdistp} with Lemma \ref{MAlu}, we obtain the existence of $(\lambda_m, u_m)\in (0,\infty)\times \E(\Omega)\setminus\{0\}$ solving 
\eqref{EVPlamm}.
Moreover,  the eigenvalue $\lambda_m$ is unique and is given by a variational characterization.
Clearly
\[\lambda_m = \frac{E(u_m)}{ \int_\Omega |u_m|^{n+1}\, d\nu_m}\geq \frac{E(u_m)}{ \int_\Omega |u_m|^{n+1}\, d\nu}\geq\lambda[\Omega,\nu].\]
For $m\in \N$, denote $\Omega_m:= \{x\in\Omega: \dist(x, \p\Omega)>1/m\}$.

Let $m<m'$. Then, $\Omega_m\subset\Omega_{m'}$, and by Theorems \ref{IBPf} and \ref{mMAnu}, we can estimate
\begin{equation}
\label{lammm}
\begin{split}
\int_\Omega\lambda_m |u_m|^n  |u_{m'}|\, d\nu_m &=\int_\Omega |u_{m'}|\, d\mu_{u_m}=\int_\Omega |u_m|\,d\mu_n[u_{m'}, u_m, \cdots, u_m]\\
&\geq \int_\Omega |u_m| (\lambda_{m'} |u_{m'}|^n \chi_{\Omega_{m'}})^{\frac{1}{n}}(\lambda_m |u_m|^n \chi_{\Omega_m})^{\frac{n-1}{n}}\, d\nu\\
&\geq \int_\Omega \lambda_{m'}^{\frac{1}{n}} \lambda_m^{\frac{n-1}{n}}  |u_m|^n  |u_{m'}|\, d\nu_m.
\end{split}
\end{equation}
It follows that $\lambda_m\geq \lambda_{m'}$, so the sequence $\{\lambda_m\}_{m=1}^{\infty}$ is nonincreasing and has a finite limit.

On the other hand, for any fixed $v\in \E(\Omega)$, the Monotone Convergence Theorem gives
\[ \lim_{m\to\infty}\int_\Omega |v|^{n+1}\, d\nu_m=  \int_\Omega |v|^{n+1}\, d\nu.\]
Recall that the eigenvalue $\lambda_m$ is given by a variational characterization.
It follows that
\[\lambda_m = \frac{E(u_m)}{ \int_\Omega |u_m|^{n+1}\, d\nu_m} \leq  \frac{E(v)}{ \int_\Omega |v|^{n+1}\, d\nu_m}.\]
Letting $m\to\infty$, we find
\[\lim_{m\to\infty} \lambda_m  \leq  \frac{E(v)}{ \int_\Omega |v|^{n+1}\, d\nu}.\]
Taking the infimum over $v\in \E(\Omega)$ gives
$\lim_{m\to\infty} \lambda_m \leq \lambda[\Omega,\nu]$.
Thus, we must have 
\[\lim_{m\to\infty} \lambda_m=\lambda[\Omega,\nu].\]
The theorem is proved.
\end{proof}
We use truncation and the variational characterization of the Monge--Amp\`ere eigenvalues to establish a Poincar\'e-type inequality. The following proposition is an extension of  Lu--Zeriahi \cite[Corollary 3.7]{LZ} for $u$ in the energy class $\E(\Omega)$ to general convex functions $u\in C(\overline{\Omega})$.

\begin{prop}[Poincar\'e-type inequality]
 \label{unmuE}
Let $\Omega$ be a bounded convex domain in $\R^n$, and $E$ be as in \eqref{Edef}. Let $u\in C(\overline{\Omega})$ be a nonzero convex function on $\Omega$ with $u=0$ on $\p\Omega$. Then 
\[\int_\Omega |v|^{n+1} |u|^{-n}\, d\mu_u\leq E(v)\quad \mbox{for all $v\in\E(\Omega)$}.\]
Consequently, if $\nu$ is a Borel measure with $\nu(\Omega)>0$ and $(\lambda_0, w)\in (0,\infty)\times C(\overline{\Omega})\setminus \{0\}$ is a subsolution to 
the Monge--Amp\`ere eigenvalue problem with measure $\nu$; that is,
\begin{equation}
\label{EVPPI}
   \mu_{w}\geq \lambda_0 |w|^n\nu \quad\text{in} ~\Omega, \quad
w=0\quad \text{on}~\p \Omega,
\end{equation}
then
\[\inf\bigg\{\frac{E(v)}{\int_\Omega |v|^{n+1}\, d\nu}: v\in \E(\Omega)\bigg\}\geq \lambda_0.\]
\end{prop}
\begin{proof} Let $u_m\in C(\overline{\Omega})$ be standard approximants of $u$ (see Remark \ref{Stappro}),  so $u_m$ solves
\begin{equation*}
   \mu_{u_m}=\chi_{\{x\in\Omega: \dist(x, \p\Omega)>1/m\}}\mu_u \quad\text{in} ~\Omega, \quad
u_m =0\quad \text{on}~\p \Omega.
\end{equation*}
Note that  $0> u_m\geq u$ in $\Omega$, and $u_m$ converges uniformly to $u$ in $\overline{\Omega}$.  Clearly $u_m\in \E(\Omega)$ and 
$(1, u_m)\in (0,\infty)\times \E(\Omega)\setminus\{0\}$ is  a solution of the Monge--Amp\`ere eigenvalue problem on $\Omega$ with measure $\nu_m:=(-u_m)^{-n} \mu_{u_m}$. 
Since $\nu_m$ is compactly supported, by combining Theorem \ref{Dirdistp} with Lemma \ref{MAlu}, we find that the Monge--Amp\`ere eigenvalue $1$ is unique and is given by a variational characterization.
Consequently,
 we have for all $v\in\E(\Omega)$
\[\int_\Omega |v|^{n+1} \, d\nu_m\leq E(v). \]
Since $|u_m|  \leq |u|$ and $\nu_m:=|u_m|^{-n} \mu_{u_m}$, this implies
\[\int_\Omega |v|^{n+1} |u|^{-n}\, d\mu_{u_m}\leq E(v).\]
Since $\mu_{u_m}$ converges weakly to $\mu_u$, letting $m\to\infty$ gives the claimed Poincar\'e-type inequality.

For the consequence, we just use $|w|^{-n}\mu_w\geq \lambda_0 \nu$ in $\Omega$, so that for any $v\in \E(\Omega)$, we have
\[E(v) \geq \int_\Omega |v|^{n+1} |w|^{-n} \, d\mu_w\geq \lambda_0 \int_\Omega |v|^{n+1}\, d\nu. \]
The proposition is proved.
\end{proof}

When the measure $\nu$ satisfies a Poincar\'e-type inequality in the energy class $\E(\Omega)$, we show that the infimum of the Rayleigh quotient in the energy class is the limit of an inverse iterative scheme, originally introduced by Abedin--Kitagawa \cite{AK}. This proves Theorem \ref{EVP1} (iii).
\begin{prop}[Infimum of Rayleigh quotient and inverse iterative scheme]
\label{allmono}
Let $\Omega$ be a bounded convex domain in $\R^n$. Let $\nu$ be a locally finite Borel measure with $\nu(\Omega)>0$ and satisfies the Poincar\'e-type inequality: There exists a constant $c>0$ such that
\[R_\nu(u)\geq c\quad \text{for all }u\in \E(\Omega)\quad \text{where }R_\nu (u):= \frac{E(u)}{\int_\Omega |u|^{n+1}\, d\nu}\equiv \frac{\int_\Omega|u|\, d\mu_u}{\int_\Omega |u|^{n+1}\, d\nu}. \]
 Fix $u_0\in \E(\Omega)\setminus\{0\}$. 
Consider the following iterative scheme  
\begin{equation}\label{IIS}
\mu_{u_{k+1}} = R_\nu(u_k) |u_k|^n\nu\quad\quad \text{in } \Omega, \quad
u_{k+1} = 0  \quad \text{on } \partial \Omega,
\end{equation}
which has a unique solution in $\E(\Omega)$ for each nonnegative integer $k$ by Lemma \ref{Distex2}. Then for all $k\geq 0$, we have
\begin{enumerate}
\item[(i)] $E(u_k)\leq E(u_{k+1})$;  \, (ii) $\|u_k\|_{L^{n+1}(\Omega,d\nu)} \leq \|u_{k+1}\|_{L^{n+1}(\Omega,d\nu)}$; \, (iii) $R_\nu(u_{k+1}) \leq R_\nu(u_k)$;
\item[(iv)] 
$\lim_{k\to\infty} R_\nu(u_k) =\lambda[\Omega, \nu]:=\inf\Big\{R_\nu(u): u\in \E(\Omega)\Big\}$. 
\end{enumerate}
\end{prop}
\begin{proof} Note that parts (i)--(iii) are the real counterparts of Lemma 5.1 in Lu--Zeriahi \cite{LZ}. 

Multiplying both sides of the first equation of (\ref{IIS}) by $|u_k|$, integrating over $\Omega$ and then using Theorem \ref{gCS}, we have
\[E(u_k)= R_\nu(u_k) \|u_k\|^{n+1}_{L^{n+1}(\Omega,d\nu)}=\int_{\Omega} |u_k|d \mu_{u_{k+1}}
 \leq 
[E(u_k)]^{\frac{1}{n+1}} [E(u_{k+1})]^{\frac{n}{n+1}}.
\]
Then, part (i) easily follows.  

Multiply \eqref{IIS} by $|u_{k+1}|$, integrate over $\Omega$ and then use the H\"older inequality to obtain
\begin{equation*}
\begin{split}
E(u_{k+1})=\int_{\Omega} |u_{k+1}|\, d\mu_{u_{k+1}} &= R_\nu(u_k) \int_{\Omega} |u_k|^n |u_{k+1}|d\nu \\&\leq   R_\nu(u_k) \|u_k\|^n_{L^{n+1}(\Omega, d\nu)} \|u_{k+1}\|_{L^{n+1}(\Omega, d\nu)}.
\end{split}
\end{equation*}
Rewrite this as one of the following inequalities:
\[\frac{E(u_{k+1})}{ \|u_{k+1}\|_{L^{n+1}(\Omega, d\nu)}} \leq \frac{E(u_k)}{ \|u_k\|_{L^{n+1}(\Omega, d\nu)}} \leq \frac{E(u_0)}{ \|u_0\|_{L^{n+1}(\Omega, d\nu)}}= C_0(n, u_0, \Omega,\nu),\] 
or\[R_\nu(u_{k+1}) \|u_{k+1}\|^n_{L^{n+1}(\Omega,d\nu)}  \leq R_\nu(u_k) \|u_k\|^n_{L^{n+1}(\Omega, d\nu)}.  \]
The first inequality combined with (i) gives  (ii) while the second inequality combined with (ii) gives (iii). It follows that there exists a limit
$\lim_{k\to\infty} R_\nu(u_k) =\Lambda$.

\medskip
We now prove part (iv).  Note that $\lambda[\Omega, \nu]\geq c>0$. Therefore,
\[E(u_k)\leq C_0  \|u_k\|_{L^{n+1}(\Omega, d\nu)} \leq C_0[\lambda[\Omega, \nu]^{-1} E(u_k)]^{\frac{1}{n+1}}.\]
This implies that
\[E(u_k) +  \|u_k\|_{L^{n+1}(\Omega, d\nu)}\leq C_1 (n, u_0,\Omega, \nu)\quad \text{for all } k\geq 0.\]
For each $m\in\N$, let 
\[\nu_m= \chi_{\{x\in\Omega: \dist(x, \p\Omega)>1/m\}}\nu\equiv \chi_{\Omega_m}\nu.\]
 Let $v_m\in\E(\Omega)$ be a nonzero solution to the Monge--Amp\`ere eigenvalue problem
\begin{equation*}
   \mu_{v_m}=\lambda_m |v_m|^n\nu_m \quad\text{in} ~\Omega, \quad
v_{m} =0\quad \text{on}~\p \Omega.
\end{equation*}
Then, by Theorem \ref{monolam},
$\lim_{m\to\infty} \lambda_m=\lambda[\Omega, \nu].
$
For each $k$, we have
\[R_\nu(u_k)\geq \Lambda \geq \lambda[\Omega, \nu].\]
To prove (iv), we will show that $\lambda_m\geq \Lambda$. Observe, using Theorems \ref{IBPf} and \ref{mMAnu}, that
\begin{equation*}
\begin{split}
\int_\Omega \lambda_m |v_m|^n |u_{k+1}|d\nu_m =\int_\Omega |u_{k+1}| \mu_{v_m} &=\int_\Omega  |v_m|\, d\mu_n[u_{k+1}, v_m, \cdots, v_m]\\
&\geq \int_\Omega  |v_m| (R_\nu(u_k) |u_k|^n)^{\frac{1}{n}}  (\lambda_m |v_m|^n\chi_{\Omega_m})^{\frac{n-1}{n}}\, d\nu\\
&\geq \int_\Omega (\Lambda/\lambda_m)^{\frac{1}{n}}  \lambda_m   |v_m|^m |u_k|  \, d\nu_m.
\end{split}
\end{equation*}
Iterating, we find for each $k\in\N$ that 
\begin{equation*}
\begin{split}  \int_\Omega (\Lambda/\lambda_m)^{\frac{k}{n}}    |v_m|^m |u_0|  \, d\nu_m \leq \int_\Omega  |v_m|^n |u_k|d\nu_m &\leq 
 \|v_m\|^n_{L^{n+1}(\Omega, d\nu_m)} \|u_k\|_{L^{n+1}(\Omega, d\nu_m)} \\&\leq C_1 \|v_m\|^n_{L^{n+1}(\Omega, d\nu_m)}.
 \end{split}
\end{equation*}
It follows that 
\[\lambda_m\geq \Lambda\bigg( \frac{ \int_\Omega    |v_m|^m |u_0|  \, d\nu_m }{C_1 \|v_m\|^n_{L^{n+1}(\Omega, d\nu_m)}}\bigg)^{\frac{n}{k}}.\]
Letting $k\to\infty$ shows
 $\lambda_m\geq\Lambda$. This completes the proof of part (iv) and the theorem.
\end{proof}

We now show that if the Monge--Amp\`ere eigenvalue  problem has a solution in the energy class, then the eigenvalue is uniquely given by the infimum of the Rayleigh quotient. Our result here is the real counterpart of Theorem 3.5 in Lu--Zeriahi \cite{LZ}. Our proof strategy is different, however.
\begin{thm}\label{varchac} Let $\nu$ be a locally finite Borel measure on a bounded convex domain $\Omega\subset\R^n$ with $\nu(\Omega)>0$. 
If $(\lambda, u)\in (0,\infty)\times \E(\Omega)\setminus\{0\}$ solves 
  the Monge--Amp\`ere eigenvalue problem
 \[\mu_u=\lambda |u|^n\nu\quad\text{in }\Omega,\quad u=0\quad \text{on }\p\Omega,\]
then $\lambda$ is given by the variational characterization
\[\lambda=\lambda[\Omega,\nu]:=\inf\bigg\{\frac{E(u)}{\int_\Omega |u|^{n+1}\, d\nu}: u\in \E(\Omega)\bigg\}.\]
\end{thm}

\begin{proof} 
Obviously, $\lambda[\Omega,\nu]\leq \lambda$.
 For each $m\in\N$, let 
$\nu_m= \chi_{\{x\in\Omega: \dist(x, \p\Omega)>1/m\}}\nu$.
Let $u_m\in\E(\Omega)$ be a nonzero solution to the Monge--Amp\`ere eigenvalue problem
\begin{equation*}
   \mu_{u_m}=\lambda_m |u_m|^n\nu_m \quad\text{in} ~\Omega, \quad
u_{m} =0\quad \text{on}~\p \Omega.
\end{equation*}
Then, by Theorem \ref{monolam},
$\lim_{m\to\infty} \lambda_m=\lambda[\Omega,\nu]$.
Moreover, replacing $u_{m'}$ by $u$ in the proof of the monotonicity of $\{\lambda_m\}_{m=1}^{\infty}$ (see \eqref{lammm}), we find
\[\lambda_m\geq \lambda.\]
Thus $\lambda[\Omega,\nu]\geq \lambda$. Therefore, we must have $\lambda[\Omega,\nu]=\lambda$. The theorem is proved.
\end{proof}
An immediate consequence of Proposition \ref{unmuE} and Theorem \ref{varchac} in the case the Monge--Amp\`ere eigenvalue problem with measure $\nu$ has nonzero convex solutions in the energy class is the equality of the infimum of the Rayleigh quotient and the supremum of the subeigenvalues. 
\begin{cor}[Spectral characterization of the Monge--Amp\`ere eigenvalue] 
\label{BPcor}
Let $\nu$ be a locally finite Borel measure on a bounded convex domain $\Omega\subset \R^n$ with $\nu(\Omega)>0$. Assume 
the Monge--Amp\`ere eigenvalue problem 
\begin{equation*}
   \mu_{w}= \lambda |w|^n\nu \quad\text{in} ~\Omega, \quad
w=0\quad \text{on}~\p \Omega
\end{equation*}
has a solution $(\lambda, w)\in (0,\infty)\times \E(\Omega)\setminus \{0\}$. 
Then
\[\lambda[\Omega,\nu]:=\inf\bigg\{\frac{E(v)}{\int_\Omega |v|^{n+1}\, d\nu}: v\in \E(\Omega)\bigg\}=\sup_{\Sigma_\nu} \Lambda,\]
where 
\begin{equation*}
\Sigma_\nu:=\Big\{\Lambda\in\R: \mbox{there exists a convex $u\in  C(\overline{\Omega})\setminus\{0\}$, $u=0$ on $\p\Omega$, such that $\mu_u\geq \Lambda|u|^n\nu$}\Big\}.
\end{equation*}
\end{cor}
\begin{rem}
Corollary \ref{BPcor} extends the result of Birindelli--Payne \cite[Theorem 6.6]{BP} where the authors considered the case of smooth, uniformly bounded convex domains $\Omega\subset \R^n$ and Lebesgue measure $\nu=d\L^n$. In this case, they showed that
\[\lambda[\Omega,d\L^n]=\sup_{\Sigma'} \Lambda,\]
where 
\begin{multline*}
\Sigma':=\Big\{\Lambda\in\R: \mbox{there exists a convex $u\in  C(\overline{\Omega})\setminus\{0\}$, $u=0$ on $\p\Omega$}\\ \mbox{such that $\det D^2 u\geq \Lambda|u|^n$ in the viscosity sense}\Big\}.
\end{multline*}
The subsolution $\det D^2 u\geq \Lambda|u|^n$ in the viscosity sense above is understood as follows. For each $x_0\in\Omega$ and each function $\varphi$ which is $C^2$ at $x_0$, one has that: \[\mbox{if $u-\varphi$ has a local maximum at $x_0$, then $\det D^2\varphi(x_0)\geq \Lambda|u(x_0)|^n$}.\]
Since $|u|>0$ in $\Omega$, a simple adaptation of the proof of Proposition 7.11 in \cite{Lbook} shows that the subsolution $u$ in the viscosity sense is also a subsolution in the sense of Aleksandrov; that is
$\mu_u\geq \Lambda|u|^n$.
Therefore, Corollary \ref{BPcor} extends  Birindelli--Payne \cite[Theorem 6.6]{BP} to the case of singular Borel measures. 
\end{rem}

Turning to eigenfunctions, we first show that the convex envelope of the difference between a supersolution and a subsolution to the Monge--Amp\`ere eigenvalue equation with the same eigenvalue is again a supersolution. Our result here is the real counterpart of Lemma 3.3 in Lu--Zeriahi \cite{LZ} where the functions involved were required to be in the energy class.
\begin{lem}[Monge--Amp\`ere eigenfunctions and convex envelopes] 
\label{EFconv}
Let $\nu$ be a Borel measure on a bounded convex domain $\Omega\subset\R^n$ with $\nu(\Omega)>0$.
 Let $\lambda>0$ and $u, v\in C(\overline{\Omega})$ be convex functions vanishing on $\p\Omega$
 such that 
 \[\mu_u \leq \lambda |u|^n\nu,\quad \mu_v \geq \lambda |v|^n\nu.\]
 Let $w:=\Gamma_{\min\{u-v, 0\}}$ where $\Gamma$ denotes the convex envelope (see Definition \ref{cvx_env_defn}). Then 
\[ \mu_w \leq \lambda |w|^n\nu.\]
\end{lem}
\begin{proof} If $u\geq v$ in $\Omega$, then $w\equiv 0$, and the conclusion is obvious. Assume now $\{u<v\}\neq\emptyset$.

 Since $w$ is convex, $w=0$ on $\p\Omega$ and $w\leq \min\{u-v, 0\}$ in $\Omega$, so $w<0$ somewhere in $\Omega$, we have $w<0$ in $\Omega$.
 By Lemma \ref{Enlinear}, $\mu_w$ is concentrated on the contact set \[\mathcal{C}= \{w= \min\{u-v, 0\}\}\cap\Omega = \{w= u-v\}.\]
 Thus
 $\mu_w =\chi_{\mathcal{C}}\mu_w$.
  Since $\mu_w+\mu_v\leq \mu_{w+ v}$ (see \cite[Lemma 3.10]{Lbook}), we find
  $\mu_w \leq \chi_{\mathcal{C}}\mu_{w+ v}$.
  
  \medskip
 Note that $u\geq w+ v$ in $\Omega$ and $u=w+ v$ on $\mathcal{C}$, so for each $x\in \mathcal{C}$, we have $\p (w+ v)(x)\subset \p u(x)$. Hence
 $\chi_{\mathcal{C}}\mu_{w+ v} \leq \chi_{\mathcal{C}}\mu_u$.
Therefore
\[\mu_w\leq \chi_{\mathcal{C}}\mu_u\leq \chi_{\mathcal{C}}\lambda|u|^n\nu.\]
Let $\mu_w = \lambda f^n\nu$,
where $f\geq 0$ and $f^n\in L^1_{\text{loc}}(\Omega, d\nu)$.
Then 
$ f^n\leq \chi_{\mathcal{C}}|u|^n$.
By Theorem \ref{mMAnu},
\[\mu_{v+ w}\geq \lambda (|v| + f)^n\nu.\]
Hence
\[\chi_{\mathcal{C}}\lambda (-u)^n\nu\geq  \chi_{\mathcal{C}}\mu_u\geq \chi_{\mathcal{C}}\mu_{w+ v}\geq \chi_{\mathcal{C}} \lambda (-v + f)^n\nu.\]
Thus, almost everywhere with respect to $\nu$, we have 
$\chi_{\mathcal{C}} f \leq \chi_{\mathcal{C}} (v-u)$. Therefore
\[\mu_w =\chi_{\mathcal{C}}\mu_w =\chi_{\mathcal{C}} \lambda f^n\nu\leq\chi_{\mathcal{C}}\lambda (v-u)^n\nu =\chi_{\mathcal{C}}\lambda (-w)^n\nu \leq \lambda (-w)^n\nu.\]
The lemma is proved.
\end{proof}
\begin{rem} In Lemma \ref{EFconv}, we can replace the exponent $n$ in $|u|^n$ and $|v|^n$ by $p\in (0, n]$.
\end{rem}

Finally, we establish the uniqueness of the Monge--Amp\`ere eigenfunctions in the class of globally continuous, convex functions when one of the eigenfunctions is in the energy class. When all functions are in the energy class, our result here is the real counterpart of Theorem 3.8 in Lu--Zeriahi \cite{LZ}.
\begin{thm}[Uniqueness of the Monge--Amp\`ere eigenfunctions]
\label{MAfu}
Let $\nu$ be a Borel measure on a bounded convex domain $\Omega\subset\R^n$ with $\nu(\Omega)>0$.
 Let $u\in\E(\Omega)\setminus\{0\}$ and $v\in C(\overline{\Omega})\setminus\{0\}$ be convex functions vanishing on $\p\Omega$ such that 
 \[\mu_u =\lambda |u|^n\nu,\quad \mu_v = \Lambda |v|^n\nu,\]
where $\Lambda\geq \lambda$. Then, there exists some constant $c>0$ such that \[u= cv,\]
so $\lambda=\Lambda$ and $v\in\E(\Omega)$.
\end{thm}
\begin{proof} In vew of Theorem \ref{varchac}, $\lambda$ has a  variational characterization. Let $a>0$ be a constant such that $\{u<av\}\neq \emptyset$. We show that
$av\geq u \quad\text{in }\Omega$.

\medskip
Let $w:=\Gamma_{\min\{u-av, 0\}}$. Since $w=0$ on $\p\Omega$ and $w\geq u$, we have $w\in\E(\Omega)$, by Proposition \ref{latticelem}. Then $w<0$ in $\Omega$ and, by Lemma \ref{EFconv},
$\mu_w \leq \lambda |w|^n\nu$.
Multiplying with $|w|$  and integrating, we find
\[E(w)=\int_\Omega |w|\, d\mu_w \leq \lambda \int_\Omega |w|^{n+1}d\nu.\]
Thus, by the variational characterization of $\lambda$, we must have equality and therefore
\[ |w|\mu_w = \lambda |w|^{n+1}\nu\quad\text{in }\Omega,\]
so
\[ \mu_w= \lambda |w|^n\nu.\]
Since $\mu_w$ is supported on the set $\{w= u-av\} \subset \{u-av<0\}$, we have
\[\mu_w(\{av<u\})=0.\]
From the above equation for $\mu_w$ and $\lambda>0, |w|>0$, we find
$\nu(\{av<u\})=0$.
Therefore
\[\mu_{av}(\{av<u\})=(\Lambda |av|^n\nu)(\{av<u\})=0.\]
By the domination principle in Lemma \ref{domi_prin}, we have $av\geq u$, as claimed.

Now, let $c>0$ be the supremum of all $a>0$ such that $\{u<av\}\neq \emptyset$. By the above argument, we have $u\leq av$ for all such $a$. Hence $u\leq cv$ in $\Omega$. Evaluating at the minimum of $v$ shows that $c$ is finite. From the definition of $c$, we have $u\geq cv$. Therefore $u= cv$.
From this we have $\lambda=\Lambda$ and $v\in\E(\Omega)$. This completes the proof of the theorem.
\end{proof}

\subsection{Examples of nonuniqueness and solvability outside energy class}
In view of Proposition \ref{unmuE}, if the Monge--Amp\`ere eigenvalue problem 
with measure $\nu$ has a  solution $(\lambda, u)\in (0,\infty)\times C(\overline{\Omega})\setminus\{0\}$, then 
a Poincar\'e inequality must necessarily hold in $\E(\Omega)$ and hence $\E(\Omega)\subset L^{n+1}(\Omega, d\nu)$. Here, we give examples of Borel measures $\nu$ of the form $\nu= |v|^{-n}\mu_v$ where $v\not\in \E(\Omega)$ to show that $\E(\Omega)\subset L^{n+1}(\Omega, d\nu)$ is not sufficient for the solvability 
in the energy class of 
the Monge--Amp\`ere eigenvalue problem. Moreover, 
in one dimension, it has infinitely many families of eigenvalues and convex eigenfunctions that break symmetry.
 \begin{exa}[Raleigh quotient and insolvability in energy class]
 \label{exaRn}
 For $n\geq 1$, let
 \[\Omega=B_1(0)\subset\R^n,\quad\alpha_n:= \frac{n}{n+1},\quad v(x) =-(1-|x|^2)^{\alpha_n}, \quad \nu =\frac{\mu_v}{|v|^n} = \frac{\det D^2 v}{|v|^n}\, d\L^n.\]
 \begin{enumerate}
\item[(i)] We show that
 \[ \lambda[\Omega, \nu]:=\inf\bigg\{\frac{\int_\Omega |u|\mu_u}{\int_\Omega |u|^{n+1}\, d\nu}: u\in \E(\Omega)\bigg\}=1.\]
  Since  $(1, v)\in (0,\infty)\times C(\overline{\Omega})\setminus\{0\}$ solves the Monge--Amp\`ere eigenvalue problem
  \begin{equation}
  \label{EVPexa1}
   \mu_u=\lambda |u|^n\nu \quad\text{in} ~\Omega, \quad
u =0\quad\text{on}~\p \Omega.
\end{equation}
  Proposition \ref{unmuE} gives
 $\lambda[\Omega, \nu]\geq 1$. For the other direction,
 consider \[u_\e= -\e^{\frac{1}{n+1}}(1-|x|^2)^{\alpha_\e}\in\E(\Omega),\quad \alpha_\e=\alpha_n+\e,  \quad\text{where }0<\e<1/(n+1).\] 
For $\alpha\in (0, 1)$ and \[w(x)=-(1-|x|^2)^\alpha\equiv W(|x|) \quad \text{where }r=|x|,\quad W(r) =-(1-r^2)^{\alpha},\] we compute
 \[\det D^2 w(x)= W''(r)[W'(r)/r]^{n-1}= (2\alpha)^n (1-r^2)^{n(\alpha-1)-1} [1+ r^2(1-2\alpha)].\]
 Using polar coordinates, we have
  \begin{equation*}
 \begin{split}R_\nu(u_\e) =\frac{\int_\Omega |u_\e|\mu_{u_\e}}{\int_\Omega |u_\e|^{n+1}\, d\nu} &=\frac{\int_\Omega |u_\e|\det D^2 u_\e\, dx}{\int_\Omega |u_\e|^{n+1}\det D^2v/|v|^n\, dx} \\
 &=\Big(\frac{\alpha_n+\e}{\alpha_n}\Big)^{n}\frac{\int_0^1 r^{n-1} (1-r^2)^{(n+1)(\alpha_\e-1)} [1+ r^2(1-2\alpha_\e)]\, dr}{\int_0^1 r^{n-1} (1-r^2)^{(n+1)(\alpha_\e-1)} [1+ r^2(1-2\alpha_n)]\, dr}\\
 &\leq \Big(\frac{\alpha_n+\e}{\alpha_n}\Big)^{n}\to 1\quad \text{when } \e\to 0^+.
  \end{split}
 \end{equation*}
 Since $(n+1)(\alpha_\e-1) =(n+1)\e-1>-1$, the claim that $u_\e\in \E(\Omega)$ follows from
 \[E(u_\e) =n|B_1(0)| (2\alpha_\e)^n \int_0^1 r^{n-1} (1-r^2)^{(n+1)(\alpha_\e-1)} [1+ r^2(1-2\alpha_\e)]\, dr<\infty.\]
  It follows that
 $\lambda[\Omega, \nu] \leq \liminf_{\e\to 0^+} R_\nu(u_\e)\leq 1$.
  Hence
  $\lambda[\Omega, \nu]=1$.
  \item[(ii)] We now show that the Monge--Amp\`ere eigenvalue problem \eqref{EVPexa1} 
  has no nonzero solutions in the energy class $\E(\Omega)$. Indeed, if there were a solution $(\Lambda, u)$ where $u\in \E(\Omega)$, then, by Theorem \ref{varchac}, $\Lambda=\lambda[\Omega,\nu]=1$.
Applying Theorem \ref{MAfu} to $u$ and $v$ tells us that $v\in \E(\Omega)$. However,  because $(n+1)(\alpha_n-1)=-1$, we find  $v\not \in \E(\Omega)$ from
\[E(v) =n|B_1(0)| (2\alpha_n)^n \int_0^1 r^{n-1} (1-r^2)^{(n+1)(\alpha_n-1)} [1+ r^2(1-2\alpha_n)]\, dr=\infty.\]
 \end{enumerate}
 \end{exa}
 
 Note that the function $v$ in Example \ref{exaRn} barely fails to be in the energy class $\E(\Omega)$.
Next, we specialize Example \ref{exaRn} to one dimension. 
\begin{exa}[Nonuniqueness and symmetry breaking]
\label{exaH1}
Let \[\Omega=(-1, 1)\subset\R, \quad v(x) =-(1-x^2)^{1/2},\quad
\nu= \frac{\mu_v}{|v|}= \frac{1}{(1-x^2)^2}\, d\L^1.\]
\begin{enumerate}
\item[(i)] From Example \ref{exaRn}, we have 
 \[\lambda[(-1,1),\nu]:=\inf\bigg\{\frac{\int_{-1}^1 (-u)\, d\mu_u}{\int_{-1}^1 |u|^{2}/(1-x^2)^2\, dx}: u\in \E(-1, 1)\bigg\}=1,\]
 and the Monge--Amp\`ere eigenvalue problem with measure $\nu$ on $(-1, 1)$
  has no nonzero solutions in the energy class $\E(\Omega)$.
 \item[(ii)]  For each $\alpha \in (-1/2, 1/2)$ and 
$\lambda_\alpha= 1-4\alpha^2$, let
 \[v_\alpha(x):= -(1-x^2)^{1/2} \Big(\frac{1+x}{1-x}\Big)^\alpha.\]
 Then
 \[v''_\alpha=\lambda_\alpha |v_\alpha|\nu\quad\quad\text{in } \quad(-1, 1),\quad v_\alpha (\pm 1)=0.\]
 Hence, we have infinitely many families of eigenvalues $\lambda_{|\alpha|}$. Within each family $\lambda_{|\alpha|}$ with $\alpha\neq 0$, we have symmetry breaking and nonuniqueness of convex eigenfunctions.  
 \item[(iii)] Consider $u_\e= -\sqrt{\e}(1-|x|^2)^{1/2 +\e}\in\E(-1,1)$ where $0<\e<1/2$. Integrating by parts, we find that
 the Monge--Amp\`ere energy of $u_\e$ is
  \begin{equation*}
 \begin{split}E(u_\e)=\int_{-1}^1 (-u_\e) u_\e''\, dx
&=  2\e(1+2\e)^2 \int_{0}^1  x^2 (1-x^2)^{2\e-1}\, dx\\&\leq   2\e(1+2\e)^2 \int_{0}^1  (1-x)^{2\e-1}\, dx=(1+2\e)^2\leq 4.
  \end{split}
 \end{equation*}
 However, the vanishing mass condition \eqref{vamass} fails. Indeed, we estimate
   \begin{equation*}
 \begin{split}
 \int_{\{x\in\Omega: \dist(x,\p\Omega)\leq 1/m\}} u_\e^2 d\nu&= 2\int_{1-1/m}^1 \e (1-|x|^2)^{2\e-1}\, dx \\&\geq  \int_{1-1/m}^1 \e (1-x)^{2\e-1}\, dx
 = \frac{m^{-2\e}}{2}\to \frac{1}{2}\quad \text{when }\e=\frac{1}{m}\to 0.
  \end{split}
 \end{equation*}
 \end{enumerate}
\end{exa}

\medskip
We now complete the proofs of Theorems \ref{Dist0pn} and \ref{EVP1}.
\begin{proof}[Proof of Theorem \ref{Dist0pn}]  Part (i) is proved in Theorem \ref{Dirdistp}. Part (ii) is proved in Theorem \ref{Dirdist}. Part (iii) is proved in Theorem \ref{MApnu}. Part (iv) follows from combining Theorem  \ref{Dirdistp} (for the existence of the Monge--Amp\`ere eigenvalue and eigenfunction and its variational characterization), Lemma \ref{MAlu} (for uniqueness of the eigenvalue), Theorem \ref{varchac} and Corollary \ref{BPcor} (for variational and  spectral characterizations of the eigenvalue), and Theorem \ref{MAfu} (for uniqueness of the eigenfunction).
\end{proof}
\begin{proof}[Proof of Theorem \ref{EVP1}] 
Part (i) is proved in Theorem  \ref{monolam}. Part (ii) is proved in Proposition \ref{unmuE}. Part (iii) is proved in Proposition \ref{allmono}. Part (iv) follows from Theorems \ref{varchac} and \ref{MAfu} for the insolvability statement, and Example \ref{exaH1} for the nonuniqueness statement.
\end{proof}

Some very natural questions arise from out analysis. We list below two of them.
\noindent
\begin{question} Let $\nu$ be a Borel measure on a bounded convex domain $\Omega\subset\R^n$.
\begin{enumerate} 
\item[(i)] Assume that there is a solution $(\lambda, u)\in (0,\infty)\times\E(\Omega)\setminus\{0\}$ to the Monge--Amp\`ere eigenvalue problem 
\begin{equation*}
   \mu_u=\lambda |u|^n\nu \quad\text{in} ~\Omega\subset\R^n, \quad
u =0\quad\text{on}~\p \Omega.
\end{equation*}
Is it true that it has no globally continuous, convex solution outside the energy class?
\item[(ii)] If we remove the compactly supported condition (ii) of the measure $\nu$ in Theorem \ref{compnu}, does the comparison principle there still hold?
\end{enumerate}
\end{question}

\medskip
\noindent
{\bf Acknowledgements.} The author  would like to thank Professors Norm Levenberg and Chinh H. Lu for helpful discussions.

\end{document}